\documentclass[11pt]{amsart}
\usepackage{amscd}     
\usepackage{amssymb}
\usepackage{amsmath, amsthm, graphics, dsfont}
\usepackage{xypic}     
\LaTeXdiagrams        
\usepackage[all]{xy}
\xyoption{2cell} \UseAllTwocells \xyoption{frame} \CompileMatrices
\allowdisplaybreaks[3]
\usepackage{amsfonts}
\usepackage[normalem]{ulem}
\usepackage{hyperref}
\usepackage{mathtools}
\usepackage{verbatim} 
\usepackage{MnSymbol}
\usepackage{latexsym}
\usepackage{epsfig}
\usepackage{enumerate}
\usepackage{times}
\usepackage{xcolor}
\usepackage{geometry}
\usepackage{tikz-cd}
\newgeometry{vmargin={25mm,25mm}, hmargin={25mm,25mm}}   

\usetikzlibrary{calc}

\DeclareMathAlphabet{\mathpzc}{T1}{pzc}{m}{it}

\DeclarePairedDelimiterX{\infdivx}[2]{(}{)}{%
  #1\;\delimsize\|\;#2%
}

\newtheorem{theorem}{Theorem}[section]
\newtheorem{question}[theorem]{Question}
\newtheorem{proposition}[theorem]{Proposition}
\newtheorem{lemma}[theorem]{Lemma}
\newtheorem{corollary}[theorem]{Corollary}
\theoremstyle{definition}

\newtheorem{alephtheorem}{Theorem}

\theoremstyle{remark}
\newtheorem{remark}[theorem]{Remark}

\newcommand{\CC}{\mathbb{C}}
\newcommand{\ubullet}{\,\underline{\bullet}\,}
\newcommand{\QQ}{\mathbb{Q}}
\DeclareMathOperator{\Aut}{Aut}
\DeclareMathOperator{\ev}{ev}
\DeclareMathOperator{\diag}{diag}
\DeclareMathOperator{\End}{End}
\DeclareMathOperator{\id}{Id}
\DeclareMathOperator{\Spec}{Spec}
\DeclareMathOperator{\Tr}{Tr}
\newcommand{\strata}{\mathcal S}
\newcommand{\cI}{\mathcal I}

\newcommand{\Disc}{\Delta}
\newcommand{\rank}{\mathrm{rank}}
\newcommand{\cM}{\mathcal M}
\newcommand{\ocM}{\overline{\mathcal M}}

\newcommand{\Product}{\mathrm{Prod}}

\newcommand{\spin}{{3-\text{spin}}}

\newcommand{\Ch}{\mathsf{C}}
\newcommand{\Th}{\mathsf{\Theta}}
\newcommand{\Om}{\mathsf{\Omega}}

\newcommand{\vv}{v}
\newcommand{\ww}{w}

\title[From $3$-spin to Negative $r$-spin: Limits of CohFTs and Tautological Relations]{From $3$-spin to Negative $r$-spin: Limits of Cohomological Field Theories and Tautological Relations}
\date{\today}
\author[Genlik]{Deniz Genlik}
\author[Janda]{Felix Janda}

\begin{document}

\begin{abstract}
  We prove that the negative $r$-spin tautological relations of  Chidambaram--Garcia-Failde--Giacchetto are a consequence of Pixton's $3$-spin relations, for every $r\ge2$, and deduce that they hold in the tautological Chow ring of $\ocM_{g,n}$. In particular, our result for $r=2$ implies that Kazarian--Norbury $K$-relations hold in Chow.

  The proof factors through two essential parts. The first part shows that the tautological relations arising from Chiodo's class imply the negative $r$-spin relations. This requires an intricate study of limiting behavior of the Givental--Teleman graph contributions. In the second part, we show that the Chiodo relations are a consequence of Pixton's $3$-spin relations as an application of a generalization of Janda's work on tautological relations from CohFTs.
\end{abstract}

\maketitle

\section{Introduction}

The tautological subring $R^*(\ocM_{g,n})\subset A^*(\ocM_{g,n})$, together with its image $RH^*(\ocM_{g,n})\subset H^*(\ocM_{g,n})$ in cohomology, has been a central object of enumerative algebraic geometry since Mumford's seminal paper \cite{Mu83}. An important cornerstone in the study of tautological rings of $\ocM_{g,n}$ is the work \cite{Pi12P}, where Pixton introduced a system of classes $\mathsf{P}_{g,n,d,\sigma, a_1,\ldots,a_n}$ in $R^*(\ocM_{g,n})$ and conjectured the following tautological relations:
\begin{equation}\label{eqn:Pixtons_relations}
\mathsf{P}_{g,n,d,\sigma, a_1,\ldots,a_n}=0 \in R^*(\ocM_{g,n})\quad \text{for} \quad 3d \geq g+1 + \vert \sigma \vert+\sum_{i=1}^n a_i.
\end{equation}

In \cite{PPZ15}, Pandharipande--Pixton--Zvonkine proved these relations in the tautological cohomology ring $RH^*(\ocM_{g,n})$. Shifting Witten's $3$-spin CohFT to a semisimple point of the associated Frobenius manifold and applying the Givental--Teleman classification \cite{Te12} produces an explicit formula for the CohFT in terms of its topological field theory and $R$-matrix; the degree vanishing of the $3$-spin class above its geometric support then forces these contributions to vanish, yielding tautological relations that are equivalent to the system of relations given by \eqref{eqn:Pixtons_relations}. These relations are now known as \emph{Pixton's relations} or \emph{$3$-spin relations}. In \cite{PPZ19}, Pandharipande--Pixton--Zvonkine obtained tautological relations from  Witten's $r$-spin CohFT extending their strategy to all $r\geq 3$ .

As the proof of Givental--Teleman classification \cite{Te12} uses topological techniques, the proofs in \cite{PPZ15} and \cite{PPZ19} do not imply that the relations obtained hold in the tautological Chow ring $R^*(\ocM_{g,n})$. In \cite{Ja17}, Janda obtained Pixton's relations via the equivariant Gromov-Witten theory of $\mathbb{P}^1$ and concluded that Pixton's relations hold in the tautological Chow ring $R^*(\ocM_{g,n})$ since the proof relies on the virtual localization theorem of \cite{GrPa99}. Moreover, Janda \cite{Ja18} proved that for any convergent, generically semisimple CohFT with a flat unit and nonempty discriminant locus, the ideal of tautological relations associated with its shifted family equals the ideal of $3$-spin relations. In particular, any tautological relations obtained by applying Givental--Teleman reconstruction at a semisimple point and taking a limit toward the discriminant locus are automatically implied by Pixton's relations. In other words, the method of \cite{PPZ15} and \cite{PPZ19} cannot produce genuinely new relations within this class of shifted convergent CohFTs. In \cite{paper1}, we extend Janda's comparison theorem \cite{Ja18} from CohFTs with flat unit to generically semisimple CohFTs with a vacuum, allowing the unit to be non-flat.

Norbury defined the theta CohFT $\Th^{2}$ in \cite{Nor23} and conjectured that its descendant potential is the Br\'ezin--Gross-Witten tau function. The CohFT $\Th^2$ has found several applications: it appears in Jackiw--Teitelboim
supergravity \cite{SW20} and in volume formulas for moduli of super hyperbolic
Riemann surfaces with geodesic boundary \cite{Nor26}. Moreover, it has been coupled to the
Gromov--Witten theory of $\mathbb{P}^1$ \cite{Nor22}, and its relation with the  Br\'ezin--Gross-Witten tau function motivated the $K$- and $J$- $\kappa$-class-polynomials and their vanishing conjectures developed by Kazarian--Norbury \cite{KaNo24}.

Chidambaram--Garcia-Failde--Giacchetto generalized this construction to obtain rank-$(r-1)$ CohFTs 
$\Th^{r}$ for $r\geq 3$ in \cite{CGG25}. They constructed $\Th^r$ geometrically from the top Chern class of a vector bundle associated with $r^{\text{th}}$ roots of the log anticanonical bundle and proved that the resulting theory is a nowhere-semisimple CohFT without a flat unit. They established a topological recursion description and $\mathcal{W}$-constraints for the descendant theory, proved Norbury's Br\'ezin--Gross-Witten conjecture in the case $r=2$, conjectured the corresponding $r$-BGW statement for $r\geq 3$, proved it for $r=3$, and provided the negative $r$-spin tautological relations in the tautological cohomology ring $RH^*(\ocM_{g,n})$. More recently, Guo--Zhang developed Virasoro constraints for homogeneous CohFTs with vacuum and applied their formalism to the deformed negative $r$-spin theory \cite{GZ25}.

The construction of $\Th^{r}$ is based on the top degree part of the Chiodo's Chern class construction \cite{Ch08}. The CohFTs $\Th^{r}$ are now known as the \emph{theta CohFTs} or the \emph{negative $r$-spin CohFTs} as they are negative spin analogues of Witten's $r$-spin CohFTs. The underlying Frobenius manifold of $\Th^{r}$ is nowhere semisimple. As a result, the shifting strategy of \cite{PPZ15} to produce relations is not applicable in this setting. In \cite{CGG25}, the authors overcome this technical barrier by introducing a one-parameter deformation $\Th^{r,\epsilon}$ which is a semisimple CohFT with a non-flat unit when $\epsilon\neq 0$. Applying Givental--Teleman reconstruction to $\Th^{r,\epsilon}$ at a semisimple point expresses the full CohFT in terms of its topological field theory, $R$-matrix and vacuum vector; degree vanishing of $\Th^r$ above its geometric support produces tautological relations called the \emph{negative $r$-spin relations}. The authors pose two natural questions for these relations \cite{CGG25}. As negative $r$-spin relations obtained from $\Th^{r,\epsilon}$ by Givental--Teleman classification, the first question is:
\begin{question}\label{ques:Chow_negative_r_spin}
Do negative $r$-spin relations hold in the tautological Chow ring ?
\end{question}
Janda's theorem on tautological relations \cite{Ja18} does not apply either, for two independent reasons
emphasized in \cite{CGG25}:
\begin{itemize}
    \item The CohFT $\Th^{r,\epsilon}$ is not a shift of the theta CohFT $\Th^r$ to a semisimple point, it is rather a deformation.
    \item The unit of the deformed theta CohFT $\Th^{r,\epsilon}$ is non-flat.
\end{itemize}
Hence, another main question in \cite{CGG25} is:
\begin{question}\label{ques:3_spin_to_negative_r_spin}
Are the negative $r$-spin relations also a consequence of Pixton's $3$-spin relations?
\end{question}

\begin{remark}
  The authors of \cite{CGG25} state that they expect positive answers to these questions. However, they do not have a proof. 
\end{remark}

We should remark that a positive answer to Question \ref{ques:3_spin_to_negative_r_spin} yields a positive answer to Question \ref{ques:Chow_negative_r_spin} as the $3$-spin relations hold in the tautological Chow ring $R^*(\ocM_{g,n})$ \cite{Ja17}. The main aim of our paper is to prove that $3$-spin relations imply negative $r$-spin relations and conclude that the negative $r$-spin relations hold in the tautological Chow ring $R^*(\ocM_{g,n})$. While achieving this goal, we also obtain fundamental structural results regarding the Chiodo and theta CohFTs.

\subsection{Main results}
The main result of this paper is the following.

\begin{alephtheorem}[{see Theorem~\ref{thm:main}}] \label{ThmA}
For every $r \geq 2$, the ideal $\cI_{g,n}^{\Th^{r,\epsilon}}$ of negative $r$-spin relations is contained in the ideal $\cI_{g,n}^{\Om^{\spin}}$ of $3$-spin relations:
  \begin{equation*}
    \cI_{g,n}^{\Th^{r,\epsilon}}  \subseteq  \cI_{g,n}^{\Om^{\spin}}.
  \end{equation*}
\end{alephtheorem}
As an immediate
corollary, we obtain a Chow refinement of a recent result of \cite{CGG25}.
\begin{alephtheorem} \label{ThmB}
For $r\geq 2$, the negative $r$-spin relations hold in Chow.
\end{alephtheorem}

\begin{proof}
  As proven in \cite{Ja17}, the $3$-spin relations are tautological relations in Chow.
  Since $\cI_{g,n}^{\Th^{r,\epsilon}}  \subseteq  \cI_{g,n}^{\Om^{\spin}}$ by  Theorem~\ref{ThmA}, the negative $r$-spin relations must also hold in Chow.
\end{proof}

The case $r=2$ has a concrete consequence for the polynomial relations among $\kappa$ classes conjectured by Kazarian--Norbury \cite{KaNo24}. Let $s_m\in \QQ$ be defined by
\begin{equation*}
\exp \left( -\sum_{m=1}^{\infty}s_m z^m\right)=\sum_{k=0}^{\infty} (-1)^k (2k+1)!! z^k
\end{equation*}
and set
\begin{equation*}
K=\exp \left( \sum_{m=1}^{\infty}s_m \kappa_m\right)=\sum_{d=0}^{\infty}K_d.
\end{equation*}
Chidambaram--Garcia-Failde--Giacchetto identified the vanishing of the classes $K_d$
conjectured in part (1) of \cite[Conjecture 1]{KaNo24} with their $r=2$ negative spin
relations. Hence, they proved Kazarian--Norbury's $K$-relations in  cohomology \cite[Corollary C]{CGG25}. Moreover, in the same corollary, they identified the top degree class $K_{2g-2+n}$ with Norbury's theta class $\Th^2_{g,n}$, proving \cite[Conjecture 4]{KaNo24} in cohomology. Our Theorem~\ref{ThmB} therefore yields the following
Chow refinement of their results for the classes $K_d$.

\begin{alephtheorem} \label{ThmC}
  Let $g,n$ satisfy $2g-2+n>0$ with $n>0$. Then,
  \begin{equation*}
    K_d = 0 \in R^{d}(\ocM_{g,n})  \quad\text{for all } d>2g-2+n.
  \end{equation*}
That is, the Kazarian--Norbury $K$-relations hold in the tautological Chow ring $R^*(\ocM_{g,n})$.
\end{alephtheorem}

\subsection{Proof strategy and additional results}
Next, we outline the proof strategy for Theorem~\ref{ThmA}.
At the same time, we state other main theorems and important structural results. The proof of Theorem~\ref{ThmA} is obtained by composing the following three inclusions:
\begin{equation}\label{eqn:inclusion_breakdown_of_the_proof}
  \cI^{\Th^{r,\epsilon}}
  \;\overset{\text{Theorem~\ref{ThmE}}}{\subseteq}\;
  \cI^{\Ch^{r,\epsilon},\,t\neq0}
  \;\subseteq\;
  \cI^{\Ch^{r,\vec\epsilon},\,t}
  \;\overset{\text{Theorem~\ref{ThmF}}}{=}\;
  \cI^{\Om^{\spin}}.
\end{equation}
First, the inclusion on the left is Theorem~\ref{ThmE} and is proved by the technical input of our paper, the graphwise limit of Theorem~\ref{ThmD}. The middle inclusion follows from the observation that specializing the shift parameters can only lose relations. Lastly, the equality on the right is Theorem F and it is obtained by applying our results of \cite{paper1}. We note that only the inclusion $\subseteq$ is needed in the last step. The equality of Theorem~\ref{ThmF} is a strictly stronger structural result and it is of independent interest. Theorem~\ref{ThmF} is proved using the main result of \cite{paper1}.

\subsubsection{Givental--Teleman reconstruction and graph limits}
First, we show that the $\epsilon v_0$-shifted Chiodo CohFT $\Ch^{r,\epsilon}$ is generically semisimple with a non-flat unit over the vector space ${V}=\left\langle v_0,\ldots,v_{r-1}\right\rangle_{\QQ}$. Here $t$ records the Chern grading in the normalization of the Chiodo
class, so that specialization at $t=0$ retains the normalized
top-Chern contribution. The parameter $\epsilon$, on the other hand, records the additional
$v_0$-insertions due to CohFT shift. For precise definition of $\Ch^{r,\epsilon}$, see Section~\ref{sec:chiodo-theta-background}. 

The $\epsilon v_0$-deformed theta CohFT $\Th^{r,\epsilon}$ may be defined as the limit of restriction of $\Ch^{r,\epsilon}$ to the subspace $\underline{V}=\left\langle v_1,\ldots,v_{r-1}\right\rangle_{\QQ}\subset V$:
\begin{equation*}
\Th^{r,\epsilon} \coloneqq \lim_{t\to 0}\Ch^{r,\epsilon}\vert_{\underline{V}}.
\end{equation*}
Both theories are generically semisimple, but over different base rings. The CohFT $\Ch^{r,\epsilon}$ is defined over $\QQ[\epsilon,t^{\pm 1}]$ and is semisimple exactly where the discriminant function $(-1)^{r-1} r^r t^{r - 1}-(r - 1)^{r - 1} \epsilon^r$ is invertible, while $\Th^{r,\epsilon}$ is defined over $\QQ[\epsilon]$ and is semisimple exactly for $\epsilon\neq0$. Hence, they can be reconstructed from their topological field theories by Givental--Teleman classification away from their non-semisimple loci. Yet, it is not clear a priori if the  reconstruction ``commutes'' with this limiting procedure. The first obstacle we face is that the unit of $\Ch^{r,\epsilon}$ is
\begin{equation*}
  \mathbf{1} = \frac{-v_{r - 1} + \epsilon v_0}t.
\end{equation*}
and it does not admit a limit as $t\to 0$. Moreover, it substantially differs from the unit of $\Th^{r,\epsilon}$:
\begin{equation*}
  \underline{\mathbf{1}}
  =
  \begin{cases}
    -\frac{v_1}{\epsilon^2}\quad &\text{if}\quad r=2,\\
    -\frac{v_{r - 2}}\epsilon \quad &\text{if}\quad r\geq 3.
  \end{cases}
\end{equation*}
Similarly, the vacuum vector $\mathbf{v}(z)$ of $\Ch^{r,\epsilon}$ is a Laurent series in $t$:
\begin{equation*}
\mathbf{v}(z)
=
\frac{\epsilon v_0-v_{r-1}}{t-z}
=
(\epsilon v_0-v_{r-1})
\sum_{k=0}^{\infty} z^k t^{-k-1}.
\end{equation*}
In particular, for $k\geq 0$, the coefficient of $z^k$ in $\mathbf{v}(z)$ is
\begin{equation*}
t^{-k-1}(\epsilon v_0-v_{r-1}).
\end{equation*}
So, the $z$-expansion of the vacuum vector  $\mathbf{v}(z)$ has no limit as $t \to 0$. Another difficulty is that the flatness differential equation for the $R$-matrix of $\Ch^{r,\epsilon}$ is quite complex to solve in closed form. Moreover, the $R$-matrix is
determined by the flatness equation only up to integration constants whose orders of $t$ are potentially negative. Thus, it is not the case that the ingredients in the reconstruction of $\Th^{r,\epsilon}$ are obtained by taking limits of those of $\Ch^{r,\epsilon}$ individually. The correct limiting statement must therefore be formulated after all vertex, leg, and edge factors associated with a fixed stable graph have been combined.

Our first main technical result is that the limits of individual graph contributions for shifted Chiodo CohFT $\Ch^{r,\epsilon}$ nevertheless exist, and they are equal to the corresponding graph contribution of deformed theta CohFT $\Th^{r,\epsilon}$.
\begin{alephtheorem} \label{ThmD}
Let $2g-2+n>0$ with $n>0$, and let $s_1,\dotsc,s_n\in\{1,\dotsc,r-1\}$. Then, for every stable graph $\Gamma$ of genus $g$ with $n$ legs,
  \begin{equation*}
    \lim_{t\to 0}\mathrm{Cont}^{\Ch^{r,\epsilon}}_{\Gamma}
      (v_{s_1}\otimes\cdots\otimes v_{s_n})
    = \mathrm{Cont}^{\Th^{r,\epsilon}}_{\Gamma}
      (v_{s_1}\otimes\cdots\otimes v_{s_n}),
  \end{equation*}
where $\mathrm{Cont}_{\Gamma}$ denotes the contribution of $\Gamma$ to the Givental--Teleman reconstruction.
\end{alephtheorem}

Theorem~\ref{ThmD} is the technical heart of the paper.  The roots of characteristic polynomial $p_t(x)$ of $\Ch^{r,\epsilon}$ determine the  canonical coordinates of $\Ch^{r,\epsilon}$. Away from the discriminant locus, these roots split as $t \to 0$. One of the roots tends to zero and the remaining $r-1$ roots converge to the canonical roots of $\Th^{r,\epsilon}$.  We first determine the resulting $t$-orders of the transition matrices and of the topological field theory of $\Ch^{r,\epsilon}$.  We then use the flatness equation for its $R$-matrix, together with Harer stability on the smooth locus, to control the integration constants and the possible polar terms in the $R$-matrix and $T$-vector.  Individual factors may still have negative or fractional $t$-orders.  The key step is to show that, for each fixed stable graph, these singular orders cancel after all factors are combined, and that the constant term is exactly the corresponding $\Th^{r,\epsilon}$ graph contribution.  Thus, the reconstruction commutes with $t\to0$ graph by graph, even though the reconstruction data does not converge individually.

\subsubsection{From Chiodo relations to theta relations}

Theorem~\ref{ThmD} shows that the $t \to 0$ limit exists for every stable graph contribution to $\Ch^{r,\epsilon}$.
This allows us to prove that the tautological relations obtained from the shifted Chiodo CohFT $\Ch^{r,\epsilon}$ imply the negative $r$-spin relations.
\begin{alephtheorem}[{see Proposition~\ref{prop:chiodo-to-theta-rels}}] \label{ThmE}
Let $r\ge2$. Let $\cI^{\Ch^{r,\epsilon},t\neq0}$ denote the ideal of tautological relations of $\Ch^{r,\epsilon}$, viewed as a CohFT over $\QQ[\epsilon,t^{\pm 1}]$, and let $\cI^{\Th^{r,\epsilon}}$ denote the ideal of negative $r$-spin relations. Then,
\begin{equation*}
    \cI^{\Th^{r,\epsilon}} \subseteq \cI^{\Ch^{r,\epsilon},t\neq0}.
\end{equation*}
\end{alephtheorem}

\subsubsection{From $3$-spin relations to Chiodo relations}
Now, it remains to prove the other two comparisons in the chain \eqref{eqn:inclusion_breakdown_of_the_proof}. These will be proved via the techniques and main results of our paper \cite{paper1}.

To apply the result, we have to fix a value of $t \neq 0$, and allow shifting by a general element $\Vec\epsilon=\sum_{i=0}^{r-1}\epsilon_i v_i$ of the state space $V$.
Call the shifted Chiodo CohFT $\Ch^{r,\Vec\epsilon}$; the one-parameter family
$\Ch^{r,\epsilon}$ of Theorem~\ref{ThmE} is the specialization $\epsilon_0=\epsilon$,
$\epsilon_1=\cdots=\epsilon_{r-1}=0$.
By Lemma~\ref{lem:general_shift_polynomiality_lemma} the classes of $\Ch^{r,\Vec\epsilon}$
depend on $\epsilon_{r-1}$ only through nonnegative powers of
$(1-\epsilon_{r-1})^{-1}$, so we take as base the polynomial ring
\begin{equation*}
  \QQ[\epsilon_0,\dotsc,\epsilon_{r-2},(1-\epsilon_{r-1})^{-1}].
\end{equation*}
\begin{alephtheorem}[{see Corollary~\ref{cor:3spin-to-chiodo-rels}}] \label{ThmF}
  Fix an arbitrary nonzero value for $t$ and $r \ge 2$.
  View $\Ch^{r, \Vec\epsilon}$ as a CohFT with parameters in  $\QQ[\epsilon_0, \dotsc, \epsilon_{r - 2},(1-\epsilon_{r-1})^{-1}]$, and let $\cI_{g,n}^{\Ch^{r, \Vec\epsilon}, t}$ be the corresponding ideal of tautological relations.
  We then have
  \begin{equation*}
    \cI_{g,n}^{\Ch^{r, \Vec\epsilon}, t} = \cI_{g,n}^{\Om^\spin}.
  \end{equation*}
\end{alephtheorem}

\subsection{Some remarks}

\subsubsection{General $(r,s)$-theta CohFTs and taking limits }
The negative $r$-spin construction admits variants for $-r+1\leq s\leq-1$.  As observed in \cite[Remark 2.10]{CGG25}, the same construction gives generalized classes $\Th^{r,s}$ and deformations $\Th^{r,s,\epsilon}$. With the notation of \cite[Remark 2.10]{CGG25}, the CohFTs $\Th^{r}$ and $\Th^{r,\epsilon}$ correspond to $\Th^{r,-1}$ and  $\Th^{r,-1,\epsilon}$ respectively. The explicit Givental--Teleman reconstruction and the tautological relations obtained in \cite{CGG25} are developed for $s=-1$. Our results likewise concern only $s=-1$.

There is now a broader picture for the generalized classes $\Th^{r,s}$.  In \cite{BCGS25}, the topological recursion data associated with the classes $\Th^{r,s}$ are further analyzed\footnote{The positive parameter denoted $s$ in \cite{BCGS25}
corresponds to $r+s$ in our conventions.} in the context of taking limits in topological recursion to show that the descendant invariants of $\Th^{r,s}$ can be computed from $(r,s)$-spectral curve. In \cite{BBCKS25}, Borot--Bouchard--Chidambaram--Kramer--Shadrin develop a general criterion to explain when taking limits commute with the topological recursion data. To the best of our knowledge, our Theorem~\ref{ThmD} can be considered as the first instance of analogous version for the limits commuting with the Givental--Teleman reconstruction at the level of stable graphs, despite the singular behavior of the reconstruction data. Moreover, our case concerns $\epsilon\neq 0$ for the limit $t\to 0$, whose topological recursion counterpart has not been studied explicitly for the shifted Chiodo CohFT $\Ch^{r,\epsilon}$ where $s=-1$.

\subsubsection{Kazarian--Norbury's $J$-relations} Kazarian--Norbury also conjectured a second family of tautological relations in part (2) of \cite[Conjecture 1]{KaNo24}; namely, the $J$-relations. Alexandrov--Bychkov--Dunin-Barkowski--Kazarian--Shadrin prove these relations in the tautological Chow ring \cite[Theorem 1.1 and Theorem 1.2]{ABDKS25}.  Their proof uses localized virtual classes and a localization/materialization analysis of the spin Gromov--Witten theory of $(\mathbb{P}^1,\mathcal O(-1))$.

\subsubsection{An open question} In \cite{CGG25}, Chidambaram--Garica-Failde-Giacchetto also ask the reverse question:
\begin{question}
Do negative $r$-spin relations imply $3$-spin relations?  
\end{question}
We remark that our paper does not address this question.

\subsubsection{Notation}
An underlined symbol stands for data belonging to the theta CohFTs. For example, all of the symbols
$\underline\eta$, $\ubullet$, $\underline e_i$, $\underline\zeta_i$,
$\underline\xi_i$, $\underline\Psi$, $\underline R$,
$\underline T$, and $\underline\omega$ are data relevant to deformed theta CohFT $\Th^{r,\epsilon}$. The corresponding non-underlined symbols refer to the shifted Chiodo CohFT $\Ch^{r,\epsilon}$.

\subsection{Outline of the paper}
Section \ref{sec:background} reviews stable graphs, CohFTs, actions of CohFTs, the Givental--Teleman reconstruction of semisimple CohFTs, and the definition of the ideal of relations in the strata algebra. Section \ref{sec:chiodo-theta-background} fixes the Chiodo and theta conventions and identifies the negative $r$-spin relations with the ideal of the deformed theta CohFT.  Section \ref{sec:chiodo-analysis} computes the Frobenius algebra of $\Ch^{r,\epsilon}$ and compares its canonical data with those of $\Th^{r,\epsilon}$ as $t\to0$.  Section \ref{sec:chiodo-rt} derives the flatness equations and the vacuum and translation vectors.  In Section \ref{sec:reconstruction}, the two strata-valued reconstructions are written as a sum over stable graphs. Section \ref{sec:asymptotic-analysis} establishes the required $t$-adic estimates, and Section \ref{sec:Limit_of_Givental_Teleman_Formula} proves the graphwise limit. Section \ref{sec:main-proof} combines that limit with the comparison theorem of \cite{paper1} to prove Theorem \ref{thm:main}.  The appendices derive the pairing and quantum product for the shifted Chiodo CohFT $\Ch^{r,\epsilon}$ and prove the convergence statement for more general shifts used in the main argument.

\subsection{Acknowledgements} We thank Nitin Kumar Chidambaram, Elba Garcia-Failde, and Alessandro Giacchetto for the open questions they raised in their paper \cite{CGG25} as well as for useful discussions and correspondence. D.G.\ was partially supported by an AMS-Simons Travel Grant.
F.J.\ was partially supported by NSF grant DMS-2422291.




\section{Background on Cohomological Field Theories}
\label{sec:background}

\subsection{Stable graphs} In this part, we will recall stable graphs and certain related notions. We refer the reader to \cite{PPZ15} and \cite{JS20} for more details.

Let $g$, $n$ be non-negative integers such that $2g-2+n>0$. A \textit{stable graph} $\Gamma$ of genus $g$ with $n$ legs consists of the following data
\begin{equation*}
\left(\mathrm{V}_{\Gamma},   \mathrm{H}_{\Gamma}, \mathrm{E}_{\Gamma},\mathrm{L}_{\Gamma}, \mathrm{g} \colon  \mathrm{V}_{\Gamma} \rightarrow \mathbb{Z}_{\geq 0}, \iota \colon  \mathrm{H}_{\Gamma} \rightarrow \mathrm{H}_{\Gamma},  \nu \colon  \mathrm{H}_{\Gamma} \rightarrow \mathrm{V}_{\Gamma}, \ell \colon \mathrm{L}_{\Gamma}\rightarrow\{1,\ldots,n\}\right)
\end{equation*}
where
\begin{enumerate}
\item $\mathrm{V}_{\Gamma}$ is the set of vertices together with a genus assignment $\mathrm{g} \colon \mathrm{V}_{\Gamma}\rightarrow\mathbb{Z}_{\geq 0}$,
    
\item $\mathrm{H}_{\Gamma}$ is the set of half-edges equipped with an involution $\iota \colon  \mathrm{H}_{\Gamma} \rightarrow \mathrm{H}_{\Gamma}$ and a vertex assignment $\nu \colon  \mathrm{H}_{\Gamma} \rightarrow \mathrm{V}_{\Gamma}$,
    
\item An edge of $\Gamma$ is given by a size two orbit of $\iota \colon  \mathrm{H}_{\Gamma} \rightarrow \mathrm{H}_{\Gamma}$, and $\mathrm{E}_{\Gamma}$ is the set of edges,\footnote{Self-edges are allowed.}  and $(\mathrm{V}_{\Gamma},\mathrm{E}_{\Gamma})$ forms a connected graph,
    
\item A leg is a half-edge that is fixed by $\iota \colon  \mathrm{H}_{\Gamma} \rightarrow \mathrm{H}_{\Gamma}$, $\mathrm{L}_{\Gamma}$ is the set of legs, and the map $\ell \colon \mathrm{L}_{\Gamma}\rightarrow\{1,\ldots,n\}$ is a bijection that labels the legs,
    
\item At each vertex $\mathfrak{v}$ the following stability condition is satisfied:
\begin{equation*}
2\mathrm{g}(\mathfrak{v})-2+\mathrm{n}(\mathfrak{v})>0
\end{equation*}
where $\mathrm{n}(\mathfrak{v})=\vert \nu^{-1}(\mathfrak{v})\vert$ is the valence of the vertex $\mathfrak{v}$.
\end{enumerate}
The \textit{genus} of $\Gamma$ is defined by
\begin{equation*}
\mathrm{g}(\Gamma)=h^1(\Gamma)+\sum_{\mathfrak{v}\in\mathrm{V}_{\Gamma}}\mathrm{g}(\mathfrak{v})
\end{equation*}
where $h^1(\Gamma)=\vert \mathrm{E}_{\Gamma}\vert -\vert \mathrm{V}_{\Gamma} \vert +1$ is the first Betti number of the underlying graph and we require $\mathrm{g}(\Gamma)=g$.
An \emph{isomorphism} of stable graphs consists of a bijection of vertex sets and a bijection of half-edge sets, which preserve the genus assignments, vertex assignments, labeling maps and involutions of half-edge sets. Let $\mathsf{G}_{g,n}$ be the \emph{set of isomorphism classes of stable graphs} of genus $g$ with $n$ legs.

Given a stable graph $\Gamma$, the canonical gluing map is denoted by
\begin{equation*}
    \mathpzc{gl}_\Gamma\colon \ocM_{\Gamma} \to \ocM_{g, n} ,
\end{equation*}
where $\ocM_{\Gamma}$ is the product space associated to $\Gamma$
\begin{equation*}
\ocM_{\Gamma}= \prod_{ \mathfrak{v} \in \mathrm{V}_{\Gamma}} \ocM_{\mathrm{g}(\mathfrak{v}), \mathrm{n}(\mathfrak{v})}.
\end{equation*}
This canonical gluing map has degree $\vert \Aut(\Gamma)\vert$ where $\Aut(\Gamma)$ is the automorphism group of the stable graph $\Gamma$.

\subsection{Cohomological field theories} \label{subsec:CohFTs}
Here, we provide some background on cohomological field theories. 

Let $A$ be a commutative $\QQ$-algebra, $V$ be a finite rank free $A$-module and $$\eta \colon V \otimes V \rightarrow A$$ be a perfect, symmetric bilinear form (i.e. a \emph{pairing}). Let $r=\rank_A(V)$ be the rank of $V$. For a choice of basis $\{w_0,\ldots, w_{r-1}\}$, set
\begin{equation*}
\eta_{i,j}=\eta (w_i,w_j)
\end{equation*}
and let $\eta^{i,j}$ denote the entries of the inverse matrix of $\eta_{i,j}$. Equivalently, the inverse pairing is the bivector
\begin{equation*}
\eta^{-1}=\sum_{i,j}\eta^{i,j}w_i\otimes w_j.
\end{equation*}
A \emph{cohomological field theory} on $(V,\eta)$ is a system $\Om=\left(\Om_{g,n}\right)_{2g-2+n>0}$ of tensors
\begin{equation*}
\Om_{g,n}\in H^*(\ocM_{g,n},A)\otimes (V^*)^{\otimes n}\quad \left(\text{or equivalently multilinear maps}\quad \Om_{g,n} \colon V^{\otimes n}\rightarrow H^*\left(\ocM_{g,n},A \right)\right)
\end{equation*}
satisfying the following axioms:
\begin{itemize}
    \item \emph{Symmetry axiom:} Each $\Om_{g,n}$ is invariant under the canonical action of the symmetric group $S_n$ permuting the markings and the copies of $V$,
    
    \item \emph{Gluing axiom:} The pullbacks by the canonical gluing maps of moduli spaces of stable curves\footnote{Here we have $g=g_1+g_2$, $n=n_1+n_2$ and $2g_i-2+(n_i+1)>0$ for $i=1,2$. }
    \begin{equation*}
    \mathpzc{q}\, \colon \ocM_{g-1,n+2}\rightarrow \ocM_{g,n}\quad \text{and} \quad  \mathpzc{r}\, \colon\ocM_{g_1,n_1+1}\times \ocM_{g_2,n_2+1}\rightarrow \ocM_{g,n}
    \end{equation*}
    satisfy the following contraction rules
    \begin{equation*}
    \begin{aligned}
    \mathpzc{q}^*(\Om_{g,n}(\phi_1\otimes \cdots \otimes \phi_n))&=\sum_{i,j}\eta^{i,j}\Om_{g-1,n+2}(\phi_1 \otimes \cdots \otimes \phi_n \otimes w_i \otimes w_j),\\
    \mathpzc{r}^*(\Om_{g,n}(\phi_1\otimes \cdots \otimes \phi_n))&=\sum_{i,j} \eta^{i,j}\, \Om_{g_1,n_1+1}(\phi_1 \otimes \cdots \otimes \phi_{n_1} \otimes w_i) \otimes \Om_{g_2,n_2+1}(\phi_{n_1+1} \otimes \cdots \otimes \phi_n \otimes  w_j).
    \end{aligned}
    \end{equation*}
    for all $\phi_k\in V$.
\end{itemize}

For a cohomological field theory $\Om$ over $(V,\eta)$, if there is a distinguished element $\mathbf{1}\in V$ satisfying the following unit axiom, then $\Om$ is said to be a \emph{cohomological field theory with a flat unit} $\mathbf{1}$.
\begin{itemize}
    \item \emph{Unit axiom:} For all $\phi_k\in V$, we have
    
    \begin{equation*}
    \begin{aligned}
    \Om_{0,3}(\phi_i\otimes \phi_j\otimes \mathbf{1})&=\eta(\phi_i,\phi_j)\quad &\text{(\emph{unit property})},\\
    \Om_{g,n+1}(\phi_1\otimes \cdots \otimes \phi_n\otimes \mathbf{1})&=\mathpzc{p}^*(\Om_{g,n}(\phi_1\otimes \cdots \otimes \phi_n)) \quad &\text{(\emph{flatness property})},
     \end{aligned} 
    \end{equation*}
    where 
    \begin{equation*}
    \mathpzc{p}\,\colon \ocM_{g,n+1}\rightarrow \ocM_{g,n}
    \end{equation*}
    is the forgetful map that forgets the last marked point.
\end{itemize}

A cohomological field theory $\Om$ induces a \emph{quantum product} $\bullet$ on $V$ via the identification
\begin{equation*}
\Om_{0,3}(\phi_i\otimes \phi_j\otimes \phi_k)=\eta(\phi_i\bullet \phi_j,\phi_k)
\end{equation*}
for all $\phi_i, \phi_j, \phi_k \in V$. The quantum product $\bullet$ is commutative and associative. Moreover, if $\Om$ is a CohFT with a flat unit $\mathbf{1}$, then $\mathbf{1}$ is the identity of the quantum product.

If the quantum product of a CohFT $\Om$ on $(V,\eta)$ has an identity element $\mathbf{1}$ that fails to satisfy the flatness property (pullback property) of the unit axiom, then $\Om$ is said to be a \emph{CohFT with a non-flat unit}.

Let $\Om$ be a CohFT on $(V,\eta)$.  A \emph{vacuum vector} $\mathbf{v}(z)$ is a series
\begin{equation*}
  \mathbf{v}(z)=\sum_{k\geq0}\mathbf{v}_k z^k\in V[\![z]\!] 
\end{equation*}
such that, for all $(g,n)$ in the stable range $2g-2+n>0$ and for all $\phi_1,\ldots,\phi_n \in V$, we have
\begin{equation*}
  \mathpzc p^*\Om_{g,n}(\phi_1\otimes\cdots\otimes\phi_n)
  =
  \Om_{g,n+1}
  (\phi_1\otimes\cdots\otimes\phi_n\otimes
   \mathbf{v}(\psi_{n+1})). 
\end{equation*}
A CohFT equipped with such a series is called a \emph{CohFT with
vacuum}.

\begin{remark}\label{rem:vacuum_at_zero}
 If the quantum product $\bullet$ admits an identity element
  $\mathbf{1}$ (which is the case for every CohFT considered in this paper, and
  in particular for every generically semisimple one), then $\mathbf{v}(0)=\mathbf{1}$. Moreover, if $\Om$ has a \emph{flat} unit $\mathbf{1}$, then the constant series
  $\mathbf{v}(z)=\mathbf{1}$ is a vacuum vector.
\end{remark}

A CohFT $\omega=(\omega_{g,n})_{2g-2+n>0}$ taking values in the degree $0$ part of the cohomology, i.e. 
\begin{equation*}
\omega_{g,n}\in H^0(\ocM_{g,n},A)\otimes (V^*)^{\otimes n}
\end{equation*}
is said to be a \emph{topological field theory (TFT).} Projection $\left[ \Om\right]_0$ of a CohFT $\Om$ to its cohomology degree zero part is a topological field theory and it is called the \emph{topological part of $\Om$.}

For more details on CohFTs, we refer the reader to \cite{PaICM18} and the references therein. For more on CohFTs with vacuum and their recent applications, see \cite{GZ25}.

\subsection{Actions on CohFTs}
\label{ss:cohft-actions}

We will start with describing two actions on CohFTs. Let $\Om$ be a CohFT on $(V,\eta)$.

\subsubsection{The translation action} Let $T$ be a $V$-valued power series
\begin{equation*}
T(z)=\sum_{k=2}^{\infty} T_k z^k = T_2z^2+T_3z^3+\cdots \in z^2 V[\![z]\!]
\end{equation*}
with vanishing terms at degrees $0$ and $1$.

The \emph{translation $T\Om=((T\Om)_{g,n})_{2g-2+n>0}$ of $\Om$ by $T$} is defined by
\begin{equation*}
(T\Om)_{g,n}(\phi_1\otimes\cdots\otimes \phi_n)=\sum_{m=0}^{\infty}\frac{1}{m!}(\mathpzc{p}_{m})_{*}\Om_{g,n+m}(\phi_1\otimes\cdots \otimes \phi_n \otimes T(\psi_{n+1})\otimes \cdots \otimes T(\psi_{n+m}))
\end{equation*}
 where 
    \begin{equation*}
    \mathpzc{p}_{m}\,\colon \ocM_{g,n+m}\rightarrow \ocM_{g,n}
    \end{equation*}
is the forgetful map that forgets the last $m$ marked points and the $i^{\text{th}}$ psi-class $\psi_i=c_1(\mathcal{L}_i)\in H^2(\ocM_{g,n},\QQ)$ is the first Chern class of the cotangent line bundle $\mathcal{L}_i$ at the $i^{\text{th}}$ marked point. The translation $T\Om$ is also a CohFT. We also write \begin{equation*}
  \kappa_i \coloneqq \mathpzc{p}_{\star}\bigl(\psi_{n+1}^{i+1}\bigr)
  \in H^{2i}(\ocM_{g,n},\QQ),  \quad i\geq 0,
\end{equation*}
for the  $i^{\text{th}}$ kappa-class, where $\mathpzc{p}\,\colon\ocM_{g,n+1}\to\ocM_{g,n}$ forgets
the last marked point. In particular, we have $\kappa_0=2g-2+n$.

\subsubsection{The $R$-matrix action} Let $R$ be a $\End(V)$-valued power series
\begin{equation*}
R(z)=\sum_{k=0}^{\infty}R_kz^k=\mathrm{Id}+R_1z+R_2z^2+\cdots \in \mathrm{Id}+z\cdot \End(V)[\![z]\!]
\end{equation*}
satisfying the \emph{symplectic condition}\footnote{Here $(-)^{\star}$ is the adjoint with respect to the bilinear form $\eta$.}
\begin{equation}
  \label{eqn:symplectic-condition}
R^*(-z)R(z)=\mathrm{Id}.
\end{equation}

Another important action on CohFTs is defined through such a matrix series $R(z)$ and the action is called \emph{$R$-matrix action}. The resulting CohFT of the action is the system $R\Om=((R\Om)_{g,n})_{2g-2+n>0}$ defined as a sum over stable graphs
\begin{equation*}
(R\Om)_{g,n} \coloneqq \sum_{\Gamma\in\mathsf{G}_{g,n}}\mathrm{Cont}^{R\Om}_{\Gamma}
\end{equation*}
with
\begin{equation*}
\mathrm{Cont}^{R\Om}_{\Gamma}
\coloneqq\frac{1}{\vert\Aut(\Gamma)\vert}\,(\mathpzc{gl}_{\Gamma})_{\star}\!\left(
\prod_{\mathfrak{v}\in \mathrm{V}_{\Gamma}}\mathrm{Cont}_{\Gamma,\mathfrak{v}}
\prod_{\mathfrak{l}\in \mathrm{L}_{\Gamma}}\mathrm{Cont}_{\Gamma,\mathfrak{l}}
\prod_{\mathfrak{e}\in \mathrm{E}_{\Gamma}}\mathrm{Cont}_{\Gamma,\mathfrak{e}}
\right)
\end{equation*}
where for an isomorphism class of a stable graph $\Gamma$, the vertex contribution $\mathrm{Cont}_{\Gamma,\mathfrak{v}}$, the leg contribution  $\mathrm{Cont}_{\Gamma,\mathfrak{l}}$, the edge contribution $\mathrm{Cont}_{\Gamma,\mathfrak{e}}$ are defined by

\begin{itemize}
\item [(i)] $\mathrm{Cont}_{\Gamma,\mathfrak{v}}=\Om_{\mathrm{g}(\mathfrak{v}), \mathrm{n}(\mathfrak{v})}$,

\item [(ii)] $\mathrm{Cont}_{\Gamma,\mathfrak{l}}=R^{-1}(\psi_{\mathfrak{l}})\,$ where $\psi_{\mathfrak{l}}$ 
is the cotangent class at the marked point corresponding to the leg $\mathfrak{l}$,

\item [(iii)] and lastly\footnote{Here, $(-)^{\top}$ denotes the matrix transpose. For the precise meaning of the notation for the edge contribution, see \cite[Section 2]{PPZ15}.}
\begin{equation}\label{eqn:R_matrix_action_edge_contribution_formula}
\mathrm{Cont}_{\Gamma,\mathfrak{e}}=\frac{\eta^{-1}-R^{-1}(\psi'_{\mathfrak{e}})\eta^{-1} R^{-1}(\psi''_{\mathfrak{e}})^{\top}}{\psi'_{\mathfrak{e}}+\psi''_{\mathfrak{e}}}\, , 
\end{equation}
where $\psi'_{\mathfrak{e}}$ and $\psi''_{\mathfrak{e}}$ are the cotangent classes at the node corresponding to the edge $\mathfrak{e}$.

\end{itemize}

For a detailed treatment of translation and $R$-matrix actions on CohFTs and properties of these actions, we refer the reader to \cite[Section 2]{PPZ15}.

\subsection{The shift of a CohFT}
Let $\Om=(\Om_{g,n})_{2g-2+n>0}$ be a CohFT over $(V,\eta)$. The shift $\Om^{\gamma}\coloneqq (\Om^{\gamma}_{g,n})_{2g-2+n>0}$ of $\Om$ by $\gamma \in V$ is defined by formally setting
\begin{equation*}
\Om^{\gamma}_{g,n}(\phi_1\otimes \cdots \otimes \phi_n)\coloneqq \sum_{m\geq 0}\frac{1}{m!}\left({\mathpzc{p}_{m}}\right)_\star\Om_{g,n+m}(\phi_{1}\otimes\cdots\otimes \phi_{n}\otimes \gamma^{\otimes m}).
\end{equation*}

\begin{remark}
Although, the definition of a shift looks similar to the definition of a translation action. Unlike translations by $T(z)\in z^2V[\![z]\!]$, the shifts are not automatically finite sums or well-defined. It must therefore be interpreted formally, or shown separately to converge or truncate in the theory under consideration.
\end{remark}

\subsection{Frobenius algebras and semisimplicity}
\label{subsec:frobenius-algebras}

Let $A$ be a reduced $\QQ$-algebra.  A Frobenius algebra over $A$ is a
finite free commutative unital $A$-algebra $(V,\bullet)$ equipped
with a perfect symmetric pairing $\eta$ satisfying the \emph{Frobenius property}:
\begin{equation*}
  \eta(u\bullet v,w)=\eta(u,v\bullet w).
\end{equation*}

A basis $e_0,\ldots,e_{r-1}$ of $V$ is a \emph{basis of orthogonal idempotents} if they satisfy
\begin{equation*}
  e_i\bullet e_j=\delta_{i,j}e_i.
\end{equation*}
The Frobenius property then gives $\eta(e_i,e_j)=0$ for $i\neq j$, and the unit of the product $\bullet$ is $\sum_i e_i$. For such a basis set
\begin{equation*}
\Disc_i \coloneqq \eta(e_i,e_i)^{-1}.
\end{equation*}

We say that $(V,\bullet)$ is \emph{generically semisimple} if it admits a basis of orthogonal idempotents after a flat extension of the base.
We say that $(V,\bullet)$ is \emph{semisimple at $\mathfrak p \in \Spec(A)$} if it admits a basis of orthogonal idempotents after passing to the algebraic closure of the residue field of $A$ at $\mathfrak p$.
In particular, if $A = \QQ$, the Frobenius algebra $(V,\bullet)$ is semisimple at the unique point of $\Spec(A)$ if it admits a basis of orthogonal idempotents after extending scalars to $\CC$.

Following \cite{paper1}, given a basis $w_0, \dotsc, w_{r-1}$, the \emph{discriminant} is defined as
\begin{equation*}
  \Disc \coloneqq \frac{\det(\Tr(w_i w_j)_{ij})}{\det(\eta(w_i, w_j)_{ij})}.
\end{equation*}
By \cite
{paper1}, a Frobenius algebra is generically semisimple if and only if $\Disc$ is not a zero divisor, and in this case, we may write
\begin{equation*}
\Disc = \prod_{i=0}^{r-1}\Disc_i.
\end{equation*}
For a generically semisimple Frobenius algebra, given a choice of roots $\Disc_i^{1/2}$ (which may require a base-change),
the \emph{normalized idempotents} are
\begin{equation*}
\widetilde e_i \coloneqq \Disc_i^{1/2}e_i.
\end{equation*}
They satisfy $\eta(\widetilde e_i,\widetilde e_j)=\delta_{ij}$.

\subsection{Semisimple CohFTs and the Givental--Teleman reconstruction}

Here, we explain the Givental--Teleman reconstruction (classification) of semisimple CohFTs. A CohFT $\Om$ over $(V,\eta)$ is said to be semisimple if the underlying Frobenius algebra $(V,\bullet)$ is semisimple.

\begin{theorem}[Givental--Teleman Classification \cite{Te12}, \cite{MOPPZ17}, \cite{paper1}] \label{thm:Givental--Teleman_Classification}
Let $A$ be a reduced $\QQ$-algebra and let $\Om$ be a generically semisimple
  CohFT on $(V,\eta)$ over $A$ whose quantum product has an identity $\mathbf{1}$ that is not necessarily flat. Let $\Disc\in A$ be its discriminant, set
  $B\coloneqq A[\Disc^{-1}]$, and let $\omega=[\Om]^{0}$ be its topological part.
  Then, there exists a unique symplectic series
    \begin{equation*}
    R(z)\in\operatorname{Id}+z\End(V)\otimes_A B[\![z]\!]
    \end{equation*}
  and a unique vacuum vector $\mathbf{v}(z)\in V\otimes_A B[\![z]\!]$ such that,
  setting
    \begin{equation*}
    T(z)\coloneqq z\bigl(\mathbf{1}-R^{-1}(z)\mathbf{v}(z)\bigr)\in z^{2}V\otimes_A B[\![z]\!],
    \end{equation*}
  one has $\Om_{g,n}=(RT\omega)_{g,n}$ in cohomology for every $g,n$ with
  $2g-2+n>0$ and $n>0$. If $\mathbf{1}$ is a flat unit, the same holds for $n=0$.
\end{theorem}

\subsection{The strata algebra and tautological relations}\label{subsec:strata_algebra}
The \emph{strata algebra} is a system of $\QQ$-algebras
\begin{equation*}
  \strata = (\strata_{g, n})_{2g - 2 + n > 0}
\end{equation*}
equipped with algebra homomorphisms
\begin{equation*}
q_A\colon \strata_{g, n} \to A^*(\ocM_{g, n}; \QQ).
\end{equation*}
As a $\QQ$-vector space, $\strata_{g, n}$ is spanned by pairs $(\Gamma, \gamma)$, where $\Gamma$ is an isomorphism class of stable graphs for $\ocM_{g, n}$ and $\gamma$ is a \emph{formal basic class} decorating $\Gamma$, that is, a monomial
\begin{equation*}
  \gamma
  = \prod_{\mathfrak{v} \in \mathrm{V}_\Gamma} \prod_{i > 0} \kappa_i[\mathfrak{v}]^{x_i[\mathfrak{v}]}
    \cdot \prod_{\mathfrak{h} \in \mathrm{H}_{\Gamma}} \psi_{\mathfrak{h}}^{y[\mathfrak{h}]},
   \quad
  \sum_{i>0} i\, x_i[\mathfrak{v}] + \sum_{\mathfrak{h} \text{ at } \mathfrak{v}} y[\mathfrak{h}]
  \le 3\mathrm{g}(\mathfrak{v}) - 3 + \mathrm{n}(\mathfrak{v})
  \quad \text{for all } \mathfrak{v} \in \mathrm{V}_\Gamma,
\end{equation*}
in the kappa-classes and psi-classes on the factors of $\ocM_\Gamma$. The product on $\strata_{g, n}$ is a formal implementation of the excess-intersection algorithm of \cite[Appendix]{GrPa03}; see \cite[Section 0.3]{PPZ15}. Interpreting $\gamma$ as the corresponding element of $A^*(\ocM_\Gamma; \QQ)$, the homomorphism $q_A$ is given by
\begin{equation*}
  q_A(\Gamma, \gamma) = \left(\mathpzc{gl}_{\Gamma}\right)_{\star}(\gamma),
\end{equation*}
and post-composing with the cycle class map yields
\begin{equation*}
q_H\colon \strata_{g, n} \to H^*(\ocM_{g, n}; \QQ).
\end{equation*}
The image of $q_A$ is the tautological Chow ring $R^*(\ocM_{g,n})$, and the image of $q_H$ is the tautological cohomology ring $RH^*(\ocM_{g,n})$. 

A \emph{tautological relation in Chow} (respectively, \emph{in cohomology}) is an element of $\ker(q_A)$ (respectively, $\ker(q_H)$).
Tautological relations are closed under multiplication by elements of the strata algebra and under push-forward and pull-back along gluing and forgetful morphisms. In particular, $\ker(q_A)$ is an ideal of $\strata$ in the sense that an \emph{ideal} $\mathcal I \subset \strata$ is a collection of ideals $\mathcal I_{g, n} \subset \strata_{g, n}$ closed under push-forward along the gluing and forgetful morphisms of Section \ref{subsec:CohFTs}.
Any collection $\{T_{g, n} \subset \ker(q_A)\}_{g, n}$ generates an ideal, obtained by closing under multiplication and these push-forwards.
In what follows, we compare systems of tautological relations by comparing the ideals they generate.

\subsection{Tautological relations from CohFTs}
\label{ss:taut-cohft}

Let $\Om$ be a generically semisimple CohFT over a reduced
$\QQ$-algebra $A$, and let $\Disc$ be its discriminant.

\subsubsection{Strata-valued Givental--Teleman reconstruction}
\label{sss:strata_valued_Givental_Teleman}
The
\emph{strata-valued Givental--Teleman reconstruction}  of  $\Om$ is
\begin{equation*}
  \Om^{\strata} \coloneqq RT\omega,
\end{equation*}
viewed as a system with values in $  \strata_{g,n}\otimes_{\QQ}A[\Disc^{-1}].$

For $n>0$, the reconstruction theorem gives
\begin{equation*}
  (q_H\otimes\operatorname{id})
  \bigl(\Om^{\strata}_{g,n}(\phi_1 \otimes \cdots \otimes  \phi_n)\bigr)
  =\Om_{g,n}(\phi_1\otimes \cdots \otimes \phi_n),
\end{equation*}
and the right-hand side is defined over $A$, not merely over
$A[\Disc^{-1}]$.

Let
\begin{equation*}
  q_{\mathrm{pol}}\colon A[\Disc^{-1}]
  \longrightarrow A[\Disc^{-1}]/A
\end{equation*}
be the polar-part quotient.  It follows that
\begin{equation*}
 (q_H\otimes\operatorname{id})
 \bigl((\operatorname{id}\otimes q_{\mathrm{pol}})
 \Om^{\strata}_{g,n}(\phi_1,\ldots,\phi_n)\bigr)=0.
\end{equation*}
The \emph{ideal of tautological relations associated to $\Om$}, denoted
$\cI^{\Om}\subset\strata$, is the tautological ideal generated by
\begin{equation*}
(\operatorname{id}_{\strata_{g,n}} \otimes \lambda)
  \bigl(\Om^{\strata}_{g,n}
  (\phi_1\otimes\cdots\otimes\phi_n)\bigr),
\end{equation*}
where $2g-2+n>0$, $n>0$, $\phi_i\in V$, and
$\lambda \, \colon A[\Disc^{-1}]\to\QQ$ is any $\QQ$-linear map vanishing
on $A$.

\subsubsection{Pixton's relations as the relations of a CohFT}
\label{sss:3-spin-ideal}

Let $\Om^{\spin}$ denote Witten's $3$-spin CohFT, see \cite[Section 3]{PPZ15}. Its Frobenius manifold is
two-dimensional and its Frobenius algebra at the origin is non-semisimple, although the Frobenius manifold is generically semisimple. Hence, the shifted Witten's $3$-spin CohFT
$(\Om^{\spin})^{\tau}$ by $\tau=(0,y)$ is generically semisimple over
$A=\QQ[y]$. We write
\begin{equation*}
  \cI^{\Om^{\spin}} \coloneqq \cI^{(\Om^{\spin})^{\tau}} \subset \strata
\end{equation*}
for the associated ideal of tautological relations in the sense above. More concretely,
it is generated by the coefficients of the strictly negative powers of $y$ in
$(\Om^{\spin})^{\tau,\strata}_{g,n}(\phi_1\otimes\cdots\otimes\phi_n)$.
By \cite
{paper1}, this ideal coincides with the
ideal generated by Pixton's relations \eqref{eqn:Pixtons_relations}, so we may and
do refer to $\cI^{\Om^{\spin}}$ as the \emph{ideal of $3$-spin relations}. By
\cite{Ja17} it consists of tautological relations in Chow.

\section{Background on Chiodo and Theta CohFTs}\label{sec:chiodo-theta-background}

\subsection{Chiodo classes}
Let $r, s, g, a_1, \ldots, a_n$ be integers with $r\geq 1$, and $g\geq 0$, satisfying
\begin{equation*}
\sum_{k=1}^na_k\equiv s(2g-2+n)\quad (\mathrm{mod}\ r)\,.
\end{equation*}
The \emph{moduli space of twisted stable spin curves} $\ocM_{g;a_1,\ldots,a_n}^{r,s}$ is the moduli space parametrizing the tuples $(C, x_1,\ldots,x_n, L)$ where $(C, x_1,\ldots,x_n)$ is a twisted stable curve in the sense of Abramovich--Vistoli \cite{AbVi02} of genus $g$ with $n$ (non-stacky) markings and $L$ is a (representable) line bundle on $C$ satisfying
\begin{equation*}
L^{\otimes r}\cong \omega_{\mathrm{log}}^{\otimes s}\left(-\sum_{k=1}^na_kx_k\right).
\end{equation*}
Here, $\omega_{\mathrm{log}}$ is the log canonical bundle of the curve $C$. The moduli space $\ocM_{g;a_1,\ldots,a_n}^{r,s}$ carries a universal curve $\overline{\mathcal{U}}_{g;a_1,\ldots,a_n}^{r,s}$ and a universal line bundle $\mathcal{L}_{g;a_1,\ldots,a_n}^{r,s}$:
\begin{equation*}
\begin{tikzcd}
\mathcal{L}_{g;a_1,\ldots,a_n}^{r,s} \arrow[r] & \overline{\mathcal{U}}_{g;a_1,\ldots,a_n}^{r,s} \arrow{d}{\pi} \\
& \ocM_{g;a_1,\ldots,a_n}^{r,s}.
\end{tikzcd}
\end{equation*}
The \emph{Chiodo class} is defined as
\begin{equation}\label{eqn:Chiodo_class_tilde}
\widetilde{\Ch}_{g,n}^{r,s,\tau}(a_1,\ldots,a_n) \coloneqq \mathpzc{f}_{\star}c_{\tau}(-R^{\bullet}\pi_{\star}\mathcal{L}_{g;a_1,\ldots,a_n}^{r,s})
\end{equation}
where $R^{\bullet}\pi_{\star}\mathcal{L}_{g;a_1,\ldots,a_n}^{r,s}$ is the derived pushforward of the universal line bundle viewed as a class in K-theory, $c_{\tau}(\bullet)$ is the Chern polynomial, and
\begin{equation*}
\mathpzc{f} \colon \ocM_{g;a_1,\ldots,a_n}^{r,s}\rightarrow \ocM_{g,n}
\end{equation*}
is the canonical forgetful map to moduli space of stable curves forgetting the extra data \cite{Ch08}.
\begin{remark}
The degree of the forgetful map $\mathpzc{f} \colon \ocM_{g;a_1,\ldots,a_n}^{r,s}\rightarrow \ocM_{g,n}$ is $r^{2g-1}$.
\end{remark}
Let
\begin{equation*}
V\coloneqq \left\langle v_0,\ldots,v_{r-1}\right\rangle_{\QQ}
\end{equation*}
be the $\QQ$-vector space spanned by $v_0,\ldots,v_{r-1}$ and
\begin{equation*}
\underline{V} \coloneqq \left\langle v_1,\ldots,v_{r-1}\right\rangle_{\QQ}
\end{equation*}
be the subspace of $V$ spanned by $v_1,\ldots,v_{r-1}$. Consider the multilinear map
\begin{equation*}
\widetilde{\Ch}_{g,n}^{r,s,\tau} \colon V^{\otimes n}\rightarrow H^*\left(\ocM_{g,n},\QQ[\tau]\right)\quad \text{defined by}\quad v_{a_1} \otimes \cdots \otimes v_{a_n}\mapsto \widetilde{\Ch}_{g,n}^{r,s,\tau}(a_1,\ldots,a_n).
\end{equation*}
The system $\widetilde{\Ch}^{r,s,\tau}=\left(\widetilde{\Ch}_{g,n}^{r,s,\tau}\right)_{2g-2+n>0}$ of multilinear maps forms a CohFT with the pairing $$\widetilde{\eta} (v_a,v_b)=\frac{1}{r}\delta_{a+b\equiv 0 (\mathrm{mod}\ r)}.$$

From now on, we assume $r\geq2$. For $-r+1\leq s\leq -1$, a Riemann-Roch calculation \cite[Section 2.2]{GLN23} shows that $-R^{\bullet} \pi_{\star}\mathcal{L}_{g;a_1,\ldots,a_n}^{r,s}$ in the definition of the Chiodo class is a genuine vector bundle $\mathcal{V}_{g;a_1,\ldots,a_n}^{r,s}$ of rank
\begin{equation}\label{eqn:rank_general_s}
\rank\left(\mathcal{V}_{g;a_1,\ldots,a_n}^{r,s}\right)
= \frac{(2g-2+n)(-s) + r(g-1) +\vert a\vert}{r}\quad \text{where} \quad \vert a\vert=\sum_{i=1}^{n}a_i.
\end{equation}
For $s=-1$, this reduces to
\begin{equation}\label{eqn:rank_general_minus_one}
\rank\left({\mathcal{V}_{g;a_1,\ldots,a_n}^{r,-1}}\right)=\frac{(r+2)(g-1)+n+\vert a\vert}{r} \quad \text{where} \quad \vert a\vert=\sum_{i=1}^{n}a_i.
\end{equation}
\begin{remark}
In this paper, we will mainly focus on the $s=-1$ case and we will often drop $s=-1$ from the notation in that case. For example, we will write $\mathcal{V}^r_{g;a_1,\ldots,a_n}$ to denote $\mathcal{V}^{r,-1}_{g;a_1,\ldots,a_n}$.
\end{remark}
Define the multilinear map $\Ch^{r}_{g,n} \colon V^{\otimes n}\rightarrow H^*\left(\ocM_{g,n},\QQ[t]\right)$ via\footnote{The polynomiality in $t$ follows from the identity
\begin{equation*}
t^{\rank(\mathcal{V})}c_{r/t}(\mathcal{V})=\sum_{k=0}^{\rank (\mathcal{V})}r^k t^{\rank(\mathcal{V})-k}\overline{c}_k(\mathcal{V})
\end{equation*} where $\overline{c}_k(\bullet)$ is the $k^{\text{th}}$ Chern class of the vector bundle $\mathcal{V}$.}
\begin{equation*}
\Ch^{r}_{g,n}(v_{a_1}\otimes\cdots\otimes v_{a_n})\coloneqq(-1)^nr^{1-g}{t}^{\rank\left({\mathcal{V}_{g;a_1,\ldots,a_n}^{r,-1}}\right)}\widetilde{\Ch}^{r,-1,r/t}_{g,n}(v_{a_1}\otimes\cdots\otimes v_{a_n})\in H^*\left(\ocM_{g,n},\QQ\right)[t].
\end{equation*}
The system $\Ch^{r}=\left(\Ch^{r}_{g,n}\right)_{2g-2+n>0}$ of multilinear maps form a CohFT over\footnote{Although each class $\Ch^{r}_{g,n}(v_{a_1}\otimes\cdots\otimes v_{a_n})$ lies in $H^*\left(\ocM_{g,n},\QQ\right)[t]$, since the pairing has a $t^{-1}$ term, $\Ch^r$ is a CohFT over the base ring $\QQ[t^{\pm 1}]$.} $\QQ[t^{\pm 1}]$ with the pairing
\begin{equation}\label{eqn:Chiodo_pairing_t}
\eta(v_i,v_j)=
\begin{cases}
t^{-1}\quad \text{if}\quad i=j=0,\\
1\quad \text{if} \quad 1\leq i=r-j\leq r-1,\\
0\quad \text{otherwise}.
\end{cases}
\end{equation}
The derivation of the rescaled pairing \eqref{eqn:Chiodo_pairing_t} is given in Appendix \ref{appendix:Metric_and_Quantum_Product}. We will refer to $\Ch^{r}$ as \emph{Chiodo CohFT}. We also consider the \emph{$\epsilon v_0$-shifted Chiodo CohFT} $\Ch^{r,\epsilon}=\left(\Ch^{r,\epsilon}_{g,n}\right)_{2g-2+n>0}$ defined by\footnote{The polynomiality in $\epsilon$ is a special case of Lemma \ref{lem:general_shift_polynomiality_lemma}.}
\begin{equation*}
\Ch_{g,n}^{r,\epsilon}(v_{a_1}\otimes\cdots\otimes v_{a_n}) \coloneqq \sum_{m\geq 0}\frac{\epsilon^m}{m!}{\mathpzc{p}_{m}}_\star\Ch_{g,n+m}^{r}(v_{a_1}\otimes\cdots\otimes v_{a_n}\otimes v_0^{\otimes m})\in H^*\left(\ocM_{g,n},\QQ\right)[\epsilon,t]
\end{equation*}
for any $v_{a_i}\in V$. Although every value of the shifted theory is polynomial in $t$, the pairing is unchanged, so $\Ch^{r,\epsilon}$ is naturally a CohFT over $\QQ[\epsilon,t^{\pm {1}}]$.

For further background on Chiodo classes, their graph expansions, and their functorial identities, see \cite[Section 2]{GLN23}. The pairing \eqref{eqn:Chiodo_pairing_t} is derived in Appendix \ref{appendix:Metric_and_Quantum_Product}.

\subsection{Theta classes}
Consider the multilinear map $\mathsf{\Upsilon}_{g,n}^{r} \colon V^{\otimes n}\rightarrow H^*\left(\ocM_{g,n},\QQ\right)$ defined by
\begin{equation*}
v_{a_1} \otimes \cdots \otimes v_{a_n}\mapsto 
\mathsf{\Upsilon}_{g,n}^{r}(v_{a_1}\otimes\cdots\otimes v_{a_n})=(-1)^nr^{\frac{2g-2+n+\vert a\vert}{r}}\mathpzc{f}_{\star}c_{\mathrm{top}}({\mathcal{V}_{g;a_1,\ldots,a_n}^{r}})
\end{equation*}
where $c_{\mathrm{top}}(\bullet)$ is the top Chern class. The \emph{theta class} is defined as
\begin{equation*}
\Th_{g,n}^{r} \coloneqq \mathsf{\Upsilon}_{g,n}^{r}\vert_{\underline{V}^{\otimes n}}.
\end{equation*}
The theta class was first constructed in \cite{Nor23} for $r=2$, and the construction was generalized to $r\geq 3$ in \cite{CGG25}.
\begin{proposition}[\cite{CGG25}]
The system $\Th^{r}=\left(\Th_{g,n}^{r}\right)_{2g-2+n>0}$ of multilinear maps forms a CohFT over $\QQ$ with the pairing $$\underline{\eta} (v_i,v_j)=\delta_{i+j\equiv 0 (\mathrm{mod}\ r)}.$$
\end{proposition}
The CohFT $\Th^{r}$ is called the \emph{theta CohFT}.
\begin{remark}
We should note that $\Th_{g,n}^{r}$ is actually the specialization of
$\Ch^{r}_{g,n}\vert_{\underline{V}^{\otimes n}}$ to $t=0$ since
\begin{equation*}
\begin{aligned}
\Ch^{r}_{g,n}(v_{a_1}\otimes\cdots\otimes v_{a_n})\vert_{t=0}
&=(-1)^nr^{1-g}{t}^{\rank\left({\mathcal{V}_{g;a_1,\ldots,a_n}^{r}}\right)}\widetilde{\Ch}^{r,-1,r/t}_{g,n}(v_{a_1}\otimes\cdots\otimes v_{a_n})\vert_{t=0}\\
&=(-1)^nr^{1-g}{t}^{\rank\left({\mathcal{V}_{g;a_1,\ldots,a_n}^{r}}\right)}\mathpzc{f}_{\star}c_{r/t}({\mathcal{V}_{g;a_1,\ldots,a_n}^{r}})\vert_{t=0}\\
&=(-1)^nr^{1-g}r^{\rank\left({\mathcal{V}_{g;a_1,\ldots,a_n}^{r}}\right)}\mathpzc{f}_{\star}c_{\mathrm{top}}({\mathcal{V}_{g;a_1,\ldots,a_n}^{r}})\\
&=(-1)^nr^{\frac{2g-2+n+\vert a\vert}{r}}\mathpzc{f}_{\star}c_{\mathrm{top}}({\mathcal{V}_{g;a_1,\ldots,a_n}^{r}})\\
&=\mathsf{\Upsilon}_{g,n}^{r}(v_{a_1}\otimes\cdots\otimes v_{a_n})
\end{aligned}
\end{equation*}
for all $v_{a_i}\in V$; the restriction to $\underline{V}$ enters only in the
definition $\Th^{r}_{g,n}=\mathsf{\Upsilon}^{r}_{g,n}\vert_{\underline{V}^{\otimes n}}$. Also, we  note that bilinear forms $\eta$ and $\underline{\eta}$ agree on $\underline{V}$ :
\begin{equation*}
\underline{\eta}=\eta\vert_{\underline{V}^{\otimes 2}}.
\end{equation*}
\end{remark}

In \cite{CGG25}, the authors define the \emph{$\epsilon v_0$-deformed theta class} to be
\begin{equation}\label{eqn:deformed_theta_class_definition}
\Th_{g,n}^{r,\epsilon}(v_{a_1}\otimes\cdots\otimes v_{a_n}) \coloneqq \sum_{m\geq 0}\frac{\epsilon^m}{m!}{\mathpzc{p}_{m}}_\star\mathsf{\Upsilon}_{g,n+m}^{r}(v_{a_1}\otimes\cdots\otimes v_{a_n}\otimes v_0^{\otimes m})
\end{equation}
for any $v_{a_i}\in \underline{V}$. However, as emphasized in \cite{CGG25}, we should note that this is not a shift in the sense of CohFTs since $v_0\notin \underline{V}$. The CohFT $\Th^{r,\epsilon}=\left(\Th_{g,n}^{r,\epsilon}\right)_{2g-2+n>0}$ is called the \emph{$\epsilon v_0$-deformed Theta CohFT}. 

\begin{remark}
In the expression \eqref{eqn:deformed_theta_class_definition}, specialization at $t=0$ commutes with
${\mathpzc{p}_{m}}_\star$ and with the sum over $m$ . Applying the previous remark on all of $V$ to each summand, we obtain that $\Th_{g,n}^{r,\epsilon}$ is the specialization of $\Ch_{g,n}^{r,\epsilon}\vert_{\underline{V}}$ to $t=0$. More precisely, for $v_{a_i}\in\underline V$,
\begin{equation}\label{eqn:theta_as_chiodo_specialization}
\Ch_{g,n}^{r,\epsilon}(v_{a_1}\otimes\cdots\otimes v_{a_n})\vert_{t=0}
=
\Th_{g,n}^{r,\epsilon}(v_{a_1}\otimes\cdots\otimes v_{a_n}).
\end{equation}
Again, this is an identity of multilinear classes after restricting the insertions, not an identification of the two CohFTs over the same state space.
\end{remark}

\subsection{Frobenius structures of $\Th^{r,\epsilon}$}

We denote by $\underline{\bullet}$ the quantum product on $\underline{V}$ induced by the deformed  theta CohFT $\Th^{r,\epsilon}$. In \cite[Lemma 3.4 and Proposition 3.5]{CGG25}, the quantum product $\underline{\bullet}$ is calculated as
\begin{equation}\label{eqn:quantum_product_of_deformed_theta}
v_i \ubullet v_j=
\begin{cases}
-v_{i+j+1}\quad &\text{if}\quad 2\leq i+j<r-1,\\
-\epsilon v_{i+j+2-r}\quad &\text{if}\quad r-1\leq i+j <2r-2,\\
-\epsilon^2v_1\quad &\text{if}\quad i+j=2r-2
\end{cases}
\end{equation} 
for $1\leq i,j\leq r-1$ and its unit is
\begin{equation}\label{eqn:unit_of_deformed_theta}
  \underline{\mathbf{1}}
  =
  \begin{cases}
    -\frac{v_1}{\epsilon^2}\quad &\text{if}\quad r=2,\\
    -\frac{v_{r - 2}}\epsilon \quad &\text{if}\quad r\geq 3.
  \end{cases}
\end{equation}
Set
\begin{equation*}
\mu \coloneqq e^{\frac{2\pi \sqrt{-1}}{r-1}}
\end{equation*}
and define
\begin{equation*}
\underline{p}(x)\coloneqq x^{r-1}+(-1)^r \epsilon.
\end{equation*}
The roots of $\underline{p}(x)$ are given by
\begin{equation}\label{eqn:underline_zeta_i_definition}
\underline{\zeta}_i\coloneqq-\mu^i\epsilon^{\frac{1}{r-1}} \quad \text{for} \quad 1\leq i \leq r-1.
\end{equation}
Furthermore, we define
\begin{equation}
  \underline{\xi}_i \coloneqq \sqrt{(1-r)\epsilon \underline{\zeta}_i}.
\end{equation}

\begin{remark}
For $r=2$, the above formulas degenerate: $\mu=1$, the index $i$ takes only the value
$1$, and $\underline{\zeta}_1=-\epsilon$, $\underline{\xi}_1^2=\epsilon^2$. As we noted, the unit \eqref{eqn:unit_of_deformed_theta} is correspondingly different. We treat the $r=2$ case separately wherever it matters in the paper.
\end{remark}

\begin{proposition}[\cite{CGG25}]\label{prop:CGG_idempotent_basis}
For $r\geq 2$, the deformed theta CohFT $\Th^{r,\epsilon}$ is semisimple for $\epsilon \neq 0$ and the idempotent basis is given by\footnote{
Our terminology differs from that of \cite{CGG25}. 
In \cite{CGG25}, the vectors $\underline{e}_k$ satisfying 
$\underline{e}_i \ubullet \underline{e}_j = \delta_{i,j}\underline{e}_j$ 
are called \emph{normalized idempotents}. In our paper, we reserve 
the term \emph{normalized idempotent} for vectors $\widetilde{\underline{e}}_k$ 
satisfying $\underline{\eta}(\widetilde{\underline{e}}_i, \widetilde{\underline{e}}_j)=\delta_{i,j}$, 
and we call $\underline{e}_k$ simply \emph{idempotents}.} are given by
\begin{equation*}
\underline{e}_{i}=-\frac{1}{r-1}\sum_{k=1}^{r-1}\mu^{-i(k+1)}\epsilon^{\frac{k+1}{1-r}}v_k
\end{equation*}
for $1\leq i \leq r-1$.
\end{proposition}

Based on Proposition \ref{prop:CGG_idempotent_basis}, we get the following lemma.

\begin{lemma}\label{lem:theta_normalized_idempotents}
For $r\geq 2$ and $1\leq i\leq r-1$, we have
\begin{equation}\label{eqn:norms_theta}
\underline{\Disc}_{i,\Th^{r,\epsilon}} \coloneqq \left(\underline{\eta}(\underline{e}_i,\underline{e}_i)\right)^{-1}=\underline{\zeta}_i^2\underline{\xi}_i^2,
\end{equation}
and the normalized idempotents of $\Th^{r,\epsilon}$ are given by
\begin{equation}\label{eqn:normalized_idempotents_theta}
\widetilde{\underline{e}}_i
=\sum_{k=1}^{r-1}
(-1)^{r-k}\frac{\underline{\zeta}_i^{r-k}}{\underline{\xi}_i}v_k.
\end{equation}
\end{lemma}

\begin{proof}
Firstly, note that
\begin{equation*}
\underline{\zeta}_i^{r-1}=(-1)^{r+1}\epsilon \quad \text{and}\quad \underline{\xi}_i^2=(1-r)\epsilon\underline{\zeta}_i.
\end{equation*}
Using these identities, we see that
\begin{equation*}
(-1)^{r-k}\frac{\underline{\zeta}_i^{r-k-1}}{\underline{\xi}_i^2}
=(-1)^{r-k}\frac{\underline{\zeta}_i^{r-1}\underline{\zeta}_i^{-k}}{\underline{\xi}_i^2}
=(-1)^{r-k}\frac{(-1)^{r+1}\epsilon\underline{\zeta}_i^{-k}}{(1-r)\epsilon\underline{\zeta}_i}=-\frac{1}{r-1}\left(-\underline{\zeta}_i\right)^{-(k+1)}=-\frac{1}{r-1}\mu^{-i(k+1)}\epsilon^{\frac{k+1}{1-r}}.
\end{equation*}
Hence, we have
\begin{equation*}
\underline{e}_{i}=\sum_{k=1}^{r-1}(-1)^{r-k}\frac{\underline{\zeta}_i^{r-k-1}}{\underline{\xi}_i^2}v_k.
\end{equation*}
Using the Frobenius property
\begin{equation*}
\underline{\eta} (\underline{e}_i,\underline{e}_i)=\underline{\eta} (\underline{e}_i,\underline{e}_i\ubullet \underline{\mathbf{1}})=\underline{\eta} (\underline{e}_i \ubullet \underline{e}_i,\underline{\mathbf{1}})=\underline{\eta} (\underline{e}_i,\underline{\mathbf{1}}),
\end{equation*}
and Proposition \ref{prop:CGG_idempotent_basis}, we get
\begin{equation*}
\underline{\eta} (\underline{e}_i,\underline{e}_i)
=\underline{\eta} (\underline{e}_i,\underline{\mathbf{1}})=\sum_{k=1}^{r-1}
(-1)^{r-k}\frac{\underline{\zeta}_i^{r-k-1}}{\underline{\xi}_i^2}\underline{\eta} (v_k,\underline{\mathbf{1}}).
\end{equation*}
For $r\geq 3$, we have
\begin{equation*}
\underline{\eta} (v_k,\underline{\mathbf{1}})=\underline{\eta} \left(v_k,-\frac{v_{r-2}}{\epsilon}\right)=-\epsilon^{-1}\delta_{k,2}=(-1)^r\underline{\zeta}_i^{1-r}\delta_{k,2}.
\end{equation*}
Combining the last two equations, we get
\begin{equation*}
\underline{\eta} (\underline{e}_i,\underline{e}_i)=\sum_{k=1}^{r-1}
(-1)^{r-k}\frac{\underline{\zeta}_i^{r-k-1}}{\underline{\xi}_i^2}(-1)^r\underline{\zeta}_i^{1-r}\delta_{k,2}=\sum_{k=1}^{r-1}
(-1)^{-k}\frac{\underline{\zeta}_i^{-k}}{\underline{\xi}_i^2}\delta_{k,2}=\frac{1}{\underline{\zeta}_i^2\underline{\xi}_i^2}.
\end{equation*}
Although $\underline{\mathbf{1}}$ is different for $r=2$, the formula for $\underline{\eta} (\underline{e}_i,\underline{e}_i)$ holds since
\begin{equation*}
\underline{\eta} (\underline{e}_1,\underline{e}_1)=\underline{\eta} \left(-\frac{v_1}{\epsilon^2},-\frac{v_1}{\epsilon^2}\right)=\frac{1}{\epsilon^4}=\frac{1}{\underline{\zeta}_1^2\underline{\xi}_1^2}.
\end{equation*}
So, we prove the equation \eqref{eqn:norms_theta}. Moreover, we see that
\begin{equation*}
\widetilde{\underline{e}}_i
=\underline{\Disc}_i^{\frac{1}{2}}\underline{e}_i
=\underline{\zeta}_i\underline{\xi}_i\underline{e}_i
=\underline{\zeta}_i\underline{\xi}_i\sum_{k=1}^{r-1}
  (-1)^{r-k}\frac{\underline{\zeta}_i^{r-k-1}}{\underline{\xi}_i^{2}}v_k
=\sum_{k=1}^{r-1}
  (-1)^{r-k}\frac{\underline{\zeta}_i^{r-k}}{\underline{\xi}_i}v_k .
\end{equation*}
which gives us the formula \eqref{eqn:normalized_idempotents_theta} for the normalized idempotents and completes the proof.
\end{proof}
Let $\underline{\Psi}$ be the transition matrix from the flat basis $\{v_i\}_{i=1}^{r-1}$ to the normalized idempotent basis $\{\widetilde{\underline{e}}_i\}_{i=1}^{r-1}$ and $\underline{\Psi}^{-1}$ be its inverse.
\begin{lemma}\label{lem:Psi_and_Psi_inverse_for_deformed_theta}
For $1\leq i,j\leq r-1$, the entries of the transition matrix $\underline{\Psi}$ and its inverse $\underline{\Psi}^{-1}$  are given by
\begin{equation*}
\underline\Psi_{i,j} =(-1)^j\frac{\underline{\zeta}_i^j}{\underline{\xi}_i} \quad \text{and}\quad \underline{\Psi}_{i,j}^{-1}=(-1)^{r-i}\frac{\underline{\zeta}_j^{r-i}}{\underline{\xi}_j}.
\end{equation*}
\end{lemma}

\begin{proof}
The expression for $\underline{\Psi}_{i,j}^{-1}$ follows directly from Lemma \ref{lem:theta_normalized_idempotents}. The expression for $\underline{\Psi}_{i,j}$ is calculated as follows:
\begin{equation*}
\underline\Psi_{i,j} =
\underline\eta(v_j,\widetilde{\underline e}_i)=\sum_{k=1}^{r-1}
(-1)^{r-k}\frac{\underline{\zeta}_i^{r-k}}{\underline{\xi}_i}\underline\eta(v_j,v_k)=\sum_{k=1}^{r-1}
(-1)^{r-k}\frac{\underline{\zeta}_i^{r-k}}{\underline{\xi}_i}\delta_{j+k,r}=(-1)^j\frac{\underline{\zeta}_i^j}{\underline{\xi}_i}.
\end{equation*}
\end{proof}

\begin{corollary}\label{cor:theta_discriminant}
  The discriminant of $\Th^{r,\epsilon}$ is
  \begin{equation*}
    \underline{\Disc}_{\Th^{r,\epsilon}}
    =\prod_{i=1}^{r-1}\underline{\Disc}_{i,\Th^{r,\epsilon}}
    =-(1-r)^{r-1}\epsilon^{r+2}.
  \end{equation*}
  In particular $\QQ[\epsilon][\underline{\Disc}_{\Th^{r,\epsilon}}^{-1}]
  =\QQ[\epsilon^{\pm 1}]$, so that the ideal $\cI^{\Th^{r,\epsilon}}$ of
  Section \ref{ss:taut-cohft} is generated by the coefficients of the strictly negative
  powers of $\epsilon$.
\end{corollary}

\begin{proof}
  Since $\underline{p}(x)=x^{r-1}+(-1)^{r}\epsilon$ is monic with constant term
  $(-1)^{r}\epsilon$, we have
  \begin{equation*}
    \prod_{i=1}^{r-1}\underline{\zeta}_i
  =(-1)^{r-1}(-1)^{r}\epsilon=-\epsilon .
  \end{equation*}
Hence, by equation \eqref{eqn:norms_theta},
  \begin{equation*}
    \prod_{i=1}^{r-1}\underline{\zeta}_i^{2}\underline{\xi}_i^{2}
    =\Bigl(\prod_{i=1}^{r-1}\underline{\zeta}_i\Bigr)^{2}
     \prod_{i=1}^{r-1}(1-r)\epsilon\underline{\zeta}_i
    =\epsilon^{2}\cdot(1-r)^{r-1}\epsilon^{r-1}\cdot(-\epsilon)
    =-(1-r)^{r-1}\epsilon^{r+2}. \qedhere
  \end{equation*}
\end{proof}

\subsection{Tautological relations from the Theta and Chiodo classes}
\label{subsec:theta_chiodo_relations}
We consider the ideal of tautological relations of the CohFTs $\Th^{r,\epsilon}$ and $\Ch^{r,\epsilon}$ in the sense of Section \ref{ss:taut-cohft}.
\subsubsection{The ideal of relations of $\Th^{r,\epsilon}$ and negative $r$-spin relations}
For $\epsilon\neq0$, the deformed Theta CohFT is semisimple.  Let
\begin{equation*}
  \Th^{r,\epsilon,\strata} \coloneqq 
  \underline R\,\underline T\,\underline\omega
\end{equation*}
denote its strata-valued Givental--Teleman reconstruction.  We write
$\cI^{\Th^{r,\epsilon}}$ for the ideal obtained from Section \ref{ss:taut-cohft}. Equivalently, it is generated by
the polar coefficients at $\epsilon=0$ of
\begin{equation*}
  \Th^{r,\epsilon,\strata}_{g,n}
  (v_{a_1}\otimes\cdots\otimes v_{a_n}),
   \quad 1\leq a_i\leq r-1,
\end{equation*}
with $n>0$. Set\footnote{When the modularity constraint holds, the number $D^r_{g;a_1,\ldots,a_n}$ is the rank of the vector bundle $\mathcal V_{g;a_1,\ldots,a_n}^{r,-1}$.}
\begin{equation*}
D^r_{g;a_1,\ldots,a_n}
 \coloneqq \frac{(r+2)(g-1)+n+\vert a\vert}{r} \quad \text{where} \quad \vert a\vert=\sum_{i=1}^{n}a_i.
\end{equation*}
The geometric class $\Th^{r,\epsilon}_{g,n}(v_{a_1}\otimes\cdots\otimes v_{a_n})$ is homogeneous of total (complex) degree $D^r_{g;a_1,\ldots,a_n}$ after assigning $\deg(\epsilon)=\frac{r-1}{r}$ and in particular has no
component in degree higher than $D^r_{g;a_1,\ldots,a_n}$. Hence \cite[Corollary 3.24]{CGG25} gives, for $n>0$,
\begin{equation*}
 q_H\left(\left[
  \Th^{r,\epsilon,\strata}_{g,n}(v_{a_1}\otimes\cdots\otimes v_{a_n})
 \right]_d\right)=0 \in H^{2d}(\ocM_{g,n};\QQ[\epsilon^{\pm 1}]),
  \quad d>D^r_{g;a_1,\ldots,a_n},
\end{equation*}
that is, the classes $\left[\Th^{r,\epsilon,\strata}_{g,n}(v_{a_1}\otimes\cdots\otimes
v_{a_n})\right]_d$ are tautological relations in cohomology for
$d>D^r_{g;a_1,\ldots,a_n}$. By homogeneity, every nonzero degree-$d$ part is a single Laurent monomial in $\epsilon$ multiplied by an element of $\strata_{g,n}$. When we speak of the corresponding relation, we mean this strata-algebra coefficient. We denote by $\cI^{\mathrm{neg},r}\subseteq\strata$ the
ideal they generate. These relations are the \emph{negative $r$-spin relations} of
\cite{CGG25}.

\begin{remark}
  In \cite[Proposition~3.15]{CGG25}, Chidambaram--Garcia-Failde--Giacchetto note that the Givental--Teleman reconstruction theorem holds for $\Th^{r,\epsilon,\strata}_{g,n}$ even when $n = 0$, except in degree $3g - 3$.
  It is a consequence of the proof of \cite[Proposition~3.15]{CGG25} and the definition of ideals of tautological relations, that even though we only used $\Th^{r,\epsilon,\strata}_{g,n}$ positive $n$ in the definition of $\cI^{\mathrm{neg},r}\subseteq\strata$, the ideal $\cI^{\mathrm{neg},r}\subseteq\strata$ does contain all relations considered in \cite{CGG25} including those for $n = 0$.
\end{remark}

\begin{lemma}
\label{lem:theta_relations_agree_with_cohft_ideal}
The ideal of negative $r$-spin relations agrees with the ideal of
tautological relations associated to the generically semisimple CohFT
$\Th^{r,\epsilon}$:
\begin{equation*}
\mathcal{I}^{\mathrm{neg},r}
=
\mathcal{I}^{\Th^{r,\epsilon}}.
\end{equation*}
\end{lemma}

\begin{proof}
Recall that the CohFT $\Th^{r,\epsilon}$ over $\QQ [\epsilon]$ is defined by
\begin{equation*}
\Th_{g,n}^{r,\epsilon}(v_{a_1}\otimes\cdots\otimes v_{a_n}) \coloneqq \sum_{m\geq 0}\frac{\epsilon^m}{m!}{\mathpzc{p}_{m}}_\star\mathsf{\Upsilon}_{g,n+m}^{r}(v_{a_1}\otimes\cdots\otimes v_{a_n}\otimes v_0^{\otimes m}).
\end{equation*}
By \cite[Section 2.3]{CGG25}, the class    ${\mathpzc{p}_{m}}_\star\mathsf{\Upsilon}_{g,n+m}^{r}(v_{a_1}\otimes\cdots\otimes v_{a_n}\otimes v_0^{\otimes m})$ has (complex) cohomological degree
\begin{equation*}
\deg \left({\mathpzc{p}_{m}}_\star\mathsf{\Upsilon}_{g,n+m}^{r}(v_{a_1}\otimes\cdots\otimes v_{a_n}\otimes v_0^{\otimes m})\right)
=
D^r_{g;a_1,\ldots,a_n}-m\frac{(r-1)}{r}.
\end{equation*}
With the degree assignment $\deg (\epsilon)=\frac{r-1}{r}$, the total (complex) degree of $ \Th_{g,n}^{r,\epsilon}(v_{a_1}\otimes\cdots\otimes v_{a_n})$ is $D^r_{g;a_1,\ldots,a_n}$.

The Givental--Teleman reconstruction preserves this homogeneity \cite[Proposition 3.6 and Theorem 3.20]{CGG25}. Hence, the (complex) cohomological degree $d$ part $\left[
  \Th^{r,\epsilon,\strata}_{g,n}
  (v_{a_1}\otimes\cdots\otimes v_{a_n})
 \right]_d$
of the strata-valued reconstruction satisfies
\begin{equation*}
\operatorname{ord}_{\epsilon} \left(\left[
  \Th^{r,\epsilon,\strata}_{g,n}(v_{a_1}\otimes\cdots\otimes v_{a_n})
 \right]_d\right)=\frac{r}{r-1}\left(D^r_{g;a_1,\ldots,a_n}-d\right),
\end{equation*}
where $\operatorname{ord}_{\epsilon}$ denotes the exponent of $\epsilon$. In
particular,
\begin{equation*}
\operatorname{ord}_{\epsilon} \left(\left[
  \Th^{r,\epsilon,\strata}_{g,n}(v_{a_1}\otimes\cdots\otimes v_{a_n})
 \right]_d\right)<0 \quad \text{if and only if} \quad d> D^r_{g;a_1,\ldots,a_n}.
\end{equation*}
Therefore, the ideal generated by the polar parts in the $\epsilon$ expansion is the same as the ideal given by the negative $r$-spin relations in \cite{CGG25}. Hence, we complete the proof.
\end{proof}
\subsubsection{The ideal of relations of $\Ch^{r,\epsilon}$ }
For later use, we also fix notation for the relations associated with the shifted Chiodo CohFT $\Ch^{r,\epsilon}$. Set
\begin{equation*}
A_{\Ch}\coloneqq \QQ[\epsilon,t^{ \pm 1}],
\end{equation*}
and let $\Disc_{\Ch^{r,\epsilon}}\in A_{\Ch}$ denote the Frobenius discriminant of $\Ch^{r,\epsilon}$. The calculations in the next section show that the CohFT $\Ch^{r,\epsilon}$ is generically semisimple. We write
\begin{equation*}
\Ch^{r,\epsilon,\strata}
\coloneqq RT\omega
\in
\strata\otimes_{\QQ}A_{\Ch}[\Disc_{\Ch^{r,\epsilon}}^{-1}]
\end{equation*}
for its strata-valued Givental--Teleman reconstruction. We denote by
\begin{equation*}
\cI^{\Ch^{r,\epsilon},\,t\neq0}\subseteq\strata
\end{equation*}
the ideal of tautological relations associated with this CohFT in the sense of Section \ref{ss:taut-cohft}.

\section{Analysis of $\epsilon v_0$-Shifted Chiodo CohFT}\label{sec:chiodo-analysis}
In this section, we determine the Frobenius algebra of the $\epsilon v_0$-shifted Chiodo CohFT $\Ch^{r,\epsilon}$, construct its idempotent and normalized idempotent bases, compute the transition matrices, and analyze the degeneration of this canonical data as $t\to0$. For the asymptotic statements, we fix $\epsilon\neq0$ and choose branches compatible with the deformed Theta theory.

\subsection{The quantum product of $\Ch^{r,\epsilon}$} 
The Chiodo pairing is given by
\begin{equation*}
  \eta(v_0, v_0) = t^{-1} \quad \text{and} \quad
  \eta(v_i, v_{r - i}) = 1\quad \text{for } i=1,\ldots,r-1
\end{equation*}
and its inverse is
\begin{equation*}
  \eta^{-1} = t v_0 \otimes v_0 + \sum_{i = 1}^{r-1} v_i \otimes v_{r - i}.
\end{equation*}

\begin{lemma}\label{lem:quantum_product_of_shifted_chiodo}
The quantum product for $\epsilon v_0$-shifted Chiodo CohFT $\Ch^{r,\epsilon}$ is given by
\begin{equation}\label{eqn:quantum_product_of_shifted_chiodo}
v_a\bullet v_b=
\begin{cases}
-v_{a+b+1}\quad &\text{if}\quad 0\leq a+b<r-1,\\
-\epsilon v_{a+b+2-r}-tv_{a+b+1-r}\quad &\text{if}\quad r-1\leq a+b <2r-2,\\
-\epsilon^2v_1-tv_{r-1}-t\epsilon v_0 \quad &\text{if}\quad a+b=2r-2.
\end{cases}
\end{equation} 
\end{lemma}
We refer the reader to Appendix \ref{appendix:Metric_and_Quantum_Product} for detailed derivations of the pairing \eqref{eqn:Chiodo_pairing_t} and quantum product \eqref{eqn:quantum_product_of_shifted_chiodo} given in Lemma \ref{lem:quantum_product_of_shifted_chiodo}  for the CohFT $\Ch^{r,\epsilon}$.

We see that the products by $v_0$ and $v_{r-1}$ are
\begin{equation}\label{eqn:special_cases_of_quantum_product}
  v_0 \bullet v_i =
  \begin{cases}
    - v_{i + 1} & \text{if }i < r-1, \\
    -\epsilon v_1 - t v_0 & \text{if }i = r-1,
  \end{cases}
\quad \text{and}\quad
  v_{r-1} \bullet v_{i} =
  \begin{cases}
    - \epsilon v_{i + 1}-t v_i & \text{if }i < r-1, \\
    -\epsilon^2 v_1 - t v_{r-1}-t\epsilon v_0 & \text{if }i = r-1.
  \end{cases}
\end{equation}
Hence, we conclude that the unit of the $\epsilon v_0$-shifted Chiodo CohFT $\Ch^{r,\epsilon}$ is given by
\begin{equation}\label{eqn:unit_of_shifted_Chiodo}
  \mathbf{1}
  = \frac{-v_{r - 1} + \epsilon v_0}t.
\end{equation}

By equation (\ref{eqn:special_cases_of_quantum_product}), we see that the basis vectors are generated cyclically\footnote{We define $v^{\bullet m}\coloneqq \underbrace{v\bullet \cdots \bullet v}_{m-\text{times}}$ for any $v\in V$.} by $v_0$:
\begin{equation}\label{eqn:v_i_interms_of_v_0}
  v_k=(-1)^{k}v_0^{\bullet (k+1)}
\end{equation}
for $0\leq k \leq r-1$. This implies that
\begin{equation}\label{eqn:characteristic_polymomial_of_chiodo}
0=-v_{r-1}+\epsilon v_0 -t\mathbf{1}=(-1)^rv_0^{\bullet r}+\epsilon v_0-t\mathbf{1}
\end{equation}
by the expression (\ref{eqn:unit_of_shifted_Chiodo}) of the unit $\mathbf{1}$ and equation (\ref{eqn:v_i_interms_of_v_0}). As a result, we see that $$p_t(x) = x^r + (-1)^r \epsilon x - (-1)^r t$$
is the characteristic polynomial of the linear operator
\begin{equation}\label{eqn:L_v0_operator}
\mathds{L}_{v_0} \colon V\rightarrow V \quad \text{defined linearly by}\quad \mathds{L}_{v_0}(v_i)=v_0\bullet v_i.
\end{equation}

\subsection{The characteristic polynomial $p_t(x)$}  
In this part, we will list certain properties of the  characteristic polynomial $p_t(x) = x^r + (-1)^r \epsilon x - (-1)^r t$.

Its discriminant is
\begin{equation}\label{eqn:disc_of_p_t_of_x}
 \Disc (p_t(x))=(-1)^{\frac{r(r-1)}{2}} ((-1)^{r-1} r^r t^{r - 1}-(r - 1)^{r - 1} \epsilon^r).
\end{equation}
Let $\zeta_i$ be the $r$ roots of $p_t(x)$. It is a well-known result that the roots of a polynomial over complex numbers are analytic in the coefficients of the polynomial away from the discriminant locus\cite{Br66}. We are especially interested in the behavior of $p_t(x)$ when $\epsilon\neq 0$ and $\vert t \vert$ is sufficiently small so that $\Disc (p_t(x))\neq 0$ for the values of $t$ we consider.

Firstly, note that we have
\begin{equation*}
p_0(x)=p_t(x)\vert_{t=0}=x^r + (-1)^r \epsilon x =x\underline{p}(x).
\end{equation*}
Hence, for $\vert t \vert$ sufficiently small it is expected that roots of $p_t(x)$ should be close to the roots $ \underline{\zeta}_1,\ldots, \underline{\zeta}_{r-1}$ of the polynomial $p_0(x)$ and $\underline{\zeta}_0\coloneqq 0$, which the following lemma establishes analytically.

\begin{lemma}\label{lem:series_expansion_for_roots}
For $\epsilon \neq 0$, the roots $\zeta_i$ of $p_t(x)$ have a power series expansion of the form
\begin{equation*}
\zeta_i=\underline{\zeta}_i+\frac{t}{\epsilon(1-r)}+O(t^2)\in \mathbb{C}[\epsilon^{\pm\frac{1}{r-1}}][\![t]\!] \quad\text{for}\quad i=1,\ldots,r-1,
\end{equation*}
and
\begin{equation*}
\zeta_0=\frac{t}{\epsilon}+O(t^2)\in \mathbb{C}[\epsilon^{\pm 1}][\![t]\!]\subset \mathbb{C}[\epsilon^{\pm\frac{1}{r-1}}][\![t]\!]
\end{equation*}
around $t=0$ with a positive radius of convergence.
\end{lemma}
\begin{proof}
Let $x=-\epsilon^{\frac{1}{r-1}}y$ and $t=-\rho\epsilon^{\frac{r}{r-1}}$. Then, we have
\begin{equation*}
p_t(x)=(-1)^r\epsilon^{\frac{r}{r-1}}y^r-(-1)^r\epsilon^{\frac{r}{r-1}}y-(-1)^rt=(-1)^r\epsilon^{\frac{r}{r-1}}\left(y^r-y+\rho\right).
\end{equation*}
For sufficiently small $\rho$, using the Lagrange inversion formula, one can express the roots of the polynomial $y^r-y+\rho$ as power series with nonzero radius of convergence \cite{Gl00}. Using the expansion\footnote{This result dates back to a century ago, even in a more general form \cite{B1927}. Yet, we thought it could be useful to refer the reader to a relatively modern reference \cite{Gl00}.} in \cite{Gl00}, we conclude that for fixed $\epsilon\neq 0$ and $t$ sufficiently small, the roots of $p_t(x)$ have the power series expansion\footnote{Recall that $\mu=e^{\frac{2\pi \sqrt{-1}}{r-1}}$.}
\begin{equation}\label{eqn:zeta_i_ring_containment}
\begin{aligned}
\zeta_i
&=-\epsilon^{\frac{1}{r-1}}\left(\mu^{i}
+\frac{t\epsilon^{-\frac{r}{r-1}}}{(r-1)}
\sum_{k=0}^{\infty}
\frac{\left(-t\epsilon^{-\frac{r}{r-1}}\mu^{-i}\right)^{k}}{\Gamma(k+2)}
\cdot
\frac{\Gamma \left(\frac{r k}{r-1}+1\right)}
     {\Gamma \left(\frac{k}{r-1}+1\right)}\right)\in \mathbb{C}[\epsilon^{\pm\frac{1}{r-1}}][\![t]\!]\\
&=\underline{\zeta}_i-\frac{t}{(r-1)\epsilon}\sum_{k=0}^{\infty}\left(\frac{t}{\epsilon \underline{\zeta}_i}\right)^k\frac{\Gamma \left(\frac{r k}{r-1}+1\right)}
     {\Gamma(k+2)\Gamma \left(\frac{k}{r-1}+1\right)}\in \mathbb{C}[\epsilon^{\pm\frac{1}{r-1}}][\![t]\!]\\
&=\underline{\zeta}_i+\frac{t}{\epsilon(1-r)}+O(t^2)\in \mathbb{C}[\epsilon^{\pm\frac{1}{r-1}}][\![t]\!] \quad\text{for}\quad i=1,\ldots,r-1,
\end{aligned}
\end{equation}
and
\begin{equation}\label{eqn:zeta_0_ring_containment}
\zeta_0=(-1)^{r-1}\epsilon \delta_{r,2}-\sum_{i=1}^{r-1}\zeta_i=\frac{t}{\epsilon}+O(t^2)\in \mathbb{C}[\epsilon^{\pm 1}][\![t]\!]\subset \mathbb{C}[\epsilon^{\pm\frac{1}{r-1}}][\![t]\!].
\end{equation}
\end{proof}

Define the degree $(r-1)$ polynomial $q_i(x)$ via
\begin{equation*}
q_i(x)= \prod_{j \neq i}(x - \zeta_j)=\sum_{k=0}^{r-1}q_{ik}x^k.
\end{equation*}
Since  $p_t(x)=(x-\zeta_i)q_i(x)$ for $0\leq i \leq r-1$, an inductive computation shows that the coefficients $q_{ik}$ are given by 
\begin{equation}\label{eqn:q_i(x)_coefficients}
q_{i0}=\frac{(-1)^rt}{\zeta_i} \quad \text{and} \quad q_{ik}=\zeta_i^{r-k-1} \quad \text{for}\quad k=1,\ldots,r-1.
\end{equation}

Let $\xi_i$ be a square root of
\begin{equation}\label{eqn:xi_i_definition}
  (-1)^r\zeta_i q_i(\zeta_i)=(-1)^r\zeta_i \prod_{j \neq i} (\zeta_i - \zeta_j) = ((1 - r)\epsilon \zeta_i + rt)
\end{equation}
chosen compatible with the limit
\begin{equation} \label{eqn:xi_limit_comparison}
  \lim_{t\to 0}\xi_i =\sqrt{(1-r)\epsilon \underline{\zeta}_i}=\underline{\xi}_i.
\end{equation}
Note also that we have the following identities
\begin{equation}\label{eqn:properties_of_q_i(x)}
\zeta_j q_i(\zeta_j)=\delta_{i,j}\zeta_i q_i(\zeta_i)=(-1)^r ((1 - r)\epsilon \zeta_i +rt)\delta_{i,j}=(-1)^r\xi_i^2\delta_{i,j}.
\end{equation}

Some of the quantities that we need to analyze involve finite negative
powers of $t$, while the appearance of square roots requires us to work with
$t^{1/2}$. For this reason, we consider the ring
\begin{equation*}
\mathbb{C}[\epsilon^{\pm\frac{1}{2(r-1)}}](\!(t^{\frac{1}{2}})\!)\coloneqq \left\{\sum_{m\geq m_0} \alpha_m t^{m/2}\; \vert \; \alpha_m\in \mathbb{C}[\epsilon^{\pm\frac{1}{2(r-1)}}],\ m_0\in \mathbb{Z}\right\},
\end{equation*}
i.e. the ring of Laurent series in $t^{1/2}$ with coefficients in the ring $\mathbb{C}[\epsilon^{\pm\frac{1}{2(r-1)}}]$. As usual, define the order of $t$
\begin{equation*}
\operatorname{ord}_t \colon  \mathbb{C}[\epsilon^{\pm\frac{1}{2(r-1)}}](\!(t^{\frac{1}{2}})\!)\rightarrow \frac{1}{2}\mathbb{Z}\cup\{+\infty\}
\end{equation*}
by
\begin{equation*}
\operatorname{ord}_t (\alpha)=
\begin{cases}
k \quad & \text{if} \quad   \alpha=t^k\sum_{j=0}^{\infty}\alpha_jt^{\frac{j}{2}} \quad \text{with}\, \alpha_0\neq 0\, \text{and}\, k\in \frac{1}{2}\mathbb{Z},\\
+\infty \quad & \text{if} \quad  \alpha=0.
\end{cases}
\end{equation*}

\begin{lemma}\label{lem:xi_i_and_zeta_i_and_inverses_ring_containment}
For $0\leq i\neq j \leq r-1$, we have $\zeta_i^{\pm 1},\xi_i^{\pm 1}, (\zeta_i-\zeta_j)^{-1} \in \mathbb{C}[\epsilon^{\pm\frac{1}{2(r-1)}}](\!(t^{\frac{1}{2}})\!)$ with
\begin{equation}
\operatorname{ord}_t (\zeta_i)=\delta_{i,0},\quad \operatorname{ord}_t (\xi_i)=\frac{1}{2}\delta_{i,0} \quad 
\text{and}\quad \operatorname{ord}_t (\zeta_i-\zeta_j)=0.
\end{equation}
\end{lemma}
\begin{proof}
By Lemma \ref{lem:series_expansion_for_roots}, we see that
\begin{equation}
\begin{aligned}\label{eqn:ring_containment_1}
\zeta_0^{-1}=&\frac{\epsilon}{t}\left(1+O(t)\right)\in \frac{1}{t}\cdot \mathbb{C}[\epsilon^{\pm\frac{1}{r-1}}][\![t]\!],\\
\zeta_i^{-1}=&\underline{\zeta}_i^{-1}\left(1+O(t)\right)\in \mathbb{C}[\epsilon^{\pm\frac{1}{r-1}}][\![t]\!] \quad 1\leq i \leq r-1,\\
\zeta_i-\zeta_j=&\underbrace{\underline{\zeta}_i-\underline{\zeta}_j}_{\neq 0}+O(t)\in \mathbb{C}[\epsilon^{\pm\frac{1}{r-1}}][\![t]\!]^\times, \quad 0\leq i\neq j \leq r-1,
\end{aligned}
\end{equation}
and consequently, we have
\begin{equation}\label{eqn:ring_containment_2}
(\zeta_i-\zeta_j)^{-1}=(\underline{\zeta}_i-\underline{\zeta}_j)^{-1}+O(t) \in \mathbb{C}[\epsilon^{\pm\frac{1}{r-1}}][\![t]\!]^\times \quad 0\leq i \neq j\leq r-1.
\end{equation}

Also, by equation \eqref{eqn:xi_i_definition} and Lemma \ref{lem:series_expansion_for_roots}, we observe that
\begin{equation}\label{eqn:ring_containment_3}
\begin{aligned}
\xi_0&=\sqrt{(1 - r)\epsilon \zeta_0 + rt}=\sqrt{t+O(t^2)}\in \sqrt{t}\cdot \mathbb{C}[\epsilon^{\pm\frac{1}{r-1}}][\![t]\!],\\
\xi_i&=\sqrt{ (1 - r)\epsilon \zeta_i + rt}=\sqrt{(1-r)\underline{\zeta}_i\epsilon+(r+1)t+O(t^2)}\in\sqrt{(1-r)\underline{\zeta}_i\epsilon}\cdot \mathbb{C}[\epsilon^{\pm\frac{1}{r-1}}][\![t]\!],\quad 1\leq i \leq r-1,
\end{aligned}
\end{equation}
which imply
\begin{equation}\label{eqn:ring_containment_4}
\xi_0^{-1}\in\frac{1}{\sqrt{t}}\cdot \mathbb{C}[\epsilon^{\pm 1}][\![t]\!]\quad \text{and} \quad \xi_i^{-1} \in \frac{1}{\sqrt{(1-r)\underline{\zeta}_i\epsilon}}\cdot \mathbb{C}[\underline{\zeta}_i^{-1},\epsilon^{-1}][\![t]\!], \quad 1\leq i \leq r-1.
\end{equation}
By the definition \eqref{eqn:underline_zeta_i_definition} of $\underline{\zeta_i}$, we see that all of the rings that appear in equations \eqref{eqn:ring_containment_1}, \eqref{eqn:ring_containment_2}, \eqref{eqn:ring_containment_3} and \eqref{eqn:ring_containment_4} are contained in the ring $\mathbb{C}[\epsilon^{\pm\frac{1}{2(r-1)}}](\!(t^{\frac{1}{2}})\!)$
which completes the first assertion of lemma.  The orders of  $\zeta_i$ and $\xi_i$ in $t$ follow from Lemma \ref{lem:series_expansion_for_roots} and equation \eqref{eqn:ring_containment_3} completing the proof.
\end{proof}

\subsection{Idempotents and transition matrices}
We start with a quick observation following from equation (\ref{eqn:v_i_interms_of_v_0}):
\begin{equation}\label{eqn:v_0bulletq_i(v_0)_coefficients}
v_0\bullet q_i(v_0)=\sum_{k=0}^{r-1}q_{ik}v_0^{\bullet (k+1)}=\sum_{k=0}^{r-1}(-1)^{k}q_{ik}v_{k}.
\end{equation}
Define\footnote{We define $\prod^{\bullet}_{j\neq i}\phi_j \coloneqq \phi_0 \bullet \cdots \bullet \phi_{i-1}\bullet \phi_{i+1}\bullet \cdots \bullet \phi_{r-1} $.}
\begin{equation*}
  e_i
  =\frac{q_i(v_0)}{q_i(\zeta_i)}=\frac{\prod^{\bullet}_{j \neq i}(v_0 - \zeta_j\mathbf{1})}{\prod_{j \neq i} (\zeta_i - \zeta_j)}.
\end{equation*}
Since $p_t(v_0)=0$ and $p_t(x)=(x-\zeta_i)q_i(x)$, we see that each $e_i$ is an eigenvector of $\mathds{L}_{v_0}$:
\begin{equation*}
  v_0 \bullet e_i
  = \zeta_i e_i.
\end{equation*}
Combining with equation (\ref{eqn:properties_of_q_i(x)}), we obtain an alternative expression for $e_i$:
\begin{equation}\label{eqn:idempotents_alternate_expression}
e_i=\frac{v_0\bullet q_i(v_0)}{\zeta_i q_i(\zeta_i)}=\frac{v_0\bullet q_i(v_0)}{(-1)^r\xi_i^2}
\end{equation}
which turns out to be more useful in computations.

Also, equation (\ref{eqn:properties_of_q_i(x)}) implies that $e_i$ forms an idempotent basis
\begin{equation*}
e_i\bullet e_j=\delta_{i,j}e_i
\end{equation*}
for all $0\leq i,j\leq r-1$. We also see that
\begin{equation}
  \label{eqn:v0-in-idempotents}
  v_0 = \sum_{i=0}^{r-1} \zeta_i e_i.
\end{equation}
This quickly generalizes, by equation (\ref{eqn:v_i_interms_of_v_0}), we get
\begin{equation}\label{eqn:vk-in-idempotents}
v_{k} = (-1)^k\sum_{i=0}^{r-1} \zeta_i^{k+1} e_i
\end{equation}
for $0\leq k \leq r-1$.

Observe the following computation for the norms of the idempotents:
\begin{equation*}
\eta(e_i,e_i)
= \eta(\mathbf{1},e_i)
= \frac{1}{t} \eta(-v_{r - 1} + \epsilon v_0, \frac{v_0\bullet q_i(v_0)}{(-1)^r\xi_i^2})
= \frac{1}{t(-1)^r\xi_i^2}\left( -\eta(v_{r - 1}, v_0\bullet q_i(v_0))+\epsilon\eta(v_0, v_0\bullet q_i(v_0))\right).
\end{equation*}
Then, by equation (\ref{eqn:v_0bulletq_i(v_0)_coefficients}) and the definition of $\eta$, we obtain that
\begin{equation}
\eta(e_i,e_i)=\frac{1}{t(-1)^r\xi_i^2}(q_{i1}+\epsilon t^{-1}q_{i0})=\frac{1}{t(-1)^r\xi_i^2}(\zeta_{i}^{r-2}+\epsilon t^{-1}(-1)^rt\zeta_i^{-1})=\frac{(-1)^r}{t\zeta_i^2\xi_i^2}(\zeta_{i}^{r}+\epsilon (-1)^r\zeta_i).
\end{equation}
Since $p_t(\zeta_i)=0$, we obtain that the norms of the idempotents are
\begin{equation}\label{eqn:norm_of_idempotents_shifted_chiodo}
\eta(e_i,e_i)=\frac{1}{\zeta_i^2\xi_i^2}
\end{equation}
and we set
\begin{equation} \label{eqn:Disc_i}
\Disc_{i,\Ch^{r,\epsilon}} \coloneqq \left(\eta(e_i,e_i)\right)^{-1}= \zeta_i^2 \xi_i^2.
\end{equation}
Define the Frobenius discriminant $\Delta_{\Ch^{r,\epsilon}}$ of $\Ch^{r,\epsilon}$
\begin{equation*}
\Disc_{\Ch^{r,\epsilon}} \coloneqq \prod_{i=0}^{r-1}\Disc_{i,\Ch^{r,\epsilon}} .
\end{equation*}
\begin{lemma}\label{lem:chiodo-frobenius-discriminant}
The discriminant $\Delta \left(p_t(x)\right)$ of $p_t(x)$
and the Frobenius discriminant $\Delta_{\Ch^{r,\epsilon}}$ for $\Ch^{r,\epsilon}$ is related by the identity:
\begin{equation*}
  \Disc_{\Ch^{r,\epsilon}}=(-1)^{\frac{r^2+r+2}{2}}t^3\Delta(p_t(x)).
\end{equation*}
\end{lemma}
\begin{proof}
By equations \eqref{eqn:xi_i_definition} and \eqref{eqn:Disc_i}, we get
\begin{equation*}
\Disc_{\Ch^{r,\epsilon}}
= \prod_{i=0}^{r-1} \Disc_{i,\Ch^{r,\epsilon}} 
=\prod_{i=0}^{r-1}\zeta_i^2 \xi_i^2=\prod_{i=0}^{r-1}\left(\zeta_i^2\cdot(-1)^r\zeta_i \prod_{j \neq i} (\zeta_i - \zeta_j)\right)
=(-1)^{r^2}\left(\prod_{i=0}^{r-1}\zeta_i^3\right)\left(\prod_{i=0}^{r-1}\prod_{j \neq i} (\zeta_i - \zeta_j)\right).
\end{equation*}
Then, by the definition of the discriminant $\Delta (p_t(x))$ and the formula of the product $\zeta_0 \zeta_1 \cdot \cdots \cdot \zeta_{r-1}=-t$ for the roots of $p_t(x)$, we conclude
\begin{equation*}
\Disc_{\Ch^{r,\epsilon}}=(-1)^{r^2}(-t)^3(-1)^{\frac{r(r-1)}{2}}\Delta(p_t(x))=(-1)^{\frac{r^2+r+2}{2}}t^3\Delta(p_t(x)).
\end{equation*}
\end{proof}

\begin{corollary}\label{cor:chiodo-semisimple-locus}
Over $A_{\Ch}=\QQ[\epsilon,t^{ \pm 1}]$, localizing at $\Disc_{\Ch^{r,\epsilon}}$ is equivalent to localizing at $\Delta \left(p_t(x)\right)$. Hence, the CohFT $\Ch^{r,\epsilon}$ is generically semisimple, and its semisimple locus is the complement of
\begin{equation*}
(-1)^{r-1}r^rt^{r-1}-(r-1)^{r-1}\epsilon^r=0.
\end{equation*}
\end{corollary}
\begin{proof}
The assertion follows from Lemma \ref{lem:chiodo-frobenius-discriminant}, since $t$ is invertible in $A$.
\end{proof}

By equation \eqref{eqn:norm_of_idempotents_shifted_chiodo}, the normalized\footnote{Normalized in the sense $\eta(\widetilde{e}_i,\widetilde{e}_j)=\delta_{i,j}$.} idempotents are given by
\begin{equation*}
  \widetilde{e}_i \coloneqq \zeta_i\xi_i e_i.
\end{equation*}
We observe that equations (\ref{eqn:q_i(x)_coefficients}), (\ref{eqn:v_0bulletq_i(v_0)_coefficients})  and (\ref{eqn:idempotents_alternate_expression}) imply
\begin{equation}\label{eqn:normalized_idempotents_in_vk}
\widetilde{e}_i=(-1)^r\frac{\zeta_i}{\xi_i} v_0 \bullet q_i(v_0)=\frac{t}{\xi_i}v_0+\sum_{k=1}^{r-1}(-1)^{r-k}\frac{\zeta_i^{r-k}}{\xi_i}v_k.
\end{equation}
Also, by the equation (\ref{eqn:vk-in-idempotents}), we get
\begin{equation}\label{eqn:vk_in_normalized_idempotents}
\begin{aligned}
v_k=\sum_{i=0}^{r-1} (-1)^k\zeta_i^{k+1}\frac{1}{\zeta_i\xi_i}\widetilde{e}_i=\sum_{i=0}^{r-1}\frac{(-1)^k\zeta_i^k}{\xi_i}\widetilde{e}_i
\end{aligned}
\end{equation}
for $0\leq k \leq r-1$.

Let $\Psi$ be the transition matrix from the flat basis $\{v_i\}_{i=0}^{r-1}$ to the normalized idempotent basis $\{\widetilde{e}_i\}_{i=0}^{r-1}$. Then, by equation (\ref{eqn:vk_in_normalized_idempotents}), we conclude that the entries of $\Psi$ are given by
\begin{equation} \label{eqn:Psi-formula}
\Psi_{i,j}=\eta(v_j,\widetilde{e}_i)=\frac{(-1)^j\zeta_i^j}{\xi_i}
\end{equation}
for $0\leq i, j \leq r-1$. Similarly, equation (\ref{eqn:normalized_idempotents_in_vk}) implies that the entries of the inverse $\Psi^{-1}$ are given by
\begin{equation}\label{eqn:Psi_Inverse}
\Psi^{-1}_{i,j}=\frac{1}{\xi_j}\cdot
\begin{cases}
t\quad &\text{if}\quad i=0,\\
(-1)^{r-i}\zeta_j^{r-i}\quad &\text{if}\quad 1\leq i\leq r-1
\end{cases}
\end{equation}
for $0\leq j \leq r-1$.

\subsection{Comparison of Frobenius structures as $t \to 0$}

\begin{remark}[Convention for the limit $t\to0$]
\label{rem:t-limit-convention}
Fix $\epsilon\neq0$ and the branches of the roots $\zeta_i(t)$ and
square roots $\xi_i(t)$ chosen above.  We regard all quantities as
Laurent series in $t^{1/2}$ with coefficients in $  \CC[\epsilon^{\pm 1/(2(r-1))}].$ The symbol $t\to0$ always refers to the constant term specialization of those Laurent expansions when it exists.

For a vector in $V$, the limit $t\to0$ is taken coefficientwise in the
fixed flat basis $v_0,\ldots,v_{r-1}$; for a matrix, it is taken
entrywise in the bases specified in its definition.  When comparing
with the deformed Theta theory, we identify
$\underline V=\langle v_1,\ldots,v_{r-1}\rangle$ with the corresponding
subspace of $V$ and restrict all inputs and matrix blocks to the
nonzero colors $1,\ldots,r-1$.
\end{remark}

\subsubsection{Limits recovering the deformed Theta Frobenius structure}
We see that the quantum product $\bullet$ of $\epsilon v_0$-shifted Chiodo CohFT $\Ch^{r,\epsilon}$ admits a limit as $t\to 0$, which coincides with the quantum product of $\epsilon v_0$-deformed theta CohFT $\Th^{r,\epsilon}$ when restricted to the state space $\underline{V}$  of  $\Th^{r,\epsilon}$.

\begin{corollary}\label{cor:limit_of_quantum_product} For all $\vv, \ww \in {V}$, the limit in the fixed flat basis
\begin{equation*}
\lim_{t \to 0} \vv \bullet \ww 
\end{equation*}
exists. Moreover, if $\vv, \ww \in \underline{V}$, then we have the limit that coincides with the quantum product of $\Th^{r,\epsilon}$:
  \begin{equation*}
   \lim_{t \to 0} \vv \bullet \ww = \vv \bullet \ww \vert_{t = 0}
    = \vv \ubullet \ww.
  \end{equation*}
\end{corollary}
\begin{proof}
It is clear that the quantum product expression \eqref{eqn:quantum_product_of_shifted_chiodo} admits a limit as $t\to 0$ and it matches with the quantum product \eqref{eqn:quantum_product_of_deformed_theta} of deformed theta CohFT when restricted on $\underline{V}$.
\end{proof}
\begin{lemma}\label{lem:limit_of_frobenius_structures}
The following limits hold.

\begin{enumerate}
    \item For $1\leq i,j\leq r-1$, we have
\begin{equation*}
\lim_{t \to 0}\Disc_{i,\Ch^{r,\epsilon}}=\underline{\Disc}_{i,\Th^{r,\epsilon}},\quad 
\lim_{t \to 0} \widetilde{e}_i=\widetilde{\underline{e}}_i,\quad
\lim_{t \to 0} \Psi_{i,j}=\underline{\Psi}_{i,j},\quad
\lim_{t \to 0} \Psi_{i,j}^{-1}=\underline{\Psi}^{-1}_{i,j}.
\end{equation*}
    \item For $1\leq i,j\leq r-1$, we have
\begin{equation*}
\lim_{t\to 0}\Psi_{0,j}=0,\quad \lim_{t\to 0}\Psi_{i,0}=\underline{\xi}_i^{-1}.
\end{equation*}
\item For $0\leq i\leq r-1$, we have
\begin{equation*}
\lim_{t\to 0}\Psi_{0,i}^{-1}=0.
\end{equation*}
\end{enumerate}

\end{lemma}

\begin{proof} For part (1), let $1\leq i,j\leq r-1$.
By
Lemma \ref{lem:series_expansion_for_roots} and the choice of branches in
equation \eqref{eqn:xi_limit_comparison}, we have
\begin{equation*}
  \lim_{t\to 0}\zeta_i=\underline{\zeta}_i \quad \text{and} \quad
  \lim_{t\to 0}\xi_i=\underline{\xi}_i.
\end{equation*}
Hence, by equations \eqref{eqn:norms_theta} and \eqref{eqn:Disc_i}, we obtain
\begin{equation*}
  \lim_{t\to 0}\Disc_{i,\Ch^{r,\epsilon}}
  =\lim_{t\to 0}\zeta_i^2\xi_i^2=
  \underline{\zeta}_i^2\underline{\xi}_i^2
  =\underline{\Disc}_{i,\Th^{r,\epsilon}}.
\end{equation*}
Since $\lim_{t\to 0}\xi_i=\underline{\xi}_i\neq0$, the coefficient $t/\xi_i$
tends to zero. As a result, we have
\begin{equation*}
\lim_{t\to 0}\widetilde e_i
=\lim_{t\to 0}\left(\frac{t}{\xi_i}v_0 +\sum_{k=1}^{r-1}(-1)^{r-k}\frac{\zeta_i^{r-k}}{\xi_i}v_k\right)=\sum_{k=1}^{r-1}
(-1)^{r-k}\frac{\underline{\zeta}_i^{r-k}}{\underline{\xi}_i}v_k
=\widetilde{\underline e}_i.
\end{equation*}
by equations \eqref{eqn:normalized_idempotents_theta} and \eqref{eqn:normalized_idempotents_in_vk}. The transition matrix formulas \eqref{eqn:Psi-formula} and
\eqref{eqn:Psi_Inverse} now give
\begin{equation*}
\lim_{t\to 0}\Psi_{i,j}
=\lim_{t\to 0}\left((-1)^j\frac{\zeta_i^j}{\xi_i}\right)
=(-1)^j\frac{\underline{\zeta}_i^j}{\underline{\xi}_i}
=\underline\Psi_{i,j}
\end{equation*}
and
\begin{equation*}
\lim_{t\to 0}\Psi^{-1}_{i,j}
=\lim_{t\to 0}\left((-1)^{r-i}\frac{\zeta_j^{r-i}}{\xi_j}\right)
=(-1)^{r-i}\frac{\underline{\zeta}_j^{r-i}}{\underline{\xi}_j}
=\underline\Psi^{-1}_{i,j},
\end{equation*}
where the last equalities use
Lemma \ref{lem:Psi_and_Psi_inverse_for_deformed_theta}.  This proves
part (1).

For part (2), first let $1\leq j\leq r-1$.  From
equation \eqref{eqn:Psi-formula} and
Lemma \ref{lem:xi_i_and_zeta_i_and_inverses_ring_containment}, we have
\begin{equation*}
\operatorname{ord}_t(\Psi_{0,j})
=j-\frac{1}{2}>0.
\end{equation*}
So, we see that
\begin{equation*}
\lim_{t\to 0}\Psi_{0,j}=0
\end{equation*}
If $1\leq i\leq r-1$, then
\begin{equation*}
\lim_{t\to 0}\Psi_{i,0}
=\lim_{t\to 0}\xi_i^{-1}
=\underline{\xi}_i^{-1}.
\end{equation*}

Finally, by equation \eqref{eqn:Psi_Inverse} and Lemma \ref{lem:xi_i_and_zeta_i_and_inverses_ring_containment}, we have
\begin{equation*}
  \operatorname{ord}_t\Psi^{-1}_{0,i}=1-\operatorname{ord}_t\xi_i=1-\frac{1}{2}\delta_{i,0}>0
\end{equation*}
for $0\leq i \leq r-1$. Hence, we get
\begin{equation*}
\lim_{t\to 0}\Psi_{0,i}^{-1}=0.
\end{equation*}
This proves part (3) and completes the proof.
\end{proof}

\subsubsection{Singular behaviors in Chiodo Frobenius structure}

\begin{lemma}
The unit
\begin{equation*}
  \mathbf{1}
  = \frac{-v_{r - 1} + \epsilon v_0}t
\end{equation*}
does not have a limit as $t \to 0$.
\end{lemma}

\begin{proof}
In the fixed flat basis, the unit $\mathbf{1}$ has a pole of order one at $t=0$
and has no limit in $V$.
\end{proof}

\begin{remark}
Although the limit of quantum product $\bullet$ matches with $\underline{\bullet}$ as $t\to 0$, we see that the unit  $\mathbf 1$ of the $\epsilon v_0$-shifted CohFT $\Ch^{r,\epsilon}$ does not admit a limit as $t \to 0$, and is very different from the unit $\underline{\mathbf 1}$ of the $\epsilon v_0$-deformed theta CohFT $\Th^{r,\epsilon}$ given by equation \eqref{eqn:unit_of_deformed_theta}. This limiting behavior is quite interesting and we are not aware of such a limiting phenomenon in the literature of CohFTs arising from enumerative geometry.
\end{remark}

\begin{lemma}
The limit
\begin{equation*}
\lim_{t \to 0}\Psi_{0,0}
\end{equation*}
does not exist. 
\end{lemma}

\begin{proof}
By definition, $\Psi_{0,0}=\xi_0^{-1}$.  Since
$\zeta_0=t/\epsilon+O(t^2)$, one has
\begin{equation*}
 \xi_0^2=(1-r)\epsilon\zeta_0+rt
 =(1-r)t+rt+O(t^2)=t+O(t^2).
\end{equation*}
With the chosen branch of the square root,
\begin{equation*}
\xi_0=t^{1/2}(1+O(t)),
  \quad
 \Psi_{0,0}=t^{-1/2}(1+O(t)).
\end{equation*}
Hence, $\Psi_{0,0}$ has a pole of order $1/2$ in $t$ and no finite limit.
\end{proof}

\section{$R$-Matrix, $T$-Vector of Shifted Chiodo CohFT $\Ch^{r,\epsilon}$}\label{sec:chiodo-rt}

In this section, we will analyze the dependence of the $R$-matrix and vacuum vector of shifted Chiodo CohFT $\Ch^{r,\epsilon}$ with respect to the parameter $\epsilon$. We will provide the explicit forms of the vacuum vector and the flatness equation governing the  $\epsilon$-dependence of $R$-matrix of $\Ch^{r,\epsilon}$.

Let
\begin{equation*}
R(z)=\sum_{k=0}^\infty R_kz^k=\mathrm{Id}+R_1z+R_2z^2+\cdots \in \mathrm{Id}+z\cdot \End(V)[\![z]\!]
\end{equation*}
be the $R$-matrix of the shifted Chiodo CohFT $\Ch^{r,\epsilon}$.  First, observe that we have
\begin{equation}\label{eqn:R_inverse_as_adjoint_series}
R^{-1}(z)=R^*(-z)=\sum_{k\geq 0}(-1)^kR^*_kz^k
\end{equation}
by the symplectic condition \eqref{eqn:symplectic-condition}.

\subsection{Flatness equations and $R$-matrix }
For background on Frobenius manifolds, canonical coordinates, and the discriminant locus, we refer to \cite[Section 2]{Ja18} and the references therein.

For $k\geq 0$, let $R_k^{\widetilde{e}}$ denote the matrix representation of $R_k$ in the
normalized idempotent basis $\{\widetilde{e}_i\}_{i=0}^{r-1}$, and let
$R_{i,j}^{k}$ denote its entries, so that
$R_k\widetilde{e}_j=\sum_{i}R^{k}_{i,j}\widetilde{e}_i$. Since the basis is
orthonormal with respect to the pairing $\eta$, this is equivalent to
\begin{equation}\label{eqn:R_matrix_in_normalized_idempotent_coordinates}
R_{i,j}^k=\eta (\widetilde{e}_i, R_k \widetilde{e}_j)=\eta (R_k^* \widetilde{e}_i, \widetilde{e}_j).
\end{equation}

\begin{proposition}\label{prop:flatness_equation}
For all $k\geq 0$, the following flatness equation holds\footnote{By convention, we have $R_{-1}^{\widetilde{e}}=0$.}
\begin{equation*}
\Psi \frac{\partial \Psi^{-1}}{\partial \epsilon}R_{k-1}^{\widetilde{e}}+\frac{\partial R_{k-1}^{\widetilde{e}}}{\partial \epsilon}=\frac{\partial U}{\partial \epsilon}R_k^{\widetilde{e}}-R_k^{\widetilde{e}}\frac{\partial U}{\partial \epsilon}
\end{equation*}
where $U$ denotes the diagonal matrix of canonical coordinates, and
\begin{equation*}
\frac{\partial U }{\partial \epsilon}=\diag (\zeta_0,\ldots, \zeta_{r-1}).
\end{equation*}
\end{proposition}

\begin{proof}
To prove the statement, we will follow an approach that is similar to the strategies in \cite[Section 1.5]{GT24}, \cite[Section 3.4]{LP19} \cite[Section 3.3]{PT19}.
Let
\begin{equation*}
\Vec{\tau}\coloneqq \sum_{i=0}^{r-1}\tau_iv_i.
\end{equation*}
The genus-zero potential of the unshifted Chiodo CohFT $\Ch^{r}$ is given by
\begin{equation*}
\mathcal{F}_0^r(\Vec{\tau})\coloneqq   \sum_{n\ge 3}\frac1{n!}
  \int_{\ocM_{0,n}}
  \Ch^r_{0,n}(\Vec{\tau}^{\otimes n}).
\end{equation*}
Set $\widehat{\mathcal{F}}_0^r (\Vec{\tau})\coloneqq \mathcal{F}_0^r(\Vec{\tau}+\epsilon v_0) $. We have
\begin{equation*}
\widehat{\mathcal{F}}_0^r (\Vec{\tau})
=\sum_{n\ge 3}\frac1{n!} \int_{\ocM_{0,n}} \Ch^r_{0,n}(\left(\Vec{\tau}+\epsilon v_0\right)^{\otimes n})
=\sum_{\substack{n,m\geq 0\\ n+m\geq 3}}\frac{\epsilon^m}{n!\,m!} \int_{\ocM_{0,n+m}}\Ch^{r}_{0,n+m}\big(\Vec{\tau}^{\otimes n}\otimes v_0^{\otimes m}\big)
\end{equation*}
and
\begin{equation}\label{eqn:partial_wrt_tau_is_partial_wrt_epsilon}
\frac{\partial \widehat{\mathcal{F}}_0^r (\Vec{\tau})}{\partial \tau_0}=\frac{\partial \widehat{\mathcal{F}}_0^r (\Vec{\tau})}{\partial \epsilon}.
\end{equation}
The genus-zero potential of the shifted Chiodo CohFT $\Ch^{r,\epsilon}$ is given by
\begin{equation*}
\mathcal{F}_0^{r,\epsilon}(\Vec{\tau})\coloneqq   \sum_{n\ge 3}\frac1{n!}
  \int_{\ocM_{0,n}}
  \Ch^{r,\epsilon}_{0,n}(\Vec{\tau}^{\otimes n})
\end{equation*}
and we have
\begin{equation*}
\widehat{\mathcal{F}}_0^r (\Vec{\tau})=\mathcal{F}_0^{r,\epsilon}(\Vec{\tau})+\text{Terms degree at most } 2 \text{ in }\tau_i.
\end{equation*}
Since computing the quantum product requires taking third partial derivatives in $\tau_i$, we see that both $\mathcal{F}_0^{r,\epsilon}(\Vec{\tau})$ and $\widehat{\mathcal{F}}_0^r (\Vec{\tau})$ yield the same quantum product on $V$ and induce the same Frobenius manifold structures.
By equation \eqref{eqn:partial_wrt_tau_is_partial_wrt_epsilon}, we conclude that the operator
\begin{equation*}
D\coloneqq \frac{\partial}{\partial \tau_0}-\frac{\partial}{\partial \epsilon}
\end{equation*}
also annihilates the Frobenius structures associated to shifted Chiodo CohFT $\Ch^{r,\epsilon}$ at a semisimple point $\Vec{\tau}$; namely, the transition matrix $\Psi_{\Vec{\tau}}$, its inverse $\Psi_{\Vec{\tau}}^{-1}$, the canonical coordinates $u_{\Vec{\tau},i}$, and the $R$-matrix $R_{\Vec{\tau}}(z)$ in the full set of variables $\tau_0,\ldots, \tau_{r-1}$:

\begin{equation*}
  D u_{\Vec{\tau},i}=0,
  \quad
  D\Psi_{\Vec{\tau}}=0, \quad D\Psi_{\Vec{\tau}}^{-1}=0, \quad \text{and}\quad DR_{\Vec{\tau}}(z)=0.
\end{equation*}
The flatness equation is given by
\begin{equation}\label{eqn:general_flatness_equation}
dR^{\widetilde{e}}_{\Vec{\tau}}(z)+\Psi_{\Vec{\tau}} (d\Psi^{-1}_{\Vec{\tau}})R^{\widetilde{e}}_{\Vec{\tau}}(z)=\frac{1}{z}\left [dU_{\Vec{\tau}},R^{\widetilde{e}}_{\Vec{\tau}}(z)\right ]
\end{equation}
where $d$ is the total differential with respect to $\tau_0,\ldots, \tau_{r-1}$, $U_{\Vec{\tau}}$ is the diagonal matrix with diagonal entries $u_{\Vec{\tau},i}$, and $[-,-]$ is the commutator, c.f. \cite[Corollary 1]{LP04}.

The transition matrix $\Psi$ and $R$-matrix $R^{\widetilde{e}}(z)$ of the shifted Chiodo $\Ch^{r,\epsilon}$ CohFT are the values of $\Psi_{\Vec{\tau}}$ and $R^{\widetilde{e}}_{\Vec{\tau}}(z)$ at the origin $\Vec{\tau}=\Vec{0}$:
\begin{equation*}
\Psi=\Psi_{\Vec{0}},\quad \text{and} \quad R^{\widetilde{e}}(z)=R^{\widetilde{e}}_{\Vec{0}}(z).
\end{equation*}
Moreover, we set
\begin{equation*}
U\coloneqq U_{\Vec{0}}.
\end{equation*}

Note that we have
\begin{equation}\label{eqn:R-matrix_mixed_partial}
\frac{\partial R^{\widetilde{e}}(z)}{\partial \epsilon}=\frac{\partial R^{\widetilde{e}}_{\Vec{0}}(z)}{\partial \epsilon}=\frac{\partial R^{\widetilde{e}}_{\Vec{\tau}}(z)}{\partial \epsilon}\bigg \vert_{\Vec{\tau}=\Vec{0}}=\frac{\partial R^{\widetilde{e}}_{\Vec{\tau}}(z)}{\partial \tau_0}\bigg \vert_{\Vec{\tau}=\Vec{0}}
\end{equation}
and similarly, we have
\begin{equation}\label{eqn:transition_canonical_coords_mixed_partial}
\frac{\partial \Psi^{\pm 1}}{\partial \epsilon}=\frac{\partial \Psi^{\pm 1}_{\Vec{\tau}}}{\partial \tau_0}\bigg \vert_{\Vec{\tau}=\Vec{0}} \quad \text{and}\quad \frac{\partial U}{\partial \epsilon}=\frac{\partial U_{\Vec{\tau}}}{\partial \tau_0}\bigg \vert_{\Vec{\tau}=\Vec{0}}.
\end{equation}

Since $v_0=\frac{\partial}{\partial \tau_0}\big\vert_{\Vec{\tau}=\Vec{0}}$ and $e_i=\frac{\partial}{\partial u_i}\big\vert_{\Vec{\tau}=\Vec{0}}$,  we have
\begin{equation}\label{eqn:v_0_in_e_i}
v_0
=\frac{\partial}{\partial \tau_0}\bigg\vert_{\Vec{\tau}=\Vec{0}}
=\sum_{i=0}^{r-1}\frac{\partial u_{\Vec{\tau},i}}{\partial \tau_0}\bigg \vert_{\Vec{\tau}
=\Vec{0}}\frac{\partial}{\partial u_i}\bigg\vert_{\Vec{\tau}=\Vec{0}}
=\sum_{i=0}^{r-1}\frac{\partial u_{\Vec{\tau},i}}{\partial \tau_0}\bigg \vert_{\Vec{\tau}
=\Vec{0}}e_i
=\sum_{i=0}^{r-1}\left [\frac{\partial U_{\Vec{\tau}}}{\partial \tau_0}\bigg \vert_{\Vec{\tau}=\Vec{0}}\right]_{i,i}e_i
=\sum_{i=0}^{r-1}\left[\frac{\partial U}{\partial \epsilon}\right]_{i,i}e_i
\end{equation}
by the chain rule and equation \eqref{eqn:transition_canonical_coords_mixed_partial}. Comparing equation \eqref{eqn:v_0_in_e_i} with the equation \eqref{eqn:v0-in-idempotents}, we get
\begin{equation*}
\left[\frac{\partial U}{\partial \epsilon}\right]_{i,i}
=\zeta_i
\end{equation*}
for all $0\leq i \leq r-1$.

By contracting the equation \eqref{eqn:general_flatness_equation} with $\frac{\partial}{\partial \tau_0}$, setting $\Vec{\tau}=\Vec{0}$ and using equations \eqref{eqn:R-matrix_mixed_partial}-\eqref{eqn:transition_canonical_coords_mixed_partial}, we get
\begin{equation*}
\Psi \frac{\partial \Psi^{-1}}{\partial \epsilon}R^{\widetilde{e}}(z)+\frac{\partial R^{\widetilde{e}}(z)}{\partial \epsilon}=\frac{1}{z}\left(\frac{\partial U}{\partial \epsilon}R^{\widetilde{e}}(z)-R^{\widetilde{e}}(z)\frac{\partial U}{\partial \epsilon}\right ).
\end{equation*}
Keeping track of the coefficient of $z^{k-1}$ on both sides, we complete the proof.
\end{proof}

Set
\begin{equation}\label{eqn:The_Matrix_A}
Y \coloneqq \Psi\frac{\partial\Psi^{-1}}{\partial\epsilon}=-\frac{\partial\Psi}{\partial\epsilon}\Psi^{-1}.
\end{equation}

The flatness equation takes the following form for the entries $R_{ij}^k$ of $R$-matrix coefficients.
\begin{corollary}[Flatness  equation for $\Ch^{r,\epsilon}$] \label{cor:flatness_equation_in_ij_form}
The flatness equations in the normalized idempotent basis take the form
\begin{equation}\label{eqn:flatness_in_normalized_idempotent_explicit}
\frac{\partial R_{i,j}^{k-1}}{\partial   \epsilon}+\sum_{l=0}^{r-1}Y_{i,l}R_{l,j}^{k-1}=(\zeta_i-\zeta_j)R_{i,j}^k
\end{equation}
for $0\leq i,j\leq r-1$ and $k\geq 0$. Here, the matrix $Y$ is a skew-symmetric matrix with the entries
\begin{equation}\label{eqn:The_Matrix_A_entries_explicit}
    Y_{i,j}=\frac{\zeta_i\zeta_j}{\xi_i \xi_j (\zeta_i-\zeta_j)}\quad \text{for} \quad i\neq j\quad \text{and}\quad Y_{i,i}=0
\end{equation}
where $0\leq i,j\leq r-1$. 
\end{corollary}
\begin{proof}
Substituting the expression given for $Y$ in equation \eqref{eqn:The_Matrix_A} into the flatness equation proved in Proposition \ref{prop:flatness_equation}, we get
\begin{equation*}
YR_{k-1}^{\widetilde{e}}+\frac{\partial R_{k-1}^{\widetilde{e}}}{\partial \epsilon}=\frac{\partial U}{\partial \epsilon}R_k^{\widetilde{e}}-R_k^{\widetilde{e}}\frac{\partial U}{\partial \epsilon},
\end{equation*}
and taking the $(i,j)$-entry of both sides, we obtain equation \eqref{eqn:flatness_in_normalized_idempotent_explicit}.

To complete the proof, we need to prove the equation \eqref{eqn:The_Matrix_A_entries_explicit}. By \cite[Lemma 7]{LP04}, the matrix $Y$ is skew-symmetric. In particular, we have $Y_{i,i}=0$ for all diagonal entries. Hence, it remains to calculate its off-diagonal entries.

Consider the linear operator $\mathds{L}_{v_0}$ given by equation \eqref{eqn:L_v0_operator}. Let $B$ be the matrix representation of $\mathds{L}_{v_0}$ in the flat basis $\{v_0,\ldots,v_{r-1}\}$. Note that we have
\begin{equation*}
B=\Psi^{-1} \frac{\partial U}{\partial \epsilon} \Psi
\end{equation*}
as each $\zeta_i$ is the eigenvalue of quantum multiplication by $v_0$ corresponding to the eigenvector $\widetilde{e}_i$ by equation \eqref{eqn:v0-in-idempotents}. As a result, we have
\begin{equation*}
\frac{\partial B}{\partial \epsilon}=\frac{\partial \Psi^{-1}}{\partial \epsilon} \frac{\partial U}{\partial \epsilon} \Psi +\Psi^{-1}\frac{\partial^2 U}{\partial \epsilon^2}\Psi+\Psi^{-1} \frac{\partial U}{\partial \epsilon} \frac{\partial \Psi}{\partial \epsilon}.
\end{equation*}
By the definition of $Y$, we see that
\begin{equation}
\Psi\frac{\partial B}{\partial \epsilon} \Psi^{-1}=Y \frac{\partial U}{\partial \epsilon}+ \frac{\partial^2 U}{\partial \epsilon^2}-\frac{\partial U}{\partial \epsilon} Y.
\end{equation}
Looking at the off-diagonal entries of both sides, we conclude that
\begin{equation}\label{eqn:A_ij_wrt_B_zeta}
Y_{i,j}=-\frac{\left [\Psi\frac{\partial B}{\partial \epsilon} \Psi^{-1}\right]_{i,j}}{\zeta_i-\zeta_j}
\end{equation}
for $0\leq i\neq j \leq r-1$.

It is clear that from the definition \eqref{eqn:L_v0_operator} of $\mathds{L}_{v_0}$ and the product formula \eqref{eqn:special_cases_of_quantum_product}, we see that the only $\epsilon$-dependent column of the matrix of $\mathds L_{v_0}$ is the column corresponding to $v_{r-1}$, and
\begin{equation*}
\frac{\partial (\mathds{L}_{v_0}(v_i))}{\partial \epsilon}=-v_1\delta_{i,r-1}
\end{equation*}
which implies that
\begin{equation*}
\left [ \frac{\partial B}{\partial \epsilon}\right ]_{i,j}=-\delta_{i,1} \cdot \delta_{j,r-1}.
\end{equation*}
By substituting this expression into equation \eqref{eqn:A_ij_wrt_B_zeta}, we obtain
\begin{equation*}
Y_{i,j}=\frac{1}{\zeta_i-\zeta_j}\sum_{k=0}^{r-1}\sum_{l=0}^{r-1}\Psi_{i,k}\delta_{k,1}\delta_{l,r-1}\Psi_{l,j}^{-1}=\frac{\Psi_{i,1}\Psi^{-1}_{r-1,j}}{\zeta_i-\zeta_j}=\frac{\zeta_i\zeta_j}{\xi_i \xi_j (\zeta_i-\zeta_j)}
\end{equation*}
for $0\leq i\neq j \leq r-1$ and complete the proof.
\end{proof}

\begin{remark}\label{rem:flatness-does-not-fix-diagonal-calibration}
Equation\eqref{eqn:flatness_in_normalized_idempotent_explicit} determines the off-diagonal entries recursively. The  diagonal entries $R^k_{i,i}$ at the same stage are determined up to an integration constant. In the asymptotic analysis below, the orders of $t$ in these integration constants are controlled by restricting the reconstruction to the smooth locus and applying Harer stability and using the equation \eqref{eqn:flatness_in_normalized_idempotent_explicit} itself.
\end{remark}

\subsection{Vacuum vector and $T$-vector}

\begin{proposition}\label{prop:vacuum_vector_of_shifted_chiodo}
For $r\geq 2$, the vacuum vector of shifted Chiodo CohFT $\Ch^{r,\epsilon}$ is given by
\begin{equation*}
\mathbf{v}(z)=(\epsilon v_0-v_{r-1})\sum_{k=0}^{\infty} z^k t^{-k-1}=\mathbf{1}\sum_{k=0}^{\infty} z^k t^{-k}.
\end{equation*}
\end{proposition}
\begin{proof}
Firstly, note that two given expressions agree due to equation \eqref{eqn:unit_of_shifted_Chiodo}, the explicit expression of $\mathbf{1}$.

Now, let 
\begin{equation*}
\Vec{\tau}\coloneqq \sum_{i=0}^{r-1}\tau_iv_i
\end{equation*}
and $\mathcal{F}_0^{r,\epsilon}(\Vec{\tau})$ be
as in the proof of Proposition \ref{prop:flatness_equation} and let $\mathbf{v}_{\Vec{\tau}}(z)$ be the associated vacuum vector in the full set of variables $\tau_0,\ldots, \tau_{r-1}$. Then, the vacuum vector $\mathbf{v}_{\Vec{\tau}}(z)$ satisfies the equation\footnote{Note that the shift formulation of \cite{Te12} is alternating.}
\begin{equation*}
\frac{\partial \mathbf{v}_{\Vec{\tau}}}{\partial \tau_0}=\frac{1}{z}\frac{\partial}{\partial \tau_0}\bullet_{\Vec{\tau}} ( \mathbf v_{\Vec{\tau}} (z)-\mathbf 1_{\Vec{\tau}})
\end{equation*}
by \cite[Proposition 7.13]{Te12}. Restricting $\Vec{\tau}=\Vec{0}$, and using the fact $v_0
=\frac{\partial}{\partial \tau_0}\big\vert_{\Vec{\tau}=\Vec{0}}$ and the annihilation argument
\begin{equation*}
\frac{\partial \mathbf{v}(z)}{\partial \epsilon}=\frac{\partial \mathbf{v}_{\Vec{\tau}}(z)}{\partial \tau_0}\bigg \vert_{\Vec{\tau}=\Vec{0}}
\end{equation*}
as in the proof of Proposition \ref{prop:flatness_equation}, we see that
the vacuum vector $\mathbf v (z)$ of the shifted Chiodo CohFT $\Ch^{r,\epsilon}$
 satisfies
\begin{equation*}
  \frac{\partial\mathbf v (z)}{\partial\epsilon} = \frac{v_0}z \bullet (\mathbf v (z)-\mathbf 1 ).
\end{equation*}
Writing $\mathbf v(z)=\sum_{k\geq0}\mathbf v_k z^k$ and clearing denominators, this
differential equation is equivalent to the system
\begin{equation*}
  v_0\bullet(\mathbf v_0 -\mathbf 1)=0,
   \quad
  \frac{\partial \mathbf v_{k-1}}{\partial\epsilon}=\,v_0\bullet\mathbf v_k
  \quad (k\geq 1).
\end{equation*}
By equation \eqref{eqn:v0-in-idempotents}, the operator $\mathds{L}_{v_0}$ has determinant
$\zeta_0\cdots\zeta_{r-1}=-t$ and is therefore invertible for $t\neq0$. Hence the
first equation forces $\mathbf v_0=\mathbf 1$, and the remaining equations determine
$\mathbf v_k=\mathds{L}_{v_0}^{-1}(\partial\mathbf v_{k-1}/\partial\epsilon)$
uniquely. Since $\partial\mathbf 1/\partial\epsilon=v_0/t$ and
$\mathds{L}_{v_0}^{-1}v_0=\mathbf 1$, induction gives
$\mathbf v_k= t^{-k}\mathbf 1$, which is the asserted formula.
\end{proof}

Let
\begin{equation*}
T(z)=\sum_{k=2}^{\infty} T_{k}z^{k} \in z^2V[\![z]\!]
\end{equation*}
be the translation vector of shifted Chiodo CohFT $\Ch^{r,\epsilon}$.

\begin{lemma} \label{lem:translation_vector_of_chiodo}
For $k\geq 1$, the coefficient $T_{k+1}$ of $z^{k+1}$ in the translation vector of shifted Chiodo CohFT $\Ch^{r,\epsilon}$ is given by
\begin{equation*}
T_{k+1}=\sum_{m=0}^{k} (-1)^{m+1}t^{-k+m}R^*_m\mathbf{1}.
\end{equation*}
\end{lemma}
\begin{proof}
The translation vector for $\epsilon v_0$-shifted Chiodo CohFT $\Ch^{r,\epsilon}$ is given by
\begin{equation*}
T(z)=z(\mathbf{1}-R^{-1}(z)\mathbf{v}(z)).
\end{equation*}
By equation \eqref{eqn:R_inverse_as_adjoint_series} and Proposition \ref{prop:vacuum_vector_of_shifted_chiodo}, we compute
\begin{equation*}
\begin{aligned}
T(z)
&=z\left(\mathrm{Id}-\left(\sum_{l=0}^{\infty}t^{-l}z^l\right)\left(\sum_{m=0}^{\infty} (-1)^mR^*_mz^m\right)\right) \mathbf{1}\\
&=z\left(\mathrm{Id}-\sum_{k=0}^{\infty}\sum_{m=0}^{k} (-1)^mt^{-(k-m)}R^*_mz^k\right) \mathbf{1}\\
&=\sum_{k=1}^{\infty}\sum_{m=0}^{k} (-1)^{m+1}t^{-k+m}R^*_m\mathbf{1}z^{k+1}
\end{aligned}
\end{equation*}
which completes the proof.
\end{proof}

\section{Givental--Teleman Reconstruction of CohFTs $\Ch^{r,\epsilon}$ and $\Th^{r,\epsilon}$} \label{sec:reconstruction}

In this section, we write the Givental--Teleman reconstructions of the shifted Chiodo CohFT $\Ch^{r,\epsilon}$ and the deformed Theta CohFT $\Th^{r,\epsilon}$ as sums over vertex-colored stable graphs. Since the reconstruction theorem for CohFTs with non-flat unit is used here only for classes with insertions, we impose $n>0$ throughout.

The standard graph-sum convention is the one recalled in Section \ref{ss:cohft-actions}. We denote the reconstruction data of $\Ch^{r,\epsilon}$ by $(R,T,\omega^{r,\epsilon})$ and those of $\Th^{r,\epsilon}$ by $(\underline{R},\underline{T},\underline{\omega}^{r,\epsilon})$. The deformed Theta reconstruction for $n>0$ is the one established in \cite[Theorem 3.20]{CGG25}, we only re-express it in our convention.

The formulas below can be read either as tautological cohomology classes after applying the gluing maps or formally in the strata algebra. The latter interpretation is the one used in the graph by graph limit argument of the next section.

\subsection{Reconstruction formula} 

The following expression for the topological field theory will be useful. Here $\omega^{r,\epsilon}=[\Ch^{r,\epsilon}]_0$ is the topological part of $\Ch^{r,\epsilon}$.
\begin{lemma}
  \label{lem:tft-formula}
  \begin{equation*}
    \omega_{g, n}^{r,\epsilon}(\phi_{1} \otimes \cdots \otimes \phi_{n} )
    = \sum_{i = 0}^{r-1} (\xi_i \zeta_i)^{2g - 2 + n} \prod_{j = 1}^n \eta(\phi_j, \widetilde{e}_i).
  \end{equation*}
\end{lemma}
\begin{proof}
It follows from \cite
{paper1} that  we have
\begin{equation*}
\omega_{g, n}^{r,\epsilon}(\phi_{1} \otimes \cdots \otimes \phi_{n} )
=
\sum_{i = 0}^{r-1} \Disc_{i,\Ch^{r,\epsilon}}^{g-1+\frac{n}{2}} \prod_{j = 1}^n \eta(\phi_j, \widetilde{e}_i).
\end{equation*}
Combining with the formula \eqref{eqn:Disc_i} for $\Disc_{i,\Ch^{r,\epsilon}}$, we complete the proof.
\end{proof}

The identical argument gives the topological field theory of the deformed Theta theory.

\begin{lemma}\label{lem:theta-tft-formula}
For every stable pair $(g,n)$ and all $\phi_1, \ldots, \phi_n \in \underline{V}$,
\begin{equation*}
\underline{\omega}_{g,n}^{r,\epsilon}(\phi_1\otimes\cdots\otimes\phi_n)
=
\sum_{i=1}^{r-1}(\underline{\xi}_i\underline{\zeta}_i)^{2g-2+n}
\prod_{j=1}^n\underline{\eta}(\phi_j,\widetilde{\underline e}_i).
\end{equation*}
\end{lemma}
\begin{proof}
By Lemma~\ref{lem:theta_normalized_idempotents}, we have
$\underline{\Disc}_{i,\Th^{r,\epsilon}}=(\underline{\zeta}_i \underline{\xi}_i)^2$. The standard semisimple TFT formula gives the claim \cite
{paper1}.
\end{proof}
\subsubsection{Trivial graph contribution}

Before stating the full graph sum formula, we record the contribution of the trivial stable graph (one vertex of genus $g$ with $n$ legs and no edges). This will also serve as a local model for arbitrary stable graphs and help us understand the sum over decorated stable graphs with vertex coloring decorations.

Now, we are ready to state the trivial graph contribution for the shifted Chiodo CohFT $\Ch^{r,\epsilon}$.
\begin{lemma}\label{lem:trivial_graph_contribution}
Let $2g-2+n>0$ and $s_1,\ldots, s_n \in \{0,1, \ldots, r-1\}$. The contribution $$\mathrm{Cont}^{\Ch^{r, \epsilon}}_{{\Gamma}_{\mathrm{tr}}}(v_{s_1} \otimes \cdots \otimes v_{s_n})
$$ of the trivial stable graph ${\Gamma}_{\mathrm{tr}}$  to the Givental--Teleman reconstruction of $\Ch^{r,\epsilon}(v_{s_1}\otimes \cdots \otimes v_{s_n})$ is given by
\begin{equation*}
\sum_{i = 0}^{r-1}\sum_{a_1,\ldots,a_n\geq 0}\prod_{j = 1}^n (-1)^{a_j} \eta(R_{a_j}^*v_{s_j}, \widetilde{e}_i)\psi_j^{a_j}\sum_{l=0}^{\infty}\frac{(\xi_i\zeta_i)^{2g - 2 + n+l}}{l!} \sum_{ k_1,\ldots,k_l\geq 2}    \prod_{j = {1}}^{l} \eta(T_{k_j}, \widetilde{e}_i) (\mathpzc{p}_l)_{\star} \left( \prod_{j=1}^{l}\psi_{n+j}^{k_j} \right).
\end{equation*}
\end{lemma}
\begin{proof}
By the Givental--Teleman classification, we have $\Ch^{r,\epsilon}_{g,n}=(RT\omega^{r,\epsilon})_{g,n}$. The contribution of the trivial stable graph is obtained by applying $R^{-1}(\psi_j)$ to each leg input of $T\omega^{r,\epsilon}$, that is,
\begin{equation*}
(T\omega^{r,\epsilon})(R^{-1}(\psi_1)v_{s_1}\otimes \cdots \otimes R^{-1}(\psi_n)v_{s_n}).
\end{equation*}
Using equation \eqref{eqn:R_inverse_as_adjoint_series}, multilinearity and applying translation action, we get
\begin{equation*}
\begin{aligned}
& (T\omega^{r,\epsilon})_{g,n}
\left( R^{-1}(\psi_1)v_{s_1}\otimes\cdots\otimes
R^{-1}(\psi_n)v_{s_n} \right)\\
&= \sum_{a_1,\ldots,a_n\geq 0} (-1)^{a_1+\cdots+a_n} \left(\prod_{i=1}^n\psi_i^{a_i}\right)(T\omega^{r,\epsilon})_{g,n}
\left(R_{a_1}^*v_{s_1}\otimes\cdots\otimes R_{a_n}^*v_{s_n}\right)\\
&= \sum_{a_1,\ldots,a_n\geq 0} (-1)^{a_1+\cdots+a_n} \left(\prod_{i=1}^n\psi_i^{a_i}\right) \sum_{l=0}^{\infty}\frac{1}{l!} (\mathpzc{p}_l)_{\star} \omega^{r,\epsilon}_{g,n+l} \left( R_{a_1}^*v_{s_1}\otimes\cdots\otimes R_{a_n}^*v_{s_n} \otimes T(\psi_{n+1})\otimes\cdots\otimes T(\psi_{n+l}) \right).
\end{aligned}
\end{equation*}
Unfolding the translation action and using Lemma \ref{lem:tft-formula}, we obtain
\begin{equation*}
\begin{aligned}
& (T\omega^{r,\epsilon})_{g,n}
\left( R^{-1}(\psi_1)v_{s_1}\otimes\cdots\otimes
R^{-1}(\psi_n)v_{s_n} \right)\\
&=
\sum_{l=0}^{\infty}\frac{1}{l!} \sum_{\substack{a_1,\ldots,a_n\geq 0\\ k_1,\ldots,k_l\geq 2}} (-1)^{a_1+\cdots+a_n} \omega^{r,\epsilon}_{g,n+l} \left(R_{a_1}^*v_{s_1}\otimes\cdots\otimes R_{a_n}^*v_{s_n} \otimes T_{k_1}\otimes\cdots\otimes T_{k_l}\right) \left(\prod_{i=1}^n\psi_i^{a_i}\right) (\mathpzc{p}_l)_{\star} \left( \prod_{j=1}^{l}\psi_{n+j}^{k_j} \right)\\
&=\sum_{l=0}^{\infty}\frac{1}{l!} \sum_{\substack{a_1,\ldots,a_n\geq 0\\ k_1,\ldots,k_l\geq 2}} (-1)^{a_1+\cdots+a_n} \sum_{i = 0}^{r-1} (\xi_i \zeta_i)^{2g - 2 + n+l} \prod_{j = 1}^n \eta(R_{a_j}^*v_{s_j}, \widetilde{e}_i) \prod_{j = {1}}^{l} \eta(T_{k_j}, \widetilde{e}_i) \left(\prod_{i=1}^n\psi_i^{a_i}\right) (\mathpzc{p}_l)_{\star} \left( \prod_{j=1}^{l}\psi_{n+j}^{k_j} \right)
\end{aligned}
\end{equation*}
Reorganizing the final expression, we complete the proof.
\end{proof}

\begin{remark}\label{rmk:decorated_stable_graph_refinement}
By the $R$-matrix action and the Givental--Teleman classification
(Theorem \ref{thm:Givental--Teleman_Classification}), the contribution of a
general stable graph ${\Gamma}$ is obtained by placing the trivial
graph contribution at each vertex and gluing these vertex contributions
together using the bivectors determined by the edges. The key observation is
that the sum $\sum_{i=0}^{r-1}$ appearing in Lemma
\ref{lem:trivial_graph_contribution} records the choice of normalized
idempotent at the unique vertex of the trivial stable graph. This naturally
leads to decorated stable graphs whose vertices are colored by normalized
idempotents, allowing us to express graph contributions more efficiently. For
a similar detailed discussion, see \cite[Section 3.2]{GT24}.
\end{remark}

We define a \emph{decorated stable graph} to be a stable graph equipped with a coloring decoration of its vertex set, i.e.\ it is a tuple $\widetilde{\Gamma}_c=({\Gamma}, c)$ where $\Gamma$ is a stable graph and 
\begin{equation*}
c \, \colon \mathrm{V}_{{\Gamma}}\rightarrow \mathrm{C}
\end{equation*}
is a function which is a \emph{coloring decoration} of ${\Gamma}$. Here, $\mathrm{C}$ is a finite set.

A \emph{flag weighting} is a tuple
\begin{equation*}
\mathrm A=(a_{\mathfrak h})_{\mathfrak h\in\mathrm H_\Gamma}
\in\mathbb Z_{\geq 0}^{\vert \mathrm{H}_{\Gamma}\vert}.
\end{equation*}
For a leg $\mathfrak l$, we write $a_{\ell(\mathfrak l)}:=a_{\mathfrak l}$. For each edge $\mathfrak e$, choose temporarily an ordering $\mathfrak e=(\mathfrak e_1,\mathfrak e_2)$ and put $b_{\mathfrak e_j}:=a_{\mathfrak e_j}$. The edge contribution is independent of this ordering by the symplectic condition, see Remark~\ref{rem:alternate_version_of_edge_contribution}.

\begin{remark}
By definition $\widetilde{\Gamma}_c$ and $\Gamma$ just differ by a coloring of vertices. In particular, they have the same vertex, edge, leg sets. For this reason, if a structure does not depend on coloring we will use $\Gamma$ instead of $\widetilde{\Gamma}_c$ in its notation. For example, we will write $\mathrm{V}_{\Gamma}$ instead of $\mathrm{V}_{\widetilde{\Gamma}_c}$ when we are talking about the vertex set of the decorated stable graph $\widetilde{\Gamma}_c$.
\end{remark}

\begin{remark}
For the CohFT $\Ch^{r,\epsilon}$, the coloring set $C_{\Ch^{r,\epsilon}}$ is $\{0,1,\ldots, r-1\}$ and for the CohFT $\Th^{r,\epsilon}$, the coloring set $C_{\Th^{r,\epsilon}}$ is $\{1,\ldots, r-1\}$.
\end{remark}

\subsubsection{Reconstruction formula for the $\epsilon v_0$-shifted Chiodo CohFT} 
For a stable graph $\Gamma$, we refine its Givental--Teleman contribution by summing over all vertex-colorings\footnote{We emphasize that the sum runs over all colorings of the fixed representative ${\Gamma}$ of an element of $\mathsf{G}_{g,n}$.}:
\begin{equation*}
\mathrm{Cont}^{\Ch^{r, \epsilon}}_{{\Gamma}}(v_{s_1} \otimes \cdots \otimes v_{s_n})
=\sum_{c\, \colon  \, \mathrm{V}_{{\Gamma}}\to\{0,\dots,r-1\}}
\mathrm{Cont}^{\Ch^{r, \epsilon}}_{\widetilde{\Gamma}_c}(v_{s_1} \otimes \cdots \otimes v_{s_n}).
\end{equation*}

\begin{proposition}\label{prop:Graph_Sum_Formula_Shifted_Chiodo} Let $2g-2+n>0$, $n>0$, and $s_1,\ldots, s_n \in \{0,1, \ldots, r-1\}$.
For a decorated stable graph $\widetilde{\Gamma}_c=({\Gamma},c)$, its contribution  to the Givental--Teleman reconstruction of $\Ch^{r,\epsilon}(v_{s_1}\otimes \cdots \otimes v_{s_n})$ is given by\footnote{Here $\mathrm{Inv} \colon \{0,1,\ldots,r-1\}\rightarrow \{0,1,\ldots,r-1\}$  is the involution defined by $\mathrm{Inv}(0)=0$ and $\mathrm{Inv}(i)=r-i$\quad for $i=1,\ldots,r-1$.}
\begin{equation*}
\mathrm{Cont}^{\Ch^{r, \epsilon}}_{\widetilde{\Gamma}_c}(v_{s_1} \otimes \cdots \otimes v_{s_n})=\frac{1}{\vert \mathrm{Aut}({\Gamma})\vert}\left({\mathpzc{gl}_{{\Gamma}}}\right)_{\star}\left(\sum_{\mathrm{A}\in \mathbb{Z}_{\geq 0}^{\vert \mathrm{H}_{\Gamma}\vert}}\prod_{\mathfrak{v}\in \mathrm{V}_{\Gamma}}\mathrm{Cont}^{\mathrm{A},\Ch^{r, \epsilon}}_{\widetilde{\Gamma}_c,\mathfrak{v}}\prod_{\mathfrak{l}\in \mathrm{L}_{\Gamma}}\mathrm{Cont}^{\mathrm{A},\Ch^{r, \epsilon}}_{\widetilde{\Gamma}_c,\mathfrak{l}}\prod_{\mathfrak{e}\in \mathrm{E}_{\Gamma}} \mathrm{Cont}^{\mathrm{A},\Ch^{r, \epsilon}}_{\widetilde{\Gamma}_c,\mathfrak{e}} \right)
\end{equation*}
where
\begin{equation*}
\begin{aligned}
\mathrm{Cont}^{\mathrm{A},\Ch^{r, \epsilon}}_{\widetilde{\Gamma}_c,\mathfrak{v}}
=&\sum_{l_{\mathfrak{v}}\geq 0}\frac{\left(\zeta_{\mathrm{c}(\mathfrak{v})}\xi_{\mathrm{c}(\mathfrak{v})}\right)^{2\mathrm{g}(\mathfrak{v})-2+\mathrm{n}(\mathfrak{v})+l_{\mathfrak{v}}}}{l_{\mathfrak{v}}!}\sum_{k_1,\ldots,k_{l_{\mathfrak{v}}}\geq 2}\prod_{j=1}^{l_{\mathfrak{v}}}\eta(T_{k_j},\widetilde{e}_{c(\mathfrak{v})}){\mathpzc{p}_{l_{\mathfrak{v}}}}_{*}\left(\prod_{j=1}^{l_{\mathfrak{v}}}\psi_{\mathrm{n}(\mathfrak{v})+j}^{k_j}\right),\\
\mathrm{Cont}^{\mathrm{A},\Ch^{r, \epsilon}}_{\widetilde{\Gamma}_c,\mathfrak{l}}
=&(-1)^{a_{\ell(\mathfrak{l})}}\eta(R^*_{a_{\ell(\mathfrak{l})}}v_{s_{\ell(\mathfrak{l})}},\widetilde{e}_{c(\nu (\mathfrak{l}))})\psi_{\ell(\mathfrak{l})}^{a_{\ell(\mathfrak{l})}},\\
\mathrm{Cont}^{\mathrm{A},\Ch^{r, \epsilon}}_{\widetilde{\Gamma}_c,\mathfrak{e}}
=&(-1)^{b_{\mathfrak{e}_1}+b_{\mathfrak{e}_2}}\sum_{m=0}^{b_{\mathfrak{e}_2}}(-1)^m \sum_{i = 0}^{r-1} \eta(R^*_{b_{\mathfrak{e}_1}+m+1}v_i,\widetilde{e}_{c(\mathfrak{v}_{\mathfrak{e}_1})}) \cdot\eta^{i,\mathrm{Inv}(i)}\cdot\eta(R^*_{b_{\mathfrak{e}_2}-m}v_{\mathrm{Inv}(i)},\widetilde{e}_{c(\mathfrak{v}_{\mathfrak{e}_2})})\psi_{\mathfrak{e}_1}^{b_{\mathfrak{e}_1}}\psi_{\mathfrak{e}_2}^{b_{\mathfrak{e}_2}}.
\end{aligned}
\end{equation*}
\end{proposition}

\begin{proof}
The expression for the trivial stable graph in Lemma \ref{lem:trivial_graph_contribution}, can be broken down to a sum over decorated trivial stable graphs whose unique vertex has color $c(\mathfrak{v})=i$:
\begin{equation*}
\sum_{a_1,\ldots,a_n\geq 0}\prod_{j = 1}^n \underbrace{(-1)^{a_j} \eta(R_{a_j}^*v_{s_j}, \widetilde{e}_i)\psi_j^{a_j}}_{=\mathrm{Cont}^{\mathrm{A},\Ch^{r, \epsilon}}_{\widetilde{\Gamma}_c,\mathfrak{l}}\,\text{with}\,\ell(\mathfrak{l})=j}\underbrace{\sum_{l=0}^{\infty}\frac{{(\zeta_i\xi_i)}^{2g - 2 + n+l}}{l!} \sum_{ k_1,\ldots,k_l\geq 2}    \prod_{j = {1}}^{l} \eta(T_{k_j}, \widetilde{e}_i) (\mathpzc{p}_l)_{\star} \left( \prod_{j=1}^{l}\psi_{n+j}^{k_j} \right)}_{=\mathrm{Cont}^{\mathrm{A},\Ch^{r, \epsilon}}_{\widetilde{\Gamma}_c,\mathfrak{v}}}.
\end{equation*}
As the automorphism group of the stable graph is trivial, the map $\mathpzc{gl}_{\Gamma}$ is identity and we have $\ocM_{\mathrm{g}(\mathfrak{v}),\mathrm{n}(\mathfrak{v})}=\ocM_{g,n}$. Hence, we have proved the proposition for decorated trivial stable graphs.

As we proved the statement for decorated trivial stable graphs and explained how it interplays with arbitrary decorated stable graphs in Remark \ref{rmk:decorated_stable_graph_refinement}, it only remains to prove part of the statement for the edge contributions of a decorated stable graph. The expression for the edge contribution in the $R$-matrix action corresponds to the bi-vector
\begin{equation*}
\frac{\eta^{-1}-\left(R^{-1}(\psi_{\mathfrak{e}_1}) \otimes R^{-1}(\psi_{\mathfrak{e}_2})\right)\eta^{-1}}{\psi_{\mathfrak{e}_1}+\psi_{\mathfrak{e}_2}}=\sum_{b_{\mathfrak{e}_1},b_{\mathfrak{e}_2}\geq 0}E_{b_{\mathfrak{e}_1},b_{\mathfrak{e}_2}}\psi_{\mathfrak{e}_1}^{b_{\mathfrak{e}_1}}\psi_{\mathfrak{e}_2}^{b_{\mathfrak{e}_2}}
\end{equation*}
where
\begin{equation*}
\begin{aligned}
E_{b_{\mathfrak{e}_1},b_{\mathfrak{e}_2}}
&=(-1)^{b_{\mathfrak{e}_1}+b_{\mathfrak{e}_2}} \sum_{m=0}^{b_{\mathfrak{e}_2}}(-1)^m(R^*_{b_{\mathfrak{e}_1}+m+1}\otimes R^*_{b_{\mathfrak{e}_2}-m})\eta^{-1}\\
&=(-1)^{b_{\mathfrak{e}_1}+b_{\mathfrak{e}_2}}\sum_{m=0}^{b_{\mathfrak{e}_2}}(-1)^m(R^*_{b_{\mathfrak{e}_1}+m+1}\otimes R^*_{b_{\mathfrak{e}_2}-m})\sum_{i=0}^{r-1}\eta^{i,\mathrm{Inv}(i)}v_{i}\otimes v_{\mathrm{Inv}(i)}\\
&=(-1)^{b_{\mathfrak{e}_1}+b_{\mathfrak{e}_2}}\sum_{m=0}^{b_{\mathfrak{e}_2}}(-1)^m\sum_{i=0}^{r-1}\eta^{i,\mathrm{Inv}(i)}(R^*_{b_{\mathfrak{e}_1}+m+1}v_{i}\otimes R^*_{b_{\mathfrak{e}_2}-m}v_{\mathrm{Inv}(i)}) 
\end{aligned}
\end{equation*}
Let
\begin{equation*}
E_{b_{\mathfrak{e}_1},b_{\mathfrak{e}_2}}=\sum_{\alpha,\beta}E_{b_{\mathfrak{e}_1},b_{\mathfrak{e}_2}}^{\alpha,\beta}\widetilde{e}_{\alpha}\otimes \widetilde{e}_{\beta}.
\end{equation*}
For a fixed flag $A$ and a fixed edge $\mathfrak{e}$ of a decorated stable graph, the exponents $b_{\mathfrak{e}_1}$,  $b_{\mathfrak{e}_2}$ of psi-classes are fixed and only $E_{b_{\mathfrak{e}_1},b_{\mathfrak{e}_2}}^{c(\mathfrak{v}_{\mathfrak{e}_1}),c(\mathfrak{v}_{\mathfrak{e}_2})}$ contributes as $\widetilde{e}_{c(\mathfrak{v}_{\mathfrak{e}_1})}\otimes \widetilde{e}_{c(\mathfrak{v}_{\mathfrak{e}_2})}$ is the only bi-vector that is compatible with the coloring of vertices to which the edge $\mathfrak{e}$ is attached. Hence, the edge contribution corresponds to
\begin{equation*}
\begin{aligned}
E_{a,b}^{c(\mathfrak{v}_{\mathfrak{e}_1}),c(\mathfrak{v}_{\mathfrak{e}_2})}
&=\eta(-,\widetilde{e}_{c(\mathfrak{v}_{\mathfrak{e}_1})})\otimes \eta(-,\widetilde{e}_{c(\mathfrak{v}_{\mathfrak{e}_2})})(E_{a,b})\\
&=(-1)^{b_{\mathfrak{e}_1}+b_{\mathfrak{e}_2}}\sum_{m=0}^{b_{\mathfrak{e}_2}}(-1)^m \sum_{i = 0}^{r-1} \eta(R^*_{b_{\mathfrak{e}_1}+m+1}v_i,\widetilde{e}_{c(\mathfrak{v}_{\mathfrak{e}_1})}) \cdot\eta^{i,\mathrm{Inv}(i)}\cdot\eta(R^*_{b_{\mathfrak{e}_2}-m}v_{\mathrm{Inv}(i)},\widetilde{e}_{c(\mathfrak{v}_{\mathfrak{e}_2})})
\end{aligned}
\end{equation*}
completing the proof.
\end{proof}

\begin{remark}\label{rem:alternate_version_of_edge_contribution}
Let $\mathfrak e$ be an edge of $\Gamma$ with ordered half-edges $\mathfrak e_1,\mathfrak e_2$ incident to
vertices $\mathfrak v_{\mathfrak e_1},\mathfrak v_{\mathfrak e_2}$.
The edge factor in Proposition \ref{prop:Graph_Sum_Formula_Shifted_Chiodo} is obtained by extracting the coefficient of $\psi_{\mathfrak e_1}^{b_{\mathfrak e_1}}\psi_{\mathfrak e_2}^{b_{\mathfrak e_2}}$ in the expression \eqref{eqn:R_matrix_action_edge_contribution_formula} of edge contribution of $R$-matrix action. Expanding $(\psi_{\mathfrak e_1}+\psi_{\mathfrak e_2})^{-1}$ either as a series in
$\psi_{\mathfrak e_2}/\psi_{\mathfrak e_1}$ or in $\psi_{\mathfrak e_1}/\psi_{\mathfrak e_2}$ yields two
a priori different coefficient formulas. These formulas  coincide because of the symplectic condition
$R(z)R^*(-z)=\mathrm{Id}$. Concretely, one may equivalently write
\begin{equation*}
\mathrm{Cont}^{\mathrm{A},\Ch^{r, \epsilon}}_{\widetilde{\Gamma}_c,\mathfrak{e}}
=(-1)^{b_{\mathfrak{e}_1}+b_{\mathfrak{e}_2}}\sum_{m=0}^{b_{\mathfrak{e}_1}}(-1)^m \sum_{i = 0}^{r-1} \eta(R^*_{b_{\mathfrak{e}_1}-m}v_i,\widetilde{e}_{c(\mathfrak{v}_{\mathfrak{e}_1})}) \cdot\eta^{i,\mathrm{Inv}(i)} \cdot \eta(R^*_{b_{\mathfrak{e}_2}+m+1}v_{\mathrm{Inv}(i)},\widetilde{e}_{c(\mathfrak{v}_{\mathfrak{e}_2})})\psi_{\mathfrak{e}_1}^{b_{\mathfrak{e}_1}}\psi_{\mathfrak{e}_2}^{b_{\mathfrak{e}_2}}.
\end{equation*}
\end{remark}

\subsubsection{Reconstruction formula for the $\epsilon v_0$-deformed Theta CohFT}

The CohFT $\Th^{r,\epsilon}$ is semisimple of rank $r-1$ and is supported on
\begin{equation*}
\underline V=\langle v_1,\ldots,v_{r-1}\rangle \subset V.
\end{equation*}
Recall that we denote its normalized idempotents, $R$-matrix, $T$-vector, bilinear form, and idempotent normalization constants by $\underline{\widetilde{e}}_i$, $\underline{R}$, $\underline{T}$, $\underline{\eta}$, $\underline{\xi}_i$, $\underline{\zeta}_i$ respectively, for $i\in\{1,\ldots,r-1\}$. As in the Chiodo case, for a stable graph $\Gamma$ we refine its Givental--Teleman contribution by summing over all vertex-colorings\footnote{We emphasize again that the sum runs over all colorings of the fixed representative $\Gamma$.}:
\begin{equation*}
\mathrm{Cont}^{\Th^{r, \epsilon}}_{{\Gamma}}(v_{s_1} \otimes \cdots \otimes v_{s_n})
=\sum_{c\, \colon  \, \mathrm{V}_{{\Gamma}}\to\{1,\dots,r-1\}}
\mathrm{Cont}^{\Th^{r, \epsilon}}_{\widetilde{\Gamma}_c}(v_{s_1} \otimes \cdots \otimes v_{s_n}).
\end{equation*}

\begin{proposition}\label{prop:Graph_Sum_Formula_Shifted_Theta}
Let $2g-2+n>0$, $n>0$, and $s_1,\ldots, s_n \in \{1, \ldots, r-1\}$.
For a decorated stable graph $\widetilde{\Gamma}_c$, its contribution  to the Givental--Teleman reconstruction of $\Th^{r,\epsilon}(v_{s_1}\otimes \cdots \otimes v_{s_n})$ is given by
\begin{equation*}
\mathrm{Cont}^{\Th^{r, \epsilon}}_{\widetilde{\Gamma}_c}(v_{s_1} \otimes \cdots \otimes v_{s_n})=\frac{1}{\vert \mathrm{Aut}({\Gamma})\vert} \left({\mathpzc{gl}_{{\Gamma}}}\right)_{\star}\left(\sum_{\mathrm{A}\in \mathbb{Z}_{\geq 0}^{\vert \mathrm{H}_{\Gamma}\vert}}\prod_{\mathfrak{v}\in \mathrm{V}_{\Gamma}}\mathrm{Cont}^{\mathrm{A},\Th^{r, \epsilon}}_{\widetilde{\Gamma}_c,\mathfrak{v}}\prod_{\mathfrak{l}\in \mathrm{L}_{\Gamma}}\mathrm{Cont}^{\mathrm{A},\Th^{r, \epsilon}}_{\widetilde{\Gamma}_c,\mathfrak{l}}\prod_{\mathfrak{e}\in \mathrm{E}_{\Gamma}} \mathrm{Cont}^{\mathrm{A},\Th^{r, \epsilon}}_{\widetilde{\Gamma}_c,\mathfrak{e}} \right)
\end{equation*}
where
\begin{equation*}
\begin{aligned}
\mathrm{Cont}^{\mathrm{A},\Th^{r, \epsilon}}_{\widetilde{\Gamma}_c,\mathfrak{v}}
=&\sum_{l_{\mathfrak{v}}\geq 0}\frac{\left(\underline{\zeta}_{\mathrm{c}(\mathfrak{v})}\underline{\xi}_{\mathrm{c}(\mathfrak{v})}\right)^{2\mathrm{g}(\mathfrak{v})-2+\mathrm{n}(\mathfrak{v})+l_{\mathfrak{v}}}}{l_{\mathfrak{v}}!}\sum_{k_1,\ldots,k_{l_{\mathfrak{v}}}\geq 2}\prod_{j=1}^{l_{\mathfrak{v}}}\underline{\eta}(\underline{T}_{k_j},\underline{\widetilde{e}}_{c(\mathfrak{v})}){\mathpzc{p}_{l_{\mathfrak{v}}}}_{*}\left(\prod_{j=1}^{l_{\mathfrak{v}}}\psi_{\mathrm{n}(\mathfrak{v})+j}^{k_j}\right),\\
\mathrm{Cont}^{\mathrm{A},\Th^{r, \epsilon}}_{\widetilde{\Gamma}_c,\mathfrak{l}}
=&(-1)^{a_{\ell(\mathfrak{l})}}\underline{\eta}(\underline{R}^*_{a_{\ell(\mathfrak{l})}}v_{s_{\ell(\mathfrak{l})}},\underline{\widetilde{e}}_{c(\nu (\mathfrak{l}))})\psi_{\ell(\mathfrak{l})}^{a_{\ell(\mathfrak{l})}},\\
\mathrm{Cont}^{\mathrm{A},\Th^{r, \epsilon}}_{\widetilde{\Gamma}_c,\mathfrak{e}}
=&(-1)^{b_{\mathfrak{e}_1}+b_{\mathfrak{e}_2}}\sum_{m=0}^{b_{\mathfrak{e}_2}}(-1)^m \sum_{i = 1}^{r-1} \underline{\eta}(\underline{R}^*_{b_{\mathfrak{e}_1}+m+1}v_i,\underline{\widetilde{e}}_{c(\mathfrak{v}_{\mathfrak{e}_1})}) \cdot\underline{\eta}^{i,\mathrm{Inv}(i)}\cdot\underline{\eta}(\underline{R}^*_{b_{\mathfrak{e}_2}-m}v_{\mathrm{Inv}(i)},\underline{\widetilde{e}}_{c(\mathfrak{v}_{\mathfrak{e}_2})})\psi_{\mathfrak{e}_1}^{b_{\mathfrak{e}_1}}\psi_{\mathfrak{e}_2}^{b_{\mathfrak{e}_2}}.\end{aligned}
\end{equation*}
\end{proposition}

\begin{proof}
The proof proceeds in exact parallel with the proof of Proposition \ref{prop:Graph_Sum_Formula_Shifted_Chiodo}. The trivial graph contribution $\mathrm{Cont}^{\Th^{r, \epsilon}}_{{\Gamma}_{\mathrm{tr}}}(v_{s_1} \otimes \cdots \otimes v_{s_n})$ for $\Th^{r,\epsilon}$ is obtained by the same argument as in Lemma \ref{lem:trivial_graph_contribution}, applied to $\Th^{r,\epsilon}=\underline R\,\underline T\,\underline\omega^{r,\epsilon}$, and is given by
\begin{equation*}
\sum_{i = 1}^{r-1}\sum_{a_1,\ldots,a_n\geq 0}\prod_{j = 1}^n (-1)^{a_j} \underline{\eta}(\underline{R}_{a_j}^*v_{s_j}, \underline{\widetilde{e}}_i)\psi_j^{a_j}\sum_{l=0}^{\infty}\frac{(\underline{\xi}_i\underline{\zeta}_i)^{2g - 2 + n+l}}{l!} \sum_{ k_1,\ldots,k_l\geq 2}    \prod_{j = {1}}^{l} \underline{\eta}(\underline{T}_{k_j}, \underline{\widetilde{e}}_i) (\mathpzc{p}_l)_{\star} \left( \prod_{j=1}^{l}\psi_{n+j}^{k_j} \right).
\end{equation*}
The sum over $i$ now ranges over $\{1,\ldots,r-1\}$ because $\underline\omega^{r,\epsilon}$ is diagonal in the basis $\{\underline{\widetilde e}_i\}_{i=1}^{r-1}$ of normalized idempotents of $\Th^{r,\epsilon}$. This expression can be broken down to a sum over decorated trivial stable graphs where the unique vertex has color $c(\mathfrak v)=i\in\{1,\ldots,r-1\}$:
\begin{equation*}
\sum_{a_1,\ldots,a_n\geq 0}\prod_{j = 1}^n \underbrace{(-1)^{a_j} \underline{\eta}(\underline{R}_{a_j}^*v_{s_j}, \underline{\widetilde{e}}_i)\psi_j^{a_j}}_{=\mathrm{Cont}^{\mathrm{A},\Th^{r, \epsilon}}_{\widetilde{\Gamma}_c,\mathfrak{l}}\,\text{with}\,\ell(\mathfrak{l})=j}\underbrace{\sum_{l=0}^{\infty}\frac{(\underline{\xi}_i\underline{\zeta}_i)^{2g - 2 + n+l}}{l!} \sum_{ k_1,\ldots,k_l\geq 2}    \prod_{j = {1}}^{l} \underline{\eta}(\underline{T}_{k_j}, \underline{\widetilde{e}}_i) (\mathpzc{p}_l)_{\star} \left( \prod_{j=1}^{l}\psi_{n+j}^{k_j} \right)}_{=\mathrm{Cont}^{\mathrm{A},\Th^{r, \epsilon}}_{\widetilde{\Gamma}_c,\mathfrak{v}}}.
\end{equation*}
As the automorphism group of the stable graph is trivial, the map $\mathpzc{gl}_{\Gamma}$ is identity and we have $\ocM_{\mathrm{g}(\mathfrak{v}),\mathrm{n}(\mathfrak{v})}=\ocM_{g,n}$, we prove the proposition for decorated trivial stable graphs.

As in the proof of Proposition \ref{prop:Graph_Sum_Formula_Shifted_Chiodo}, it only remains to prove the part of the statement for the edge contributions of a decorated stable graph. The expression for the edge contribution in the $\underline R$-matrix action corresponds to the bi-vector
\begin{equation*}
\frac{\underline\eta^{-1}-\left(\underline R^{-1}(\psi_{\mathfrak{e}_1}) \otimes \underline R^{-1}(\psi_{\mathfrak{e}_2})\right)\underline\eta^{-1}}{\psi_{\mathfrak{e}_1}+\psi_{\mathfrak{e}_2}}=\sum_{b_{\mathfrak{e}_1},b_{\mathfrak{e}_2}\geq 0}\underline E_{b_{\mathfrak{e}_1},b_{\mathfrak{e}_2}}\psi_{\mathfrak{e}_1}^{b_{\mathfrak{e}_1}}\psi_{\mathfrak{e}_2}^{b_{\mathfrak{e}_2}}
\end{equation*}
where
\begin{equation*}
\begin{aligned}
\underline E_{b_{\mathfrak{e}_1},b_{\mathfrak{e}_2}}
&=(-1)^{b_{\mathfrak{e}_1}+b_{\mathfrak{e}_2}} \sum_{m=0}^{b_{\mathfrak{e}_2}}(-1)^m(\underline R^*_{b_{\mathfrak{e}_1}+m+1}\otimes \underline R^*_{b_{\mathfrak{e}_2}-m})\underline\eta^{-1}\\
&=(-1)^{b_{\mathfrak{e}_1}+b_{\mathfrak{e}_2}}\sum_{m=0}^{b_{\mathfrak{e}_2}}(-1)^m(\underline R^*_{b_{\mathfrak{e}_1}+m+1}\otimes \underline R^*_{b_{\mathfrak{e}_2}-m})\sum_{i=1}^{r-1}\underline\eta^{i,\mathrm{Inv}(i)}v_{i}\otimes v_{\mathrm{Inv}(i)}\\
&=(-1)^{b_{\mathfrak{e}_1}+b_{\mathfrak{e}_2}}\sum_{m=0}^{b_{\mathfrak{e}_2}}(-1)^m\sum_{i=1}^{r-1}\underline\eta^{i,\mathrm{Inv}(i)}(\underline R^*_{b_{\mathfrak{e}_1}+m+1}v_{i}\otimes \underline R^*_{b_{\mathfrak{e}_2}-m}v_{\mathrm{Inv}(i)}).
\end{aligned}
\end{equation*}
Let
\begin{equation*}
\underline E_{b_{\mathfrak{e}_1},b_{\mathfrak{e}_2}}=\sum_{\alpha,\beta}\underline E_{b_{\mathfrak{e}_1},b_{\mathfrak{e}_2}}^{\alpha,\beta}\underline{\widetilde{e}}_{\alpha}\otimes \underline{\widetilde{e}}_{\beta}.
\end{equation*}
For a fixed flag $A$ and a fixed edge $\mathfrak{e}$ of a decorated stable graph, the exponents $b_{\mathfrak{e}_1}$, $b_{\mathfrak{e}_2}$ of psi-classes are fixed and only $\underline E_{b_{\mathfrak{e}_1},b_{\mathfrak{e}_2}}^{c(\mathfrak{v}_{\mathfrak{e}_1}),c(\mathfrak{v}_{\mathfrak{e}_2})}$ contributes as $\underline{\widetilde e}_{c(\mathfrak v_{\mathfrak e_1})}\otimes\underline{\widetilde e}_{c(\mathfrak v_{\mathfrak e_2})}$ is the only bi-vector that is compatible with the coloring of vertices to which the edge $\mathfrak{e}$ is attached. Hence, the edge contribution corresponds to
\begin{equation*}
\begin{aligned}
\underline E_{b_{\mathfrak e_1},b_{\mathfrak e_2}}^{c(\mathfrak{v}_{\mathfrak{e}_1}),c(\mathfrak{v}_{\mathfrak{e}_2})}
&=\underline\eta(-,\underline{\widetilde{e}}_{c(\mathfrak{v}_{\mathfrak{e}_1})})\otimes \underline\eta(-,\underline{\widetilde{e}}_{c(\mathfrak{v}_{\mathfrak{e}_2})})(\underline E_{b_{\mathfrak e_1},b_{\mathfrak e_2}})\\
&=(-1)^{b_{\mathfrak{e}_1}+b_{\mathfrak{e}_2}}\sum_{m=0}^{b_{\mathfrak{e}_2}}(-1)^m \sum_{i = 1}^{r-1} \underline\eta(\underline R^*_{b_{\mathfrak{e}_1}+m+1}v_i,\underline{\widetilde{e}}_{c(\mathfrak{v}_{\mathfrak{e}_1})}) \cdot\underline\eta^{i,\mathrm{Inv}(i)}\cdot\underline\eta(\underline R^*_{b_{\mathfrak{e}_2}-m}v_{\mathrm{Inv}(i)},\underline{\widetilde{e}}_{c(\mathfrak{v}_{\mathfrak{e}_2})})
\end{aligned}
\end{equation*}
completing the proof.
\end{proof}

\begin{remark}\label{rem:colored-contributions-strata-valued}
The formulas above are coefficientwise identities in the strata algebra: $(\mathpzc{gl}_\Gamma)_\star$ records the stratum indexed by $\Gamma$, and no relation among different boundary strata is used. They separate each summand into vertex, leg, and edge factors. In the next section, the factors attached to colors $1,\ldots,r-1$ are compared with their deformed-Theta counterparts as $t\to0$, while summands containing the distinguished color $0$ are shown to vanish.
\end{remark}
Hence, we have the following result.
\begin{corollary}\label{cor:full-colored-graph-reconstruction}
For $2g-2+n>0$ and $n>0$, the strata-valued reconstructions are
\begin{align*}
\Ch^{r,\epsilon,\strata}_{g,n}(v_{s_1}\otimes\cdots\otimes v_{s_n})
&=
\sum_{\Gamma\in\mathsf G_{g,n}}
\sum_{c\colon\mathrm V_\Gamma\to\{0,\ldots,r-1\}}
\mathrm{Cont}^{\Ch^{r,\epsilon}}_{\widetilde\Gamma_c}
(v_{s_1}\otimes\cdots\otimes v_{s_n}),\\
\Th^{r,\epsilon,\strata}_{g,n}(v_{s_1}\otimes\cdots\otimes v_{s_n})
&=
\sum_{\Gamma\in\mathsf G_{g,n}}
\sum_{c\colon\mathrm V_\Gamma\to\{1,\ldots,r-1\}}
\mathrm{Cont}^{\Th^{r,\epsilon}}_{\widetilde\Gamma_c}
(v_{s_1}\otimes\cdots\otimes v_{s_n}),
\end{align*}
where the second formula is stated for $1\le s_i\le r-1$. Applying $q_H$ gives the corresponding CohFT classes.
\end{corollary}


\section{Asymptotic Behavior of Shifted Chiodo CohFT as $t\to 0$}\label{sec:asymptotic-analysis}

In this section, we study the $t \to 0$ asymptotics of the $R$-matrix and translation vector of the $\epsilon v_0$-shifted Chiodo CohFT $\Ch^{r,\epsilon}$.  We first control the nonzero-color idempotent sector by restricting the Givental--Teleman reconstruction to the smooth locus and applying Harer stability.  We then combine homogeneity with the flatness equation to obtain the weaker, but sufficient, lower bounds for the zero-color sector that will be used in Section~\ref{sec:Limit_of_Givental_Teleman_Formula}.

Throughout this section, we fix $\epsilon\neq0$, choose the branches
$\zeta_i(t)$ and $\xi_i(t)$ as in Remark~\ref{rem:t-limit-convention}, and work over the ring
\begin{equation*}
    \CC[\epsilon^{\pm\frac{1}{2(r-1)}}]\bigl(\!\bigl(t^{\frac12}\bigr)\!\bigr).
\end{equation*}
All limits and $t$-orders are taken coefficientwise in this ring.  We use the convention
$\operatorname{ord}_t(0)=+\infty$.

\subsection{Quantities associated with the roots of $p_t(x)$}

We begin by recording the $t$-orders of the root-dependent quantities entering the flatness equation.  Recall that
\begin{equation*}
p_t(x)=x^r+(-1)^r\epsilon x-(-1)^rt,
\end{equation*}
and that, with our labeling, $\zeta_0=t/\epsilon+O(t^2)$ while
$\zeta_i=\underline\zeta_i+O(t)$ for $1\le i\le r-1$.

\begin{lemma}\label{lem:t_orders_of_Aij}
We have $Y_{i,j}\in \mathbb{C}[\epsilon^{\pm\frac{1}{2(r-1)}}](\!(t^{\frac{1}{2}})\!)$. Moreover, we have
\begin{equation*}
\operatorname{ord}_t\left(Y_{i,j}\right)
=\frac{1}{2}\delta_{i,0}+\frac{1}{2}\delta_{j,0}
\end{equation*}
for $0\leq i\neq j \leq r-1$.
\end{lemma}
\begin{proof}
Recall that we have $Y_{i,i}=0$ and 
\begin{equation}\label{eqn:A_ij_ring_containment}
Y_{i,j}=\frac{\zeta_i\zeta_j}{\xi_i\xi_j(\zeta_i-\zeta_j)}\quad \text{for}\quad i\neq j.
\end{equation}
As the constant term of $\zeta_0$ with respect to $t$ is zero and constant term of $\zeta_i$ with respect to $t$ is $\underline{\zeta}_i$, we have
\begin{equation}\label{eqn:zeta_i_minu_zeta_j_constant_term}
\zeta_i-\zeta_0=\underbrace{\underline{\zeta}_i}_{\neq 0}+O(t)\quad \text{and} \quad \zeta_i-\zeta_j=\underbrace{\underline{\zeta}_i-\underline{\zeta}_j}_{\neq 0}+O(t)
\end{equation}
for $i,j=1,\ldots,r-1$ and $i\neq j$. In other words, the difference $\zeta_i-\zeta_j$ is a unit of $ \mathbb{C}[\epsilon^{\pm\frac{1}{2(r-1)}}](\!(t^{\frac{1}{2}})\!)$ for $i\neq j$. Hence, by equation \eqref{eqn:A_ij_ring_containment} and Lemma \ref{lem:xi_i_and_zeta_i_and_inverses_ring_containment}, we conclude that $$Y_{i,j}\in \mathbb{C}[\epsilon^{\pm\frac{1}{2(r-1)}}](\!(t^{\frac{1}{2}})\!),$$
and
\begin{equation*}
\begin{aligned}
\operatorname{ord}_t\left(Y_{i,j}\right)
&=\operatorname{ord}_t\left(\zeta_i\zeta_j\right)-\operatorname{ord}_t\left(\xi_i\xi_j\right)-\operatorname{ord}_t\left(\zeta_i-\zeta_j\right)\\
&=\operatorname{ord}_t\left(\zeta_i\right)+\operatorname{ord}_t\left(\zeta_j\right)-\operatorname{ord}_t\left(\xi_i\right)-\operatorname{ord}_t\left(\xi_j\right)-\operatorname{ord}_t\left(\zeta_i-\zeta_j\right)\\
&=\delta_{i,0}+\delta_{j,0}-\frac{1}{2}\delta_{i,0}-\frac{1}{2}\delta_{j,0}-0\\
&=\frac{1}{2}\delta_{i,0}+\frac{1}{2}\delta_{j,0}
\end{aligned}
\end{equation*}
completing the proof.
\end{proof}

\begin{lemma}
For $0\leq i,j \leq r-1$, we have
\begin{equation}\label{eqn:Psi_orders_in_t}
\operatorname{ord}_t(\Psi_{i,j})=j\delta_{i,0}-\frac{1}{2}\delta_{i,0}
\end{equation}
and
\begin{equation}\label{eqn:Psi_Inverse_orders_in_t}
\operatorname{ord}_t(\Psi^{-1}_{i,j})=\begin{cases}
1-\frac{1}{2}\delta_{0,j}\quad &\text{if}\quad i=0,\\
\left(r-i-\frac{1}{2}\right)\delta_{0,j}\quad &\text{if}\quad 1\leq i\leq r-1
\end{cases}
\geq \frac{1}{2}\delta_{0,j}.
\end{equation}
\end{lemma}

\begin{proof}
The assertions follow directly from equations~\eqref{eqn:Psi-formula} and
\eqref{eqn:Psi_Inverse}, together with
\begin{equation*}
\operatorname{ord}_t(\zeta_i)=\delta_{i,0},
\qquad
\operatorname{ord}_t(\xi_i)=\frac12\delta_{i,0},
\end{equation*}
proved in Lemma~\ref{lem:xi_i_and_zeta_i_and_inverses_ring_containment}.
\end{proof}

\subsection{Polynomiality and bounds on the $t$ orders from Harer stability}

In this section, we use Harer stability to obtain certain polynomiality and limit results involving the $R$-matrix and $T$-vector of the $\epsilon v_0$-shifted Chiodo CohFT $\Ch^{r,\epsilon}$. The following form of Harer stability will be useful for our purposes.

\begin{proposition}[Harer stability]
  \label{prop:stability}
  If $3k < g$, then
  \begin{equation*}
    \left\{\mathpzc{p}_{l*}\left(\psi_1^{k_1} \dotsb \psi_{n + l}^{k_{n+l}}\right)\right\}_{l \ge 0, k_1, \dotsc, k_n \ge 0, k_{n+1}, \dotsc, k_{n+l} \ge 2, \sum_{i = 1}^{n+l} k_i = k + l}
  \end{equation*}
  forms a basis of $H^{2k}(\cM_{g, n})$.
\end{proposition}
\begin{proof}
  See the proof of \cite[Lemma 2.2]{MOPPZ17}.
\end{proof}

Using Harer stability, we will prove the following polynomiality and limit results for the $R$-matrix of $\epsilon v_0$-shifted Chiodo CohFT $\Ch^{r,\epsilon}$.
\begin{lemma} \label{lem:R-limit-nonzero}
  For $k\geq 0$,
  \begin{itemize}
      \item we have $$\eta(R_k^*(\phi), \widetilde{e}_i) \in \QQ[\zeta_i^{\pm 1}, \xi_i^{\pm 1}, \epsilon^{\pm 1}, t]$$ if $\phi \in V=\langle v_0, \dotsc, v_{r-1}\rangle_{\QQ}$ and $i = 0, \dotsc, r-1$,
      \item the limit $$\lim_{t \to 0} \eta(R_k^*(\phi), \widetilde{e}_i)$$ exists if $\phi \in V=\langle v_0, \dotsc, v_{r-1}\rangle_{\QQ}$ and $i = 1, \dotsc, r-1$,
      \item  and the limit is   
    $$\lim_{t \to 0} \eta(R_k^*(\phi), \widetilde{e}_i)
    = \underline{\eta}( \underline{R}_k^*(\phi), \widetilde{\underline{e}}_i)$$ if $\phi \in \underline{V}=\langle v_1, \dotsc, v_{r-1}\rangle_{\QQ}$ and $i = 1, \dotsc, r-1$.
  \end{itemize}

\end{lemma}
\begin{proof}
Let $n\geq 0$, $a_1, \dotsc, a_n \in \{1, \dotsc, r - 1\}$ and set
  \begin{equation}
    \label{eqn:ell-via-a}
    l \coloneqq n - 2 + \sum_{i=1}^n a_i.
  \end{equation}
For $\phi \in V$, the class $\Ch_{g, 1+n}^{r,\epsilon}( \phi \otimes v_{a_1} \otimes \cdots \otimes v_{a_n} )$ is a polynomial in $\epsilon$ and $t$ with cohomology coefficients in $H^*(\overline{\mathcal{M}}_{g, 1+n})$. Therefore, its restriction to the open locus of smooth curves is also polynomial:
  \begin{equation}
    \label{eqn:omega-alpha-smooth}
    \Ch_{g, 1+n}^{r,\epsilon}(\phi \otimes v_{a_1} \otimes \cdots \otimes v_{a_n})\vert_{\cM_{g, 1+n}} \in H^*(\cM_{g, 1+n})[\epsilon, t].
  \end{equation}

  To prove all the statements of the lemma, we work at the intersection of the Harer stability range $g > 3k$ and the stability range $2g-2+n>0$.
  
    By Proposition \ref{prop:stability} (Harer stability), we may consider the coefficient $B_{k, l}(r, \phi)$ of $\psi_1^k$ of \eqref{eqn:omega-alpha-smooth}.
  Since the trivial graph contribution given by Lemma \ref{lem:trivial_graph_contribution} corresponds to the locus of smooth curves, we see that
  \begin{equation*}
    B_{k, l}(r, \phi) = \sum_{i = 0}^{r-1} (\zeta_i \xi_i)^{2g - 2 + n+1} \eta((-1)^k R_k^*(\phi), \widetilde{e}_i) \prod_{j = 1}^{n} \eta(R_0^*(v_{a_j}), \widetilde{e}_i) \in \QQ[\epsilon, t].
  \end{equation*}
  Since $R_0^*$ is the identity transformation, the coefficient $B_{k, l}(r, \phi)$ further simplifies to
  \begin{equation}
        B_{k, l}(r, \phi)
        = \sum_{i = 0}^{r-1} (\zeta_i \xi_i)^{2g + n -1} \eta((-1)^k R_k^*(\phi), \widetilde{e}_i) \prod_{j = 1}^{n } \eta(v_{a_j}, \widetilde{e}_i)
  \end{equation}
and using equations \eqref{eqn:Psi-formula} and \eqref{eqn:ell-via-a}, we see that
  \begin{equation}
    \label{eqn:stability-system-R}
      \begin{aligned}
        B_{k, l}(r, \phi)
        &= \sum_{i = 0}^{r-1} (\zeta_i \xi_i)^{2g + n - 1} (-1)^k \eta(R_k^*(\phi), \widetilde{e}_i) (-1)^{l-n+2}\zeta_i^{l-n+2}\xi_i^{-n}\\
        &= (-1)^{l +n+ k} \sum_{i = 0}^{r-1} \zeta_i^{2g + 1+ l} \xi_i^{2g - 1} \eta(R_k^*(\phi), \widetilde{e}_i).
      \end{aligned}
  \end{equation}

\emph{$\bullet$ Polynomiality case I ($r\geq 3$):}  In this case, we can choose an integer $n=n_l$ and $a_1, \ldots,a_n$, so that $l$ can take any value in $\{1,\ldots, r-1\}$.  For $l \in \{0, \dotsc, r-1\}$, we view \eqref{eqn:stability-system-R} as a system of $r$ linear equations to solve for $\eta(R_k^*(\phi), \widetilde{e}_i)$.
  Using \eqref{eqn:Psi-formula}, we note that
  \begin{equation*}
    B_{k, l}(r, \phi) = (-1)^{k+n_l} \sum_{i = 0}^{r-1} \Psi_{i,l}\zeta_i^{2g + 1} \xi_i^{2g} \eta(R_k^*(\phi), \widetilde{e}_i).
  \end{equation*}
  Thus, we have
  \begin{equation*}
    \sum_{l = 0}^{r - 1} \Psi_{l, i}^{-1} (-1)^{k+n_l}B_{k, l}(r, \phi)
    = \sum_{l = 0}^{r - 1} \Psi_{l, i}^{-1} \sum_{j = 0}^{r-1} \Psi_{j,l}\zeta_j^{2g + 1} \xi_j^{2g} \eta(R_k^*(\phi), \widetilde{e}_j)
    = \zeta_i^{2g + 1} \xi_i^{2g} \eta(R_k^*(\phi), \widetilde{e}_i).
  \end{equation*}
  Therefore,
  \begin{equation}\label{eqn:eta_R_k_alpha_as_a_sum}
    \eta(R_k^*(\phi), \widetilde{e}_i)
    =  \zeta_i^{-2g - 1} \xi_i^{-2g} \sum_{l = 0}^{r - 1} \Psi_{l, i}^{-1} (-1)^{k+n_l}B_{k, l}(r, \phi).
  \end{equation}
As we have $B_{k, l}(r, \phi) \in \QQ[\epsilon, t]$ by Harer stability and $\Psi_{l, i}^{-1}  \in \QQ[t, \zeta_i, \xi_i^{-1}] $ by equation \eqref{eqn:Psi_Inverse}, we conclude the polynomiality statement in the lemma for $r\geq 3$:
  \begin{equation*}
    \eta(R_k^*(\phi), \widetilde{e}_i) \in \QQ[\epsilon, t, \zeta_i^{{\pm 1}}, \xi_i^{{\pm 1}}]\subseteq \QQ[\epsilon^{\pm 1}, t, \zeta_i^{{\pm 1}}, \xi_i^{{\pm 1}}]
  \end{equation*}
by equation \eqref{eqn:eta_R_k_alpha_as_a_sum}.

 \emph{$\bullet$ Polynomiality case II ($r=2$):} In this case, we have $a_1=\cdots =a_n=1$ and $l=2n-2$. Hence, $l$ can only take even values. Note that we have
 \begin{equation}\label{eqn:zeta_i_square_r_2}
    \zeta_i^2=t-\epsilon \zeta_i
 \end{equation}
as $p_t(x)=x^2+\epsilon x -t$ for $r=2$. Then, for $l=0$ ($n=n_0=1$) and $l=2$ ($n=n_2=2$), using equation \eqref{eqn:zeta_i_square_r_2}, we see that the equation \eqref{eqn:stability-system-R} takes the form
\begin{equation*}
\begin{aligned}
B_{k,0}(2, \phi)
&=(-1)^{k+1} \sum_{i = 0}^{1} \zeta_i^{2g + 1} \xi_i^{2g - 1} \eta(R_k^*(\phi), \widetilde{e}_i)
=(-1)^{k+1} \sum_{i = 0}^{1} \Psi_{i,0}\zeta_i^{2g + 1} \xi_i^{2g} \eta(R_k^*(\phi), \widetilde{e}_i),\\
B_{k,2}(2, \phi)
&=(-1)^{k+2} \sum_{i = 0}^{1} (t-\epsilon \zeta_i)\zeta_i^{2g + 1} \xi_i^{2g - 1} \eta(R_k^*(\phi), \widetilde{e}_i)
=(-1)^{k+2} \sum_{i = 0}^{1} (t\Psi_{i,0}+\epsilon \Psi_{i,1})\zeta_i^{2g + 1} \xi_i^{2g} \eta(R_k^*(\phi), \widetilde{e}_i).
\end{aligned}
\end{equation*}
Set $n_1\coloneqq 2$ and define
\begin{equation*}
{B}_{k,1}(2, \phi) \coloneqq \underbrace{\epsilon^{-1}(B_{k,2}(2, \phi)+tB_{k,0}(2, \phi))}_{\in \QQ [\epsilon^{\pm 1}, t]}= (-1)^{k+2}\sum_{i = 0}^{1} \Psi_{i,1}\zeta_i^{2g + 1} \xi_i^{2g} \eta(R_k^*(\phi), \widetilde{e}_i).
\end{equation*}
Then, we get $r=2$ analogous version of equation \eqref{eqn:eta_R_k_alpha_as_a_sum}
\begin{equation}\label{eqn:eta_R_k_alpha_as_a_sum_r_2}
   \eta(R_k^*(\phi), \widetilde{e}_i)
    =  \zeta_i^{-2g - 1} \xi_i^{-2g} \sum_{l = 0}^{1} \Psi_{l, i}^{-1} (-1)^{k+n_l}{B}_{k,l}(2, \phi)\in \QQ[\epsilon^{\pm 1}, t, \zeta_i^{{\pm 1}}, \xi_i^{{\pm 1}}]
\end{equation}
where the ring containment follows from $B_{k, l}(r, \phi) \in \QQ[\epsilon^{\pm 1}, t]$ by Harer stability and $\Psi_{l, i}^{-1}  \in \QQ[t, \zeta_i, \xi_i^{-1}] $ by equation \eqref{eqn:Psi_Inverse}. Hence, we complete the polynomiality claim for $r=2$.

\emph{ $\bullet$ Limit claims:} Next, we will prove the limit statements of the lemma. But, let us first recall some simple facts we have. For $1\leq i\leq r-1$,  by Lemma \ref{lem:series_expansion_for_roots} and equation \eqref{eqn:xi_limit_comparison}, we know
\begin{equation}\label{eqn:zeta_xi_i_limits_nonzero}
\lim_{t \to 0} \zeta_i=\underline{\zeta}_i\neq 0 \quad \text{and}\quad \lim_{t \to 0} \xi_i=\sqrt{(1-r)\epsilon \underline{\zeta}_i}\neq 0.
\end{equation}
Also, by Lemma \ref{lem:limit_of_frobenius_structures}, we know that 
\begin{equation}\label{eqn:Psi_Inverse_limits_t_to_0}
\lim_{t \to 0} \Psi^{-1}_{l, i} \text{ exists and }\lim_{t \to 0} \Psi^{-1}_{0, i}=0
\end{equation}
for any $0\leq i, l \leq r-1$. Hence, by equations \eqref{eqn:eta_R_k_alpha_as_a_sum}, \eqref{eqn:zeta_xi_i_limits_nonzero} and  \eqref{eqn:Psi_Inverse_limits_t_to_0}, we see  that $\lim_{t \to 0}\eta(R_k^*(\phi), \widetilde{e}_i)$ exists for any $\phi \in V$ and for any $1\leq i\leq r-1$. This completes the existence of limit statement of the lemma.

Now, we are ready to complete the proof of the remaining limit claims of the lemma. Firstly, assume $r\geq 3$. The class \eqref{eqn:omega-alpha-smooth} admits a limit
  \begin{equation}\label{eqn:omega-alpha-smooth-limit1}
    \lim_{t \to 0} \Ch_{g, 1+n}^{r,\epsilon}(\phi \otimes v_{a_1} \otimes \cdots \otimes  v_{a_n})\vert_{\cM_{g, 1+n}}\in H^*(\cM_{g, 1+n})[\epsilon].
  \end{equation}
Moreover, if $\phi \in \underline{V}=\langle v_1, \dotsc, v_{r-1}\rangle_{\mathbb{C}}$ this limit is
  \begin{equation}
    \label{eqn:omega-alpha-smooth-limit2}
    \lim_{t \to 0} \Ch_{g, 1+n}^{r,\epsilon}(\phi \otimes v_{a_1} \otimes \cdots \otimes v_{a_n})\vert_{\cM_{g, 1+n}}
    = \Th^{r,\epsilon}_{g, 1+n}(\phi \otimes v_{a_1} \otimes \cdots \otimes v_{a_n})\vert_{\cM_{g, 1+n}}
  \end{equation}
  by definition of the $\epsilon v_0$-deformed theta class $\Th^{r,\epsilon}_{g,1+n}(-)$.
    Thus, taking the coefficients of $\psi_1^k$ in equation \eqref{eqn:omega-alpha-smooth-limit2} for the both sides, we see that 
    \begin{equation*}
    \lim_{t \to 0}B_{k, l}(r,\phi)=\underline{B}_{k, l}(r,\phi)
    \end{equation*}
    where $\underline{B}_{k, l}(r,\phi)$ is the coefficient of $\psi_1^k$ on the right-hand-side of equation \eqref{eqn:omega-alpha-smooth-limit2}. Performing the same analysis in the first part of the proof for $\Th^{r,\epsilon}$, we obtain
    \begin{equation}\label{eqn:Ckb_alpha_limit_t_to_zero}
              \underline{B}_{k, l}(r,\phi)
                 = (-1)^{l + n + k} \sum_{i = 1}^{r - 1}  \underline{\zeta}_i^{2g + 1+ l} \underline{\xi}_i^{2g - 1 } \underline{\eta}( \underline{R}_k^*(\phi), \widetilde{\underline{e}}_i)
        \end{equation}
    for $\phi \in \underline{V}=\langle v_1, \dotsc, v_{r-1}\rangle_{\mathbb{C}}$ and $0\leq l \leq r-1$ and we have
    \begin{equation}
      \underline{\eta}( \underline{R}_k^*(\phi), \widetilde{\underline{e}}_i)
      =\underline{\zeta}_i^{-2g - 1} \underline{\xi}_i^{-2g} \sum_{l = 1}^{r - 1} \underline{\Psi}_{l, i}^{-1} (-1)^{k+n_l}\underline{B}_{k, l}(r, \phi)
    \end{equation}
    for $1\leq i \leq r-1$. Then,  by equations \eqref{eqn:eta_R_k_alpha_as_a_sum}, \eqref{eqn:Psi_Inverse_limits_t_to_0} and \eqref{eqn:Ckb_alpha_limit_t_to_zero}, we observe the following
  \begin{align*}
    \lim_{t \to 0} \eta(R_k^*(\phi), \widetilde{e}_i)
    &= \lim_{t \to 0} (\zeta_i^{-2g - 1} \xi_i^{-2g}) \sum_{l = 0}^{r - 1} \lim_{t \to 0} \Psi_{l, i}^{-1} (-1)^{k+n_l}\lim_{t \to 0} B_{k, l}(r,\phi) \\
    &= \lim_{t \to 0} (\zeta_i^{-2g - 1} \xi_i^{-2g}) \sum_{l = 1}^{r - 1} \lim_{t \to 0} \Psi_{l, i}^{-1} (-1)^{k+n_l}\lim_{t \to 0} B_{k, l}(r,\phi) \\
    &=\underline{\zeta}_i^{-2g - 1} \underline{\xi}_i^{-2g} \sum_{l = 1}^{r - 1} \underline{\Psi}_{l, i}^{-1} (-1)^{k+n_l}\underline{B}_{k, l}(r, \phi)\\
    &= \underline{\eta}(\underline{R}_k^*(\phi), \underline{\widetilde{e}}_i)
  \end{align*}
for any $\phi\in \underline{V}$ and $1\leq i \leq r-1$.

Finally, assume $r=2$ and $\phi\in\underline{V}$. The limits
\eqref{eqn:omega-alpha-smooth-limit1} and
\eqref{eqn:omega-alpha-smooth-limit2} also hold in this case.
Thus, taking the coefficients of $\psi_1^k$ in equation
\eqref{eqn:omega-alpha-smooth-limit2} for $n=2$ and $a_1=a_2=1$,
we see that
\begin{equation*}
\begin{aligned}
\lim_{t\to0}B_{k,2}(2,\phi)
&=\underline{B}_{k,2}(2,\phi)\\
&=(-1)^{k+2}\underline{\zeta}_1^{2g+3}
\underline{\xi}_1^{2g-1}
\underline{\eta}(
\underline{R}_k^*(\phi),\widetilde{\underline{e}}_1).
\end{aligned}
\end{equation*}
Since $B_{k,0}(2,\phi)\in\QQ[\epsilon,t]$, the definition of
$B_{k,1}(2,\phi)$ implies
\begin{equation*}
\begin{aligned}
\lim_{t\to0}B_{k,1}(2,\phi)
&=\epsilon^{-1}\lim_{t\to0}
\bigl(B_{k,2}(2,\phi)+tB_{k,0}(2,\phi)\bigr)\\
&=\epsilon^{-1}\underline{B}_{k,2}(2,\phi).
\end{aligned}
\end{equation*}
Then, by equations \eqref{eqn:eta_R_k_alpha_as_a_sum_r_2}
and \eqref{eqn:Psi_Inverse_limits_t_to_0}, we observe the following
\begin{align*}
\lim_{t\to0}\eta(R_k^*(\phi),\widetilde{e}_1)
&=\underline{\zeta}_1^{-2g-1}\underline{\xi}_1^{-2g}
\underline{\Psi}_{1,1}^{-1}(-1)^{k+2}
\epsilon^{-1}\underline{B}_{k,2}(2,\phi)\\
&=\underline{\Psi}_{1,1}^{-1}
\frac{\underline{\zeta}_1^2}{\epsilon\underline{\xi}_1}
\underline{\eta}(
\underline{R}_k^*(\phi),\widetilde{\underline{e}}_1)\\
&=\underline{\Psi}_{1,1}^{-1}\underline{\Psi}_{1,1}
\underline{\eta}(
\underline{R}_k^*(\phi),\widetilde{\underline{e}}_1)\\
&=\underline{\eta}(
\underline{R}_k^*(\phi),\widetilde{\underline{e}}_1),
\end{align*}
where we used
$\underline{\zeta}_1=-\epsilon$ and
$\underline{\Psi}_{1,1}
=-\underline{\zeta}_1/\underline{\xi}_1$.
This completes the proof.
\end{proof}

The Harer stability technique also yields the following polynomiality and limit result for the $T$-vector.

\begin{lemma}\label{lem:T-limit-nonzero}
For $k\geq 1$,
  \begin{itemize}
      \item we have
  \begin{equation*}
    \eta(T_{k+1}, \widetilde{e}_i) \in \QQ[\zeta_i^{\pm 1}, \xi_i^{{\pm 1}}, \epsilon, t].
  \end{equation*}
  if $i = 0, \dotsc, r-1$,
      \item we have the limiting behavior
  \begin{equation*}
    \lim_{t \to 0} \eta(T_{k+1}, \widetilde{e}_i)
    = \underline{\eta}( \underline{T}_{k+1}, \underline{\widetilde{e}}_i).
  \end{equation*}
  if $i = 1, \dotsc, r-1$.
  \end{itemize}
\end{lemma}
\begin{proof}
We follow a similar strategy as in the proof of Lemma \ref{lem:R-limit-nonzero}.
  For any genus in the Harer stability (Proposition \ref{prop:stability}) range $g > 3k$ and any $l \in \{0,1, \dotsc, r-1\}$, consider the class
  \begin{equation}
    \label{eqn:omega-v1-smooth}
    \Ch_{g, 1}^{r,\epsilon}(v_l)\vert_{\cM_{g, 1}} \in H^*(\cM_{g, 1})[\epsilon, t].
  \end{equation}
By Proposition \ref{prop:stability}, we may consider the coefficient $B_{k, l}(r)$ of $\mathpzc{p}_{1*} \psi_{2}^{k+1}$ of equation \eqref{eqn:omega-v1-smooth}.
  By Lemma \ref{lem:trivial_graph_contribution}, we see that
  \begin{equation*}
    B_{k, l}(r) = \sum_{j = 0}^{r-1} (\zeta_j \xi_j)^{2g} \eta(T_{k+1}, \widetilde{e}_j)  \eta(R^*_0 v_l, \widetilde{e}_j) \in \QQ[\epsilon, t].
  \end{equation*}
  Using that $R_0^*$ is the identity transformation and equation \eqref{eqn:Psi-formula}, this further simplifies to
  \begin{equation}
    \label{eqn:stability-system-T}
    B_{k, l}(r) =  (\zeta_j \xi_j)^{2g} \sum_{j = 0}^{r-1} \Psi_{j,l} \eta(T_{k+1}, \widetilde{e}_j)
  \end{equation}
  We view equation \eqref{eqn:stability-system-T} for ${l} \in \{0, \dotsc, r-1\}$ as a system of $r$ linear equation to solve for $\eta(T_{k+1}, \widetilde{e}_i)$. Thus, for any $0\leq i \leq r-1$, we have
  \begin{equation*}
    \sum_{{l} = 0}^{r-1} \Psi_{{l}, i}^{-1} B_{k, l}(r)
    = (\zeta_j \xi_j)^{2g} \sum_{{l} = 0}^{r-1} \Psi_{{l}, j}^{-1} \sum_{j = 0}^{r-1} \Psi_{j,{l}} \eta(T_{k+1}, \widetilde{e}_j)
    = \zeta_i^{2g} \xi_i^{2g} \eta(T_{k+1}, \widetilde{e}_i).
  \end{equation*}
  Therefore,
  \begin{equation}\label{eqn:T_k_plus_1_norm}
    \eta(T_{k+1}, \widetilde{e}_i)
    = \zeta_i^{-2g} \xi_i^{-2g} \sum_{{l} = 0}^{r-1} \Psi_{{l}, i}^{-1} B_{k, l}(r).
  \end{equation}
  Combining this with equation \eqref{eqn:Psi_Inverse}, we see that
  \begin{equation*}
    \eta(T_{k+1}, \widetilde{e}_i) \in \QQ[\epsilon, t, \zeta_i^{{\pm 1}}, \xi_i^{{\pm 1}}]
  \end{equation*}
which is our first claim.

Now, by definition of $\Th^{r,\epsilon}$, the class $\Ch_{g, {1}}^{r,\epsilon}(v_l)$ admits the limit 
  \begin{equation}
    \label{eqn:omega-v1-smooth-limit}
    \lim_{t \to 0} \Ch_{g, {1}}^{r,\epsilon}(v_l)\vert_{\cM_{g, {1}}}
    = \Th_{g, {1}}^{r,\epsilon}(v_l)\vert_{\cM_{g, {1}}}
  \end{equation}
  after restriction to the smooth locus $\cM_{g, {1}}$ for $1\leq l \leq r-1$. Hence, we have
    \begin{equation}\label{eqn:lim_of_B_k_l_in_T_k_plus_1}
    \lim_{t \to 0}B_{k,l}(r)=\underline{B}_{k,l}(r)\quad \text{for}\quad 1\leq l \leq r-1,
    \end{equation}
    where $\underline{B}_{k,l}(r)$ is the corresponding
coefficient for $\Th^{r,\epsilon}$. Repeating the computation above for
$\Th^{r,\epsilon}$, whose normalized idempotents are indexed by
$1\leq i\leq r-1$, gives
\begin{equation}\label{eqn:underline_B_k_l_T_k_plus_one}
    \underline{B}_{k,l}(r)
    =\sum_{i = 1}^{r-1}\underline{\Psi}_{i,l}\,
     \underline{\zeta}_i^{2g}\underline{\xi}_i^{2g}\,
     \underline{\eta}(\underline{T}_{k+1}, \underline{\widetilde{e}}_i)
    \quad\text{and}\quad
    \underline{\eta}(\underline{T}_{k+1}, \underline{\widetilde{e}}_i)
    = \underline{\zeta}_i^{-2g} \underline{\xi}_i^{-2g}
      \sum_{l = 1}^{r-1} \underline{\Psi}_{l, i}^{-1}\underline{B}_{k, l}(r).
\end{equation}
Recall that by
equation \eqref{eqn:Psi_Inverse_orders_in_t}, we have
$\operatorname{ord}_t(\Psi^{-1}_{0,i})>0$. Hence, we see that
\begin{equation}\label{eqn:ord_psi_inv_B_k_0}
\operatorname{ord}_t (\Psi^{-1}_{0,i}B_{k, 0}(r))>0
\end{equation}
as $B_{k, 0}(r))\in \QQ [\epsilon,t]$. Then, if let $t\to0$ in equation \eqref{eqn:T_k_plus_1_norm} for a fixed $1\leq i\leq r-1$, we get
\begin{equation*}
\lim_{t\to 0} \eta(T_{k+1}, \widetilde{e}_i)
= \lim_{t\to 0}\left(\zeta_i^{-2g} \xi_i^{-2g}\right) \sum_{{l} = 0}^{r-1} \lim_{t\to 0}\left(\Psi_{{l}, i}^{-1} B_{k, l}(r)\right)=\underline{\zeta}_i^{-2g} \underline{\xi}_i^{-2g}
      \sum_{l = 1}^{r-1} \underline{\Psi}_{l, i}^{-1}\underline{B}_{k, l}(r)= \underline{\eta}(\underline{T}_{k+1}, \underline{\widetilde{e}}_i)
\end{equation*}
by Lemma
\ref{lem:limit_of_frobenius_structures} and equation \eqref{eqn:lim_of_B_k_l_in_T_k_plus_1}, completing our second claim.
\end{proof}


\subsection{Bounds on the orders of $t$ for zero colors in $R$-matrix and $T$-vector}

\begin{lemma}
For $0\leq i,j\leq r-1$, $k\geq 0$ and $l\geq 1$, we have
\begin{equation*}
R_{i,j}^k,\,\eta(R_k^*(v_j), \widetilde{e}_i),\, \eta(T_{l+1}, \widetilde{e}_i)  \in\CC[\epsilon^{\pm\frac{1}{2(r-1)}}](\!(t^{\frac{1}{2}})\!).
\end{equation*}
\end{lemma}
\begin{proof}
  For $\eta(R_k^*(v_j), \widetilde{e}_i)$ and $\eta(T_{l+1}, \widetilde{e}_i)$ this
  follows from the polynomiality statements of Lemma \ref{lem:R-limit-nonzero} and
  Lemma \ref{lem:T-limit-nonzero} together with
  $\zeta_i^{\pm 1},\xi_i^{\pm 1}\in \CC[\epsilon^{\pm\frac{1}{2(r-1)}}](\!(t^{\frac{1}{2}})\!)$
  from Lemma \ref{lem:xi_i_and_zeta_i_and_inverses_ring_containment}. For
  $R_{i,j}^k$, expand the first argument in the flat basis using equation \eqref{eqn:normalized_idempotents_in_vk} and noting that
  $\widetilde{e}_i=\sum_{m=0}^{r-1}\Psi^{-1}_{m,i}v_m$, we have
  \begin{equation*}
R_{i,j}^k=\eta(R_k^*\widetilde{e}_i,\widetilde{e}_j)
    =\sum_{m=0}^{r-1}\Psi^{-1}_{m,i}\,\eta(R_k^*(v_m),\widetilde{e}_j),
    \end{equation*}
  and both factors lie in $\CC[\epsilon^{\pm\frac{1}{2(r-1)}}](\!(t^{\frac{1}{2}})\!)$
  by the previous case and equation \eqref{eqn:Psi_Inverse_orders_in_t}.
\end{proof}

The consequences of Harer stability obtained in Lemma \ref{lem:R-limit-nonzero} and Lemma \ref{lem:T-limit-nonzero} yield certain lower bounds on the orders of $t$ for these expressions.

\begin{corollary}\label{cor:t_orders_non_zero_colors}
For $1\leq i\leq r-1$, $0\leq j\leq r-1$, $k\geq 0$ and $l\geq 1$, we have
\begin{equation*}
\operatorname{ord}_t(\eta(R_k^*(v_j), \widetilde{e}_i))\geq 0\quad \text{and}\quad \operatorname{ord}_t(\eta(T_{l+1}, \widetilde{e}_i))\geq 0.
\end{equation*}
\end{corollary}

\begin{proof}
The proof follows from the fact that these expressions have limits as $t\to 0$ by Lemma \ref{lem:R-limit-nonzero} and Lemma \ref{lem:T-limit-nonzero}. Hence, they have nonnegative orders in $t$ in $\mathbb{C}[\epsilon^{\pm\frac{1}{2(r-1)}}](\!(t^{\frac{1}{2}})\!)$.
\end{proof}

However, Harer stability approach falls short for obtaining bounds for the orders of $t$ for 
\begin{equation*}
\operatorname{ord}_t(\eta(R_k^*(v_j), \widetilde{e}_0)) \quad \text{and}\quad \operatorname{ord}_t(\eta(T_{l+1},\widetilde{e}_0))
\end{equation*}
as limiting procedure for $\Ch^{r,\epsilon}$ cannot extract information for zero colorings in the graph contributions. To get information for these, we will use a homogeneity property of $\Ch^{r,\epsilon}$ together with the flatness equation governing its $R$-matrix.

The $\epsilon v_0$-shifted Chiodo CohFT is homogeneous 
in the sense of \cite{paper1}:
\begin{equation*}
  \Ch_{g, n}^{r,\epsilon}(v_{a_1}\otimes \cdots \otimes v_{a_n})
\end{equation*}
is a cohomology class homogeneous of (complex) degree
\begin{equation*}
  \frac{(r + 2)(g - 1) + n + \sum_{i = 1}^n a_i}r,
\end{equation*}
where we set\footnote{These degrees are uniquely determined by the requirement that the class 
$\Ch_{g,n}^{r,\epsilon}(v_{a_1},\ldots,v_{a_n})$ be homogeneous of complex 
cohomological degree $$\frac{(r+2)(g-1)+n+\sum_{i=1}^n a_i}{r}.$$}
\begin{equation*}
  \deg(t) = 1,  \quad \deg(\epsilon) = \frac{r-1}r,  \quad \deg (v_i)=\frac{i+1}{r}
\end{equation*}
for any $i=0,\ldots,r-1$. 

\begin{lemma}\label{lem:homogeneity_of_R_matrix_entries}
For $0\leq i,j\leq r-1$, $k\geq 0$ and $l\geq 1$, the $R$-matrix entry $R_{i,j}^k$ is homogeneous of degree $-k$.
\end{lemma}

\begin{proof}
Let $\phi \in V$ have homogeneous degree $\deg (\phi)$ with the above degree assignments of $t$, $\epsilon$ and $v_i$. Then, by \cite
{paper1}, the vector  $R_k^*\phi$ is homogeneous of degree
  \begin{equation*}
    \deg(R_k^*\phi) = \deg(\phi) - k.
  \end{equation*}
Recall that we have $R_{i,j}^k=\eta (R_k^* \widetilde{e}_i, \widetilde{e}_j)$. Also, from the expression \eqref{eqn:Chiodo_pairing_t}, it is clear that
\begin{equation*}
\deg (\eta (\phi_1,\phi_2))=\deg (\phi_1)+\deg (\phi_2)-\frac{r+2}{r}
\end{equation*}
for any two homogeneous vectors $\phi_1,\phi_2\in V$, once they are expanded in the flat basis $\{v_i\}_{i=0}^{r-1}$. Hence, we have
\begin{equation}\label{eqn:degree_of_Rijk_intermediate}
\deg (R_{ij}^k)=\deg (R_k^* \widetilde{e}_i)+\deg (\widetilde{e}_j)-\frac{r+2}{r}
=\deg (\widetilde{e}_i)+\deg (\widetilde{e}_j)-k-1-\frac{2}{r}
\end{equation}
From the proof of Lemma \ref{lem:series_expansion_for_roots}, the definition of $\zeta_i$, and the expression \eqref{eqn:normalized_idempotents_in_vk} for the normalized idempotent vectors we see that
\begin{equation*}
\deg (\zeta_i)=\frac{1}{r},\quad \deg (\xi_i)=\frac{1}{2}\quad \text{and}\quad
\deg (\widetilde{e}_i)=\frac{1}{2}+\frac{1}{r}
\end{equation*}
for all $0 \leq i \leq r-1$.
As a result, by equation \eqref{eqn:degree_of_Rijk_intermediate}, we conclude that $\deg (R_{ij}^k)=-k$ as claimed.
\end{proof}

\begin{lemma}\label{lem:order_of_df_implies_order_of_f}
  Let
  \begin{equation*}
    f \in \CC[\epsilon^{\pm\frac{1}{2(r-1)}}](\!(t^{\frac{1}{2}})\!)
  \end{equation*}
  be homogeneous of degree $-k$, and assume that
  \begin{equation*}
    t^k \frac{df}{d\epsilon} \in \CC[\epsilon^{\pm\frac{1}{2(r-1)}}][\![t^\frac{1}{2}]\!].
  \end{equation*}
  Then, we have
  \begin{equation*}
    t^k f \in \CC[\epsilon^{\pm\frac{1}{2(r-1)}}][\![t^{\frac{1}{2}}]\!].
  \end{equation*}
\end{lemma}
\begin{proof}
  We may assume without loss of generality that $k = 0$.
  Note that by the assumptions, we know that the $t^{-l}$ coefficient of
  \begin{equation*}
    \frac{df}{d\epsilon}
  \end{equation*}
  vanishes for any $l > 0$.
  Therefore, the coefficient of $t^{-l}$ in $f$ is a constant.
  By assumption, $f$ is homogeneous of degree $0$.
  Hence, the coefficient of $t^{-l}$ in $f$ vanishes for any $l > 0$.
\end{proof}

\begin{lemma} \label{lem:R_ij_k_Ring_containment_orders_in_t}
For $0\leq i\leq r-1$ and $k\geq 0$, the order of $R_{i,0}^k$ in $t$ has the following lower bound
\begin{equation*}
\operatorname{ord}_t\left(R_{i,0}^k\right)\geq -k+1-\delta_{i,0}.
\end{equation*}
\end{lemma}

\begin{proof}
  We will prove the statement of the lemma by induction on $k$. The base case is trivial since $R_{i,0}^0=\delta_{i,0}$. Now we pass to the inductive step.

  Firstly, we consider the case $i\neq 0$. By Lemma \ref{lem:xi_i_and_zeta_i_and_inverses_ring_containment} and  Lemma \ref{lem:t_orders_of_Aij}, we know that
\begin{equation}\label{eqn:zetai_minus_zetazero_and_Aij_orders}
\operatorname{ord}_t\left(\frac{1}{\zeta_i-\zeta_j}\right)=0\quad \text{and}\quad \operatorname{ord}_t(Y_{i,j})\geq 0.
\end{equation}
Recall that by the flatness equations \eqref{eqn:flatness_in_normalized_idempotent_explicit} in Corollary \ref{cor:flatness_equation_in_ij_form},  we have
\begin{equation*}
R_{i,0}^k=\frac{1}{\zeta_i-\zeta_0}\sum_{l\neq i}Y_{i,l}R_{l,0}^{k-1}+\frac{1}{\zeta_i-\zeta_0} \frac{\partial R_{i,0}^{k-1}}{\partial \epsilon}\quad \text{for $i\neq 0$}.
\end{equation*}
Then, equation \eqref{eqn:zetai_minus_zetazero_and_Aij_orders} implies the following:
\begin{equation}\label{eqn:orders_of_t_R_ij_different2}
\begin{aligned}
\operatorname{ord}_t\left(R_{i,0}^k\right)
\geq &
\min \left\{\operatorname{ord}_t\left(\frac{1}{\zeta_i-\zeta_0}\sum_{l\neq i}Y_{i,l}R_{l,0}^{k-1}\right),\operatorname{ord}_t\left(\frac{1}{\zeta_i-\zeta_0} \frac{\partial R_{i,0}^{k-1}}{\partial \epsilon}\right)\right\}\\
\geq &
\min \left\{\operatorname{ord}_t\left(\sum_{l\neq i}R_{l,0}^{k-1}\right),\operatorname{ord}_t\left( \frac{\partial R_{i,0}^{k-1}}{\partial \epsilon}\right)\right\}\\
\geq &
\min_{0\leq l\leq r-1} \{\operatorname{ord}_t\left(R_{l,0}^{k-1}\right)\}   
\end{aligned}
\end{equation}
as taking partial derivative with respect to $\epsilon$ does not lower the orders in $t$. Then, induction hypothesis and equation \eqref{eqn:orders_of_t_R_ij_different2} give us
\begin{equation}\label{eqn:orders_of_t_R_ij_different}
\operatorname{ord}_t\left(R_{i,0}^k\right)  \geq \min_{0\leq l\leq r-1}\{\operatorname{ord}_t\left(R_{l,0}^{k-1}\right)\}\geq  -(k-1)=-k+1. 
\end{equation}
Hence, we complete the inductive step for $i\neq 0$ case.

In order to complete the inductive step, now consider the case $i=0$. By the diagonal terms of the flatness equation \eqref{eqn:flatness_in_normalized_idempotent_explicit} in Corollary \ref{cor:flatness_equation_in_ij_form}, we have
\begin{equation*}
\frac{ \partial R_{0,0}^k}{\partial \epsilon}=-\sum_{l=1}^{r-1}Y_{0,l}R_{l,0}^{k}.
\end{equation*}
Then, by Lemma \ref{lem:t_orders_of_Aij} and equation \eqref{eqn:orders_of_t_R_ij_different}, we observe that
\begin{equation*}
 \operatorname{ord}_t \left(t^k\frac{ \partial R_{0,0}^k}{\partial \epsilon}\right)=\operatorname{ord}_t\left(-t^k\sum_{l=1}^{r-1}Y_{0,l}R_{l,0}^{k}\right)\geq k-k+1=1\geq 0.
\end{equation*}
This implies that
\begin{equation*}
t^k \frac{ \partial R_{0,0}^k}{\partial \epsilon} \in \CC[\epsilon^{\pm\frac{1}{2(r-1)}}][\![t^\frac{1}{2}]\!].
\end{equation*}
Then, by Lemma \ref{lem:homogeneity_of_R_matrix_entries}  and Lemma \ref{lem:order_of_df_implies_order_of_f}, we see that
\begin{equation*}
t^kR_{0,0}^k\in \CC[\epsilon^{\pm\frac{1}{2(r-1)}}][\![t^\frac{1}{2}]\!]
\end{equation*}
and we conclude
\begin{equation*}
\operatorname{ord}_t\left(R_{0,0}^k\right)\geq -k
\end{equation*}
completing the proof.
\end{proof}

Now we will obtain some lower bounds on the orders of $t$ in $\eta(R^*_k v_j, \widetilde{e}_i)$ and $\eta(T_k, \widetilde{e}_i)$.

\begin{corollary}
  \label{cor:R-matrix-bound}
  For all $0\leq j \leq r-1$, we have
\begin{equation*}
 \operatorname{ord}_t \eta(R^*_k v_j, \widetilde{e}_0)\geq -\frac{1}{2}-k+\min \left\{j,\frac{3}{2}\right\}.
\end{equation*}
\end{corollary}
\begin{proof}
We obtain
\begin{equation*}
    \eta(R^*_k v_j, \widetilde{e}_0)
    =\eta(v_j, R_k \widetilde{e}_0)
    =\eta\left(v_j,  \sum_{l=0}^{r-1}R_{l,0}^k\widetilde{e}_l\right)
    =\sum_{l=0}^{r-1} R_{l,0}^k \eta(v_j,\widetilde{e}_l)
    =\sum_{l=0}^{r-1}\Psi_{l, j}R_{l,0}^k
\end{equation*}
By equation \eqref{eqn:Psi_orders_in_t} and Lemma \ref{lem:R_ij_k_Ring_containment_orders_in_t}, we have
\begin{equation*}
\operatorname{ord}_t(\Psi_{l, j}R_{l,0}^k)=j\delta_{l,0}-\frac{1}{2}\delta_{l,0}+\operatorname{ord}_t(R_{l,0}^k)\geq j\cdot \delta_{l,0}-\frac{1}{2}\delta_{l,0} -k+1-\delta_{l,0}.
\end{equation*}
So, we have
\begin{equation*}
\operatorname{ord}_t(\Psi_{l, j}R_{l,0}^k)\geq \left(j-\frac{3}{2}\right)\delta_{l,0}-k+1.
\end{equation*}
Hence, we conclude
\begin{equation*}
 \operatorname{ord}_t \eta(R^*_k v_j, \widetilde{e}_0)\geq \min_{l}\{\operatorname{ord}_t(\Psi_{l, j}R_{l,0}^k)\}\geq
-\frac{1}{2}-k+\min \left\{j,\frac{3}{2}\right\}.
\end{equation*}
\end{proof}

\begin{corollary} For all $k\geq 2$, we have
\label{cor:T-bound}
\begin{equation*}
\operatorname{ord}_t \eta(T_k, \widetilde{e}_0) \geq -k-\frac{1}{2}.
\end{equation*}
\end{corollary}
\begin{proof}
Using Lemma \ref{lem:translation_vector_of_chiodo}, we compute
\begin{equation*}
\begin{aligned}
\eta(T_{k+1},\widetilde{e}_0)
&=\sum_{m=0}^{k} (-1)^{m+1}t^{-k+m}\eta(R^*_m\mathbf{1},\widetilde{e}_0)\\
&=\sum_{m=0}^{k}\sum_{j=0}^{r-1} (-1)^{m+1}t^{-k+m}\zeta_j^{-1}\xi_j^{-1}\eta(\widetilde{e}_j,R_m\widetilde{e}_0)\\
&=\sum_{m=0}^{k}\sum_{j=0}^{r-1} (-1)^{m+1}t^{-k+m}\zeta_j^{-1}\xi_j^{-1}R_{j,0}^m
\end{aligned}
\end{equation*}
for $k\geq 1$. Looking at the orders of $t$ in the summand, we get
\begin{equation*}
\begin{aligned}
\operatorname{ord}_t(t^{-k+m}\zeta_j^{-1}\xi_j^{-1}R_{j,0}^m)
&=-k+m-\operatorname{ord}_t(\zeta_j)-\operatorname{ord}_t(\xi_j)+\operatorname{ord}_t(R_{j,0}^m)\\
&\geq -k+m-\delta_{j,0}-\frac{1}{2}\delta_{j,0}-m+1-\delta_{j,0}\\
&=-k+1-\frac{5}{2}\delta_{j,0}
\end{aligned}
\end{equation*}
by Lemma \ref{lem:xi_i_and_zeta_i_and_inverses_ring_containment} and Lemma \ref{lem:R_ij_k_Ring_containment_orders_in_t}. As a result, we obtain
\begin{equation*}
\operatorname{ord}_t \eta(T_k, \widetilde{e}_0)
\geq \min_{0\leq j\leq r-1}
\left\{-(k-1)+1-\frac{5}{2}\delta_{j,0}\right\}
=-k-\frac{1}{2}
\end{equation*}
completing the proof.
\end{proof}
\begin{remark}
Using homogeneity and flatness equations, we can similarly obtain lower bounds for 
\begin{equation*}
\operatorname{ord}_t \eta(R^*_k v_j, \widetilde{e}_i)\quad \text{and} \quad \operatorname{ord}_t \eta(T_k, \widetilde{e}_i)
\end{equation*}
for $1\leq i \leq r-1$. However, these bounds are worse estimates compared to those of in Lemma \ref{lem:R-limit-nonzero} and Lemma \ref{lem:T-limit-nonzero}.
\end{remark}

If we interpret the results we obtained for the orders of $t$ in $\operatorname{ord}_{t}(\eta(T_{k}, \widetilde{e}_i))$ and $\operatorname{ord}_t \eta(R^*_k v_j, \widetilde{e}_i)$ so far, we can summarize them as the following corollary.

\begin{corollary}\label{cor:R_and_T_orders_in_t}
For any $0\leq i\leq r-1$ and $k\geq 0$, we have
\begin{equation*}
   \operatorname{ord}_t \eta(R^*_k v_j, \widetilde{e}_i) \geq 
  \left(-\frac{1}{2}-k+\min \left\{j,\frac{3}{2}\right\}\right)\delta_{i,0}\quad \text{and}\quad \operatorname{ord}_{t}\eta(T_{k}, \widetilde{e}_i)\geq \left(-k-\frac{1}{2}\right)\delta_{i,0}.
\end{equation*}
\end{corollary}
\begin{proof}
It immediately follows from Corollary \ref{cor:t_orders_non_zero_colors}, Corollary \ref{cor:R-matrix-bound} and Corollary \ref{cor:T-bound}.
\end{proof}

\section{Limit of Givental--Teleman Reconstruction for $\epsilon v_0$-shifted Chiodo CohFT}\label{sec:Limit_of_Givental_Teleman_Formula}

\subsection{Limits of the graph contributions} The aim of this section is to prove that for a stable graph ${\Gamma}$ and insertions in the range $1\leq s_j \leq r-1$, i.e. $v_{s_j} \in \underline{V}$, we have
\begin{equation}\label{eqn:limiting_aim_of_the_subsection}
\lim_{t\to 0}\mathrm{Cont}^{\Ch^{r, \epsilon}}_{{\Gamma}}(v_{s_1} \otimes \cdots \otimes v_{s_n})=\mathrm{Cont}^{\Th^{r, \epsilon}}_{{\Gamma}}(v_{s_1} \otimes \cdots \otimes v_{s_n})
\end{equation}
for the graph contributions in Proposition \ref{prop:Graph_Sum_Formula_Shifted_Chiodo} and Proposition \ref{prop:Graph_Sum_Formula_Shifted_Theta}. In other words, the limit of the contribution of ${\Gamma}$ for the $\epsilon v_0$-shifted Chiodo CohFT is the same as the contribution of ${\Gamma}$ of the $\epsilon v_0$-deformed theta CohFT when the insertions are elements of $\underline{V}$.

\begin{remark}
  Since our aim is to prove the limiting behavior \eqref{eqn:limiting_aim_of_the_subsection}, we will work with insertions in $\underline{V}$ throughout this section.
  Hence, every statement involving leg contributions in this section will be proven under the assumption of $1\leq s_j \leq r-1$; whereas, other statements still hold over $V$.
\end{remark}

Let $\widetilde{\Gamma}_c$ be a decorated stable graph. Then, we will split its vertex, leg and edge sets according to the coloring as follows:
\begin{itemize}
    \item  The vertex set is split  as $\mathrm{V}_{\Gamma}=\mathrm{V}_{\widetilde{\Gamma}_c}^{0} \sqcup \mathrm{V}_{\widetilde{\Gamma}_c}^{+}$ where
    \begin{equation*}
    \begin{aligned}
    \mathrm{V}_{\widetilde{\Gamma}_c}^{0}& \coloneqq \text{the set of vertices of $\widetilde{\Gamma}_c$ with zero coloring},\\
    \mathrm{V}_{\widetilde{\Gamma}_c}^{+}& \coloneqq \text{the set of vertices of $\widetilde{\Gamma}_c$ with non-zero coloring}.
    \end{aligned}
    \end{equation*}
    \item The edge set is split as $\mathrm{E}_{\Gamma}=\mathrm{E}_{\widetilde{\Gamma}_c}^{\{0,0\}}
    \sqcup \mathrm{E}_{\widetilde{\Gamma}_c}^{\{0,+\}} \sqcup \mathrm{E}_{\widetilde{\Gamma}_c}^{\{+,+\}}$ where
    \begin{equation*}
    \begin{aligned}
    \mathrm{E}_{\widetilde{\Gamma}_c}^{\{0,0\}} &\coloneqq \text{the set of edges of $\widetilde{\Gamma}_c$ with both half-edges attached to vertices of zero coloring,}\\
    \mathrm{E}_{\widetilde{\Gamma}_c}^{\{0,+\}} &\coloneqq \text{the set of edges of $\widetilde{\Gamma}_c$ with only one half-edge attached to a vertex of zero coloring,}\\ \mathrm{E}_{\widetilde{\Gamma}_c}^{\{+,+\}} &\coloneqq \text{the set of edges of $\widetilde{\Gamma}_c$ with both half-edges attached to vertices of non-zero coloring.}
    \end{aligned}
    \end{equation*}
    \item The leg set is split as $\mathrm{L}_{\Gamma} = \mathrm{L}_{\widetilde{\Gamma}_c}^{0} \sqcup \mathrm{L}_{\widetilde{\Gamma}_c}^{+}$ where
    \begin{equation*}
    \begin{aligned}
    \mathrm{L}_{\widetilde{\Gamma}_c}^{0} & \coloneqq \text{the set of legs of $\widetilde{\Gamma}_c$ attached to a vertex of zero coloring,}\\
    \mathrm{L}_{\widetilde{\Gamma}_c}^{+} & \coloneqq \text{the set of legs of $\widetilde{\Gamma}_c$ attached to a vertex of non-zero coloring.} 
    \end{aligned}
    \end{equation*}
\end{itemize}

\begin{lemma}\label{lem:limit_of_non_zero_colorings}
Let $\widetilde{\Gamma}_c$ be a decorated stable graph and fix a flag assignment $A\in \mathbb{Z}_{\geq 0}^{\vert \mathrm{H}_{\Gamma}\vert}$. Then, we have the following limiting identifications:
\begin{itemize}
    \item[(i)] For a vertex $\mathfrak{v}\in \mathrm{V}_{\widetilde{\Gamma}_c}^{+}$ with non-zero coloring ($c(\mathfrak{v})\neq 0$), we have $$\lim_{t\to 0}\mathrm{Cont}_{\widetilde{\Gamma}_c,\mathfrak{v}}^{A,\Ch^{r, \epsilon}}=\mathrm{Cont}_{\widetilde{\Gamma}_c,\mathfrak{v}}^{A,\Th^{r, \epsilon}}$$.
    \item[(ii)] For a leg $\mathfrak{l}\in \mathrm{L}_{\widetilde{\Gamma}_c}^{+}$ attached to a vertex with non-zero coloring ($c(\nu (\mathfrak{l}))\neq 0$), we have $$\lim_{t\to 0}\mathrm{Cont}_{\widetilde{\Gamma}_c,\mathfrak{l}}^{A,\Ch^{r, \epsilon}}=\mathrm{Cont}_{\widetilde{\Gamma}_c,\mathfrak{l}}^{A,\Th^{r, \epsilon}}.$$
    \item[(iii)] For an edge $\mathfrak{e}\in \mathrm{E}_{\widetilde{\Gamma}_c}^{\{+,+\}}$ that connects vertices with non-zero coloring ($c(\mathfrak{v}_{\mathfrak{e}_1})\neq 0$ and  $c(\mathfrak{v}_{\mathfrak{e}_2})\neq 0$), we have $$\lim_{t\to 0}\mathrm{Cont}_{\widetilde{\Gamma}_c,\mathfrak{e}}^{A,\Ch^{r, \epsilon}}=\mathrm{Cont}_{\widetilde{\Gamma}_c,\mathfrak{e}}^{A,\Th^{r, \epsilon}}.$$
\end{itemize}
\end{lemma}

\begin{proof} We will compare these expressions which are given in Proposition \ref{prop:Graph_Sum_Formula_Shifted_Chiodo} and Proposition \ref{prop:Graph_Sum_Formula_Shifted_Theta}. Part (i) follows from Lemma \ref{lem:T-limit-nonzero} and the fact that 
\begin{equation*}
\lim_{t\to 0}\zeta_i=\underline{\zeta}_i\quad \text{and}\quad \lim_{t\to 0}\xi_i=\underline{\xi}_i
\end{equation*}
for $1\leq i \leq r-1$. Part (ii) follows from Lemma \ref{lem:R-limit-nonzero}.
The proof of Part (iii) follows from Lemma \ref{lem:R-limit-nonzero} and $\eta^{i,\mathrm{Inv}(i)}=t^{\delta_{i,0}}$.
To be more precise, the $i=0$ summand vanishes due to the factor $\eta^{i,\mathrm{Inv}(i)}$ and the remaining parts of the expression converge to $\mathrm{Cont}_{\widetilde{\Gamma}_c,\mathfrak{e}}^{A,\Th^{r, \epsilon}}$.
\end{proof}

\begin{lemma}\label{lem:One_zero_coloring_half_edge}
Let $\widetilde{\Gamma}_c$ be a decorated stable graph and fix a flag assignment $A\in \mathbb{Z}_{\geq 0}^{\vert \mathrm{H}_{\Gamma}\vert}$. For an edge $\mathfrak{e}\in \mathrm{E}_{\widetilde{\Gamma}_c}^{\{0,+\}}$ connecting two vertices $\mathfrak{v}_{\mathfrak{e}_1}$  and $\mathfrak{v}_{\mathfrak{e}_2}$ where one vertex has zero coloring ($c(\mathfrak{v}_{\mathfrak{e}_1})=0$) and the other has non-zero coloring ($c(\mathfrak{v}_{\mathfrak{e}_2})\neq 0$), we have
\begin{equation*}
\begin{aligned}
\operatorname{ord}_t\left(\mathrm{Cont}_{\widetilde{\Gamma}_c,\mathfrak{e}}^{A,\Ch^{r, \epsilon}}\right)\geq -\frac{1}{2}-b_{\mathfrak{e}_1}+1.
\end{aligned}
\end{equation*}
\end{lemma}

\begin{proof}
  By Proposition \ref{prop:Graph_Sum_Formula_Shifted_Chiodo} and Remark \ref{rem:alternate_version_of_edge_contribution}, we have the liberty to choose the expression for the edge-contribution.
  Without loss of generality, the vertex with zero coloring is $\mathfrak{v}_{\mathfrak{e}_1}$ and we choose the expression for edge contribution $\mathrm{Cont}_{\widetilde{\Gamma}_c,\mathfrak{e}}^{A,\Ch^{r, \epsilon}}$ given in Remark \ref{rem:alternate_version_of_edge_contribution}.
  Then, by Corollary \ref{cor:R_and_T_orders_in_t} and Propostion \ref{prop:Graph_Sum_Formula_Shifted_Chiodo}, we have
\begin{equation*}
\begin{aligned}
&\operatorname{ord}_t\left(\mathrm{Cont}_{\widetilde{\Gamma}_c,\mathfrak{e}}^{A,\Ch^{r, \epsilon}}\right)\\
&\geq \min \left\{ \operatorname{ord}_t\eta(R^*_{b_{\mathfrak{e}_1}-m}v_i,\widetilde{e}_{c(\mathfrak{v}_{\mathfrak{e}_1})}) +\operatorname{ord}_t\eta^{i,\mathrm{Inv}(i)}+\operatorname{ord}_t \eta(R^*_{b_{\mathfrak{e}_2}+m+1}v_{\mathrm{Inv}(i)},\widetilde{e}_{c(\mathfrak{v}_{\mathfrak{e}_2})})\right\}_{0\leq m \leq b_{\mathfrak{e}_1}, 0\leq i\leq r-1}\\
&\geq \min \left\{ \operatorname{ord}_t\eta(R^*_{b_{\mathfrak{e}_1}-m}v_i,\widetilde{e}_{c(\mathfrak{v}_{\mathfrak{e}_1})}) +\operatorname{ord}_t\eta^{i,\mathrm{Inv}(i)}\right\}_{0\leq m \leq b_{\mathfrak{e}_1}, 0\leq i\leq r-1}
\end{aligned}
\end{equation*}
since the third factor has non-negative $t$-order. By Corollary \ref{cor:R_and_T_orders_in_t} and $\eta^{i,\mathrm{Inv}(i)}=t^{\delta_{i,0}}$, we have
\begin{equation*}
\begin{aligned}
\operatorname{ord}_t\left(\mathrm{Cont}_{\widetilde{\Gamma}_c,\mathfrak{e}}^{A,\Ch^{r, \epsilon}}\right)
&\geq \min \left\{ \left(-\frac{1}{2}-(b_{\mathfrak{e}_1}-m)+\min \left\{i,\frac{3}{2}\right\}\right)\delta_{c(\mathfrak{v}_{\mathfrak{e}_1}),0}+\delta_{i,0}\right\}_{0\leq m \leq b_{\mathfrak{e}_1}, 0\leq i\leq r-1}.
\end{aligned}
\end{equation*}

Analyzing the cases $i=0$ and $1\leq i \leq r-1$ separately, we obtain
\begin{equation*}
\operatorname{ord}_t\left(\mathrm{Cont}_{\widetilde{\Gamma}_c,\mathfrak{e}}^{A,\Ch^{r, \epsilon}}\right) \geq -\frac{1}{2}-b_{\mathfrak{e}_1}+1
\end{equation*}
as we have  $c(\mathfrak{v}_{\mathfrak{e}_1})=0$.
\end{proof}

\begin{lemma} \label{lem:both_zero_coloring_order_of_t} Let $\widetilde{\Gamma}_c$ be a decorated stable graph and fix a flag assignment $A\in \mathbb{Z}_{\geq 0}^{\vert \mathrm{H}_{\Gamma}\vert}$. For an edge $\mathfrak{e}\in \mathrm{E}_{\widetilde{\Gamma}_c}^{\{0,0\}}$ connecting  two vertices of zero coloring, we have
\begin{equation*}
\begin{aligned}
\operatorname{ord}_t\left(\mathrm{Cont}_{\widetilde{\Gamma}_c,\mathfrak{e}}^{A,\Ch^{r, \epsilon}}\right)\geq -\frac{1}{2}-b_{\mathfrak{e}_1}-\frac{1}{2}-b_{\mathfrak{e}_2}.
\end{aligned}
\end{equation*}
\end{lemma}

\begin{proof}
Without loss of generality, we choose the expression given in Remark \ref{rem:alternate_version_of_edge_contribution} for the edge contribution $\mathrm{Cont}_{\widetilde{\Gamma}_c,\mathfrak{e}}^{A,\Ch^{r, \epsilon}}$. Then, by Corollary \ref{cor:R_and_T_orders_in_t} and $\eta^{i,\mathrm{Inv}(i)}=t^{\delta_{i,0}}$, we have the following 
\begin{equation*}
\begin{aligned}
&\operatorname{ord}_t\left(\mathrm{Cont}_{\widetilde{\Gamma}_c,\mathfrak{e}}^{A,\Ch^{r, \epsilon}}\right)\\
&\geq \min \left\{ \operatorname{ord}_t\eta(R^*_{b_{\mathfrak{e}_1}-m}v_i,\widetilde{e}_{c(\mathfrak{v}_{\mathfrak{e}_1})}) +\operatorname{ord}_t\eta^{i,\mathrm{Inv}(i)}+\operatorname{ord}_t \eta(R^*_{b_{\mathfrak{e}_2}+m+1}v_{\mathrm{Inv}(i)},\widetilde{e}_{c(\mathfrak{v}_{\mathfrak{e}_2})})\right\}_{0\leq m \leq b_{\mathfrak{e}_1}, 0\leq i\leq r-1}\\
&\geq \min \left\{ \left(-\frac{1}{2}-(b_{\mathfrak{e}_1}-m)+\min \left\{i,\frac{3}{2}\right\}\right)\delta_{c(\mathfrak{v}_{\mathfrak{e}_1}),0}+\delta_{i,0}+\left(-\frac{1}{2}-(b_{\mathfrak{e}_2}+m+1)+\min \left\{\mathrm{Inv}(i),\frac{3}{2}\right\}\right)\delta_{c(\mathfrak{v}_{\mathfrak{e}_2}),0}\right\}_{0\leq m \leq b_{\mathfrak{e}_1}, 0\leq i\leq r-1}.
\end{aligned}
\end{equation*}
Noticing that $c(\mathfrak{v}_{\mathfrak{e}_1})=c(\mathfrak{v}_{\mathfrak{e}_2})=0$ and canceling $m$ terms, we obtain
\begin{equation*}
\operatorname{ord}_t\left(\mathrm{Cont}_{\widetilde{\Gamma}_c,\mathfrak{e}}^{A,\Ch^{r, \epsilon}}\right)\geq \min \left\{ -\frac{1}{2}-b_{\mathfrak{e}_1}+\min \left\{i,\frac{3}{2}\right\}+\delta_{i,0}-\frac{1}{2}-b_{\mathfrak{e}_2}-1+\min \left\{\mathrm{Inv}(i),\frac{3}{2}\right\}\right\}_{0\leq i\leq r-1}.
\end{equation*}
Analyzing the cases $i=0$ and $1\leq i \leq r-1$ separately, we obtain
\begin{equation*}
\operatorname{ord}_t\left(\mathrm{Cont}_{\widetilde{\Gamma}_c,\mathfrak{e}}^{A,\Ch^{r, \epsilon}}\right) \geq -\frac{1}{2}-b_{\mathfrak{e}_1}-\frac{1}{2}-b_{\mathfrak{e}_2}.
\end{equation*}
\end{proof}

\begin{lemma}\label{lem:order_t_bound_zero_leg}
Let $\widetilde{\Gamma}_c$ be a decorated stable graph and fix a flag assignment $A\in \mathbb{Z}_{\geq 0}^{\vert \mathrm{H}_{\Gamma}\vert}$. For a leg $\mathfrak{l}\in \mathrm{L}_{\widetilde{\Gamma}_c}^{0}$ attached to a vertex with zero coloring, we have 
\begin{equation}
\operatorname{ord}_t\left(\mathrm{Cont}_{\widetilde{\Gamma}_c,\mathfrak{l}}^{A,\Ch^{r, \epsilon}}\right)\geq -\frac{1}{2}-{a_{\ell(\mathfrak{l})}}+\min \left\{{s_{\ell(\mathfrak{l})}},\frac{3}{2}\right\}.
\end{equation}
\end{lemma}

\begin{proof}
By Corollary \ref{cor:R_and_T_orders_in_t} and Proposition \ref{prop:Graph_Sum_Formula_Shifted_Chiodo}, we have the following
\begin{equation*}
\operatorname{ord}_t\left(\mathrm{Cont}^{\mathrm{A},\Ch^{r, \epsilon}}_{\widetilde{\Gamma}_c,\mathfrak{l}}\right)=\operatorname{ord}_t\left(\eta(R^*_{a_{\ell(\mathfrak{l})}}v_{s_{\ell(\mathfrak{l})}},\widetilde{e}_{c(\nu (\mathfrak{l}))})\right)\geq -\frac{1}{2}-{a_{\ell(\mathfrak{l})}}+\min \left\{{s_{\ell(\mathfrak{l})}},\frac{3}{2}\right\}
\end{equation*}
as  $c(\nu (\mathfrak{l}))=0$.
\end{proof}

By Proposition \ref{prop:Graph_Sum_Formula_Shifted_Chiodo}, we have
\begin{equation*}
\mathrm{Cont}^{\Ch^{r, \epsilon}}_{\widetilde{\Gamma}_c}(v_{s_1} \otimes \cdots \otimes v_{s_n})=\frac{1}{\vert \mathrm{Aut}({\Gamma})\vert}{\mathpzc{gl}_{{\Gamma}}}_{\star}\left(\sum_{\mathrm{A}\in \mathbb{Z}_{\geq 0}^{\vert \mathrm{H}_{\Gamma}\vert}}\prod_{\mathfrak{v}\in \mathrm{V}_{\Gamma}}\mathrm{Cont}^{\mathrm{A},\Ch^{r, \epsilon}}_{\widetilde{\Gamma}_c,\mathfrak{v}}\prod_{\mathfrak{l}\in \mathrm{L}_{\Gamma}}\mathrm{Cont}^{\mathrm{A},\Ch^{r, \epsilon}}_{\widetilde{\Gamma}_c,\mathfrak{l}}\prod_{\mathfrak{e}\in \mathrm{E}_{\Gamma}} \mathrm{Cont}^{\mathrm{A},\Ch^{r, \epsilon}}_{\widetilde{\Gamma}_c,\mathfrak{e}} \right).
\end{equation*}
We set
\begin{equation*}
\Product_{\widetilde{\Gamma}_c}^{A,\Ch^{r, \epsilon}} \coloneqq \prod_{\mathfrak{v}\in \mathrm{V}_{\Gamma}}\mathrm{Cont}^{\mathrm{A},\Ch^{r, \epsilon}}_{\widetilde{\Gamma}_c,\mathfrak{v}}\prod_{\mathfrak{l}\in \mathrm{L}_{\Gamma}}\mathrm{Cont}^{\mathrm{A},\Ch^{r, \epsilon}}_{\widetilde{\Gamma}_c,\mathfrak{l}}\prod_{\mathfrak{e}\in \mathrm{E}_{\Gamma}} \mathrm{Cont}^{\mathrm{A},\Ch^{r, \epsilon}}_{\widetilde{\Gamma}_c,\mathfrak{e}}
\end{equation*}
to be the product expression in the graph contribution $\mathrm{Cont}^{\Ch^{r, \epsilon}}_{\widetilde{\Gamma}_c}(v_{s_1} \otimes \cdots \otimes v_{s_n})$. We split  the product $\Product_{\widetilde{\Gamma}_c}^{A,\Ch^{r, \epsilon}}$ as
\begin{equation}\label{eqn:prod_decomp}
\Product_{\widetilde{\Gamma}_c}^{A,\Ch^{r, \epsilon}}=\Product_{\widetilde{\Gamma}_c}^{A,\Ch^{r, \epsilon}}\{0,0\} \cdot \Product_{\widetilde{\Gamma}_c}^{A,\Ch^{r, \epsilon}}\{0,+\} \cdot \Product_{\widetilde{\Gamma}_c}^{A,\Ch^{r, \epsilon}}\{+,+\}
\end{equation}
where
\begin{align}
\Product_{\widetilde{\Gamma}_c}^{A,\Ch^{r, \epsilon}}\{0,0\} &\coloneqq \prod_{\mathfrak{v}\in \mathrm{V}_{\widetilde{\Gamma}_c}^{0}}\mathrm{Cont}^{\mathrm{A},\Ch^{r, \epsilon}}_{\widetilde{\Gamma}_c,\mathfrak{v}}\prod_{\mathfrak{l}\in \mathrm{L}_{\widetilde{\Gamma}_c}^{0}}\mathrm{Cont}^{\mathrm{A},\Ch^{r, \epsilon}}_{\widetilde{\Gamma}_c,\mathfrak{l}}\prod_{\mathfrak{e}\in \mathrm{E}_{\widetilde{\Gamma}_c}^{\{0,0\}}} \mathrm{Cont}^{\mathrm{A},\Ch^{r, \epsilon}}_{\widetilde{\Gamma}_c,\mathfrak{e}},\label{eqn:zero_zero_coloring_product}\\
\Product_{\widetilde{\Gamma}_c}^{A,\Ch^{r, \epsilon}}\{0,+\} &\coloneqq \prod_{\mathfrak{e}\in \mathrm{E}_{\widetilde{\Gamma}_c}^{\{0,+\}}} \mathrm{Cont}^{\mathrm{A},\Ch^{r, \epsilon}}_{\widetilde{\Gamma}_c,\mathfrak{e}},\label{eqn:zero_plus_coloring_product}\\
\Product_{\widetilde{\Gamma}_c}^{A,\Ch^{r, \epsilon}}\{+,+\} &\coloneqq \prod_{\mathfrak{v}\in \mathrm{V}_{\widetilde{\Gamma}_c}^{+}}\mathrm{Cont}^{\mathrm{A},\Ch^{r, \epsilon}}_{\widetilde{\Gamma}_c,\mathfrak{v}}\prod_{\mathfrak{l}\in \mathrm{L}_{\widetilde{\Gamma}_c}^{+}}\mathrm{Cont}^{\mathrm{A},\Ch^{r, \epsilon}}_{\widetilde{\Gamma}_c,\mathfrak{l}}\prod_{\mathfrak{e}\in \mathrm{E}_{\widetilde{\Gamma}_c}^{\{+,+\}}} \mathrm{Cont}^{\mathrm{A},\Ch^{r, \epsilon}}_{\widetilde{\Gamma}_c,\mathfrak{e}}.\label{eqn:plus_plus_coloring_product}
\end{align}

\begin{lemma}\label{lem:ord_t_P_plus_plus_geq_zero}
For any decorated stable graph $\widetilde{\Gamma}_c$ and any flag assignment $A\in \mathbb{Z}_{\geq 0}^{\vert \mathrm{H}_{\Gamma}\vert}$, we have
\begin{equation*}
\operatorname{ord}_t(\Product_{\widetilde{\Gamma}_c}^{A,\Ch^{r, \epsilon}}\{+,+\})\geq 0.
\end{equation*}
\end{lemma}

\begin{proof}
This follows from Lemma \ref{lem:limit_of_non_zero_colorings} since the limits of the factors in equation \eqref{eqn:plus_plus_coloring_product} exist as $t\to 0$.
\end{proof}

\begin{proposition}\label{prop:limit_of_nonzero_graphs}
Assume $n>0$, then for a decorated stable graph $\widetilde{\Gamma}_c$ and $s_i\in \{1,\ldots,r-1\}$, we have 
\begin{equation*}
\lim_{t\to 0}\mathrm{Cont}^{\Ch^{r, \epsilon}}_{\widetilde{\Gamma}_c}(v_{s_1} \otimes \cdots \otimes v_{s_n})=\mathrm{Cont}^{\Th^{r, \epsilon}}_{\widetilde{\Gamma}_c}(v_{s_1} \otimes \cdots \otimes v_{s_n})
\end{equation*}
if each vertex of $\widetilde{\Gamma}_c$ has a non-zero coloring.
\end{proposition}

\begin{proof}
If  $\widetilde{\Gamma}_c$ is  a decorated stable graph with only non-zero colorings, then by definition we have
\begin{equation*}
\Product_{\widetilde{\Gamma}_c}^{A,\Ch^{r, \epsilon}}=\Product_{\widetilde{\Gamma}_c}^{A,\Ch^{r, \epsilon}}\{+,+\}.
\end{equation*}
Hence, the proof follows from Proposition \ref{prop:Graph_Sum_Formula_Shifted_Chiodo} and Lemma \ref{lem:limit_of_non_zero_colorings}.
\end{proof}

\begin{proposition}
Assume $n>0$, then for a decorated stable graph $\widetilde{\Gamma}_c$ and $s_i\in \{1,\ldots,r-1\}$, we have 
\begin{equation*}
\lim_{t\to 0}\mathrm{Cont}^{\Ch^{r, \epsilon}}_{\widetilde{\Gamma}_c}(v_{s_1} \otimes \cdots \otimes v_{s_n})=0
\end{equation*}
if $\widetilde{\Gamma}_c$ has a vertex with a zero coloring.
\end{proposition}

\begin{proof}
We will prove the statement by showing that the order $\operatorname{ord}_t(\Product_{\widetilde{\Gamma}_c}^{A,\Ch^{r, \epsilon}})$ of $t$ in $\Product_{\widetilde{\Gamma}_c}^{A,\Ch^{r, \epsilon}}$ is positive for any decorated stable graph $\widetilde{\Gamma}_c$ with at least one vertex with zero coloring and arbitrary flag assignment $A\in \mathbb{Z}_{\geq 0}^{\vert \mathrm{H}_{\Gamma}\vert}$.

First, observe that we have
\begin{equation}\label{eqn:t_order_prod_geq_00_0plus}
\operatorname{ord}_t(\Product_{\widetilde{\Gamma}_c}^{A,\Ch^{r, \epsilon}}) \geq \operatorname{ord}_t(\Product_{\widetilde{\Gamma}_c}^{A,\Ch^{r, \epsilon}}\{0,0\}) + \operatorname{ord}_t(\Product_{\widetilde{\Gamma}_c}^{A,\Ch^{r, \epsilon}}\{0,+\})
\end{equation}
by Lemma \ref{lem:ord_t_P_plus_plus_geq_zero} and equation \eqref{eqn:prod_decomp}.

For vertices $\mathfrak{v}\in \mathrm{V}_{\widetilde{\Gamma}_c}^{0}$ with zero coloring, we define the following
\begin{equation*}
\begin{aligned}
\mathrm{L}_{\Gamma}(\mathfrak{v})
&\coloneqq \text{The set of legs attached to $\mathfrak{v}$},\\
\mathrm{E}_{\widetilde{\Gamma}_c}^{0}(\mathfrak{v})
&\coloneqq \text{The set of edges attached to $\mathfrak{v}$ with the other end also attached to a vertex with zero coloring},\\
\mathrm{E}_{\widetilde{\Gamma}_c}^{+}(\mathfrak{v})
&\coloneqq \text{The set of edges attached to $\mathfrak{v}$ with the other end attached to a non-zero coloring},\\
\mathrm{H}_{\Gamma}(\mathfrak{v}) &\coloneqq \text{The set of half-edges attached to $\mathfrak{v}$},
\end{aligned}
\end{equation*}
and we set
\begin{equation*}
\widetilde{\mathrm{H}}_{\Gamma}(\mathfrak{v})
\coloneqq \mathrm{H}_{\Gamma}(\mathfrak{v}) \setminus \mathrm{L}_{\Gamma}(\mathfrak{v}),
\end{equation*}
the set of half-edges of $\mathfrak v$ that belong to edges, and decompose it as
\begin{equation*}
\widetilde{\mathrm{H}}_{\Gamma}(\mathfrak{v})
=
\widetilde{\mathrm{H}}^{0}_{\widetilde{\Gamma}_c}(\mathfrak{v}) \sqcup \widetilde{\mathrm{H}}^{+}_{\widetilde{\Gamma}_c}(\mathfrak{v}),
\end{equation*}
where
\begin{equation*}
\widetilde{\mathrm{H}}^{0}_{\widetilde{\Gamma}_c}(\mathfrak{v})
\coloneqq \bigl\{\mathfrak{h}\in \widetilde{\mathrm{H}}_{\Gamma}(\mathfrak{v}) \,\big\vert\,
c\bigl(\nu(\iota(\mathfrak{h}))\bigr)=0\bigr\},
 \quad
\widetilde{\mathrm{H}}^{+}_{\widetilde{\Gamma}_c}(\mathfrak{v})
\coloneqq \bigl\{\mathfrak{h}\in \widetilde{\mathrm{H}}_{\Gamma}(\mathfrak{v}) \,\big\vert\,
c\bigl(\nu(\iota(\mathfrak{h}))\bigr)\neq 0\bigr\}.
\end{equation*}
In words, $\widetilde{\mathrm{H}}^{0}_{\widetilde{\Gamma}_c}(\mathfrak{v})$ (resp.\ $\widetilde{\mathrm{H}}^{+}_{\widetilde{\Gamma}_c}(\mathfrak{v})$)
consists of those half-edges at $\mathfrak v$ whose opposite endpoint is a vertex that has zero (resp.\ non-zero) coloring.

Furthermore, for vertices $\mathfrak{v}\in \mathrm{V}_{\widetilde{\Gamma}_c}^{0}$ with zero coloring, we define
\begin{equation*}
\begin{aligned}
\mathsf{a}(\mathfrak{v})
&\coloneqq \sum_{\mathfrak{l}\in \mathrm{L}_{\Gamma}(\mathfrak{v})} a_{\ell(\mathfrak{l})} \quad \text{(the sum of all of exponents $a_{\ell{(\mathfrak{l})}}$ of legs $\mathfrak{l}$ attached to the vertex $\mathfrak{v}$)},\\
\mathsf{b}(\mathfrak{v})
&\coloneqq \sum_{\mathfrak{h}\in \widetilde{\mathrm{H}}_{\Gamma}(\mathfrak{v})} b_{\mathfrak{h}} \quad \text{(the sum of all of exponents $b_{\mathfrak{h}}$ of half-edges $\mathfrak{h}$ attached to the vertex $\mathfrak{v}$).}
\end{aligned}
\end{equation*}

Note that we can rewrite $\operatorname{ord}_t(\Product_{\widetilde{\Gamma}_c}^{A,\Ch^{r, \epsilon}}\{0,0\})$ as sum over the vertices $\mathfrak{v}\in \mathrm{V}_{\Gamma}$ with zero coloring\footnote{The factor of $\frac{1}{2}$ appears in the expression of $\operatorname{ord}_t(\Product_{\widetilde{\Gamma}_c}^{A,\Ch^{r, \epsilon}}\{0,0\})$ due to the double counting of edges with both end vertices having zero color.}:
\begin{equation}\label{eqn:prod_00_expansion}
\begin{aligned}
&\operatorname{ord}_t(\Product_{\widetilde{\Gamma}_c}^{A,\Ch^{r, \epsilon}}\{0,0\})\\
&\hspace{4em} =\sum_{\mathfrak{v}\in \mathrm{V}_{\widetilde{\Gamma}_c}^{0}} \left( \operatorname{ord}_t(\mathrm{Cont}^{\mathrm{A},\Ch^{r, \epsilon}}_{\widetilde{\Gamma}_c,\mathfrak{v}})
+ \sum_{\mathfrak{l}\in \mathrm{L}_{\Gamma}(\mathfrak{v})}\operatorname{ord}_t(\mathrm{Cont}^{\mathrm{A},\Ch^{r, \epsilon}}_{\widetilde{\Gamma}_c,\mathfrak{l}})  + \frac{1}{2} \sum_{\mathfrak{e}\in \mathrm{E}_{\widetilde{\Gamma}_c}^{0}(\mathfrak{v})}\operatorname{ord}_t(\mathrm{Cont}^{\mathrm{A},\Ch^{r, \epsilon}}_{\widetilde{\Gamma}_c,\mathfrak{e}})\right)
\end{aligned}
\end{equation}
by equation \eqref{eqn:zero_zero_coloring_product}. Similarly, we can rewrite $\operatorname{ord}_t(\Product_{\widetilde{\Gamma}_c}^{A,\Ch^{r, \epsilon}}\{0,+\})$  as sum over the vertices $\mathfrak{v}\in \mathrm{V}_{\Gamma}$ with zero coloring
\begin{equation}\label{eqn:prod_0plus_expansion}
\operatorname{ord}_t(\Product_{\widetilde{\Gamma}_c}^{A,\Ch^{r, \epsilon}}\{0,+\})
=\sum_{\mathfrak{v}\in \mathrm{V}_{\widetilde{\Gamma}_c}^{0}}\left( \sum_{\mathfrak{e}\in \mathrm{E}_{\widetilde{\Gamma}_c}^{+}(\mathfrak{v})} \operatorname{ord}_t(\mathrm{Cont}^{\mathrm{A},\Ch^{r, \epsilon}}_{\widetilde{\Gamma}_c,\mathfrak{e}}) \right) 
\end{equation}
by equation \eqref{eqn:zero_plus_coloring_product}. Next, by Lemma \ref{lem:One_zero_coloring_half_edge} and Lemma \ref{lem:both_zero_coloring_order_of_t}, we  observe that
\begin{equation*}
\begin{aligned}
&\sum_{\mathfrak{v}\in \mathrm{V}_{\widetilde{\Gamma}_c}^{0}} \left( \frac{1}{2} \sum_{\mathfrak{e}\in \mathrm{E}_{\widetilde{\Gamma}_c}^{0}(\mathfrak{v})}\operatorname{ord}_t(\mathrm{Cont}^{\mathrm{A},\Ch^{r, \epsilon}}_{\widetilde{\Gamma}_c,\mathfrak{e}}) +  \sum_{\mathfrak{e}\in \mathrm{E}_{\widetilde{\Gamma}_c}^{+}(\mathfrak{v})} \operatorname{ord}_t(\mathrm{Cont}^{\mathrm{A},\Ch^{r, \epsilon}}_{\widetilde{\Gamma}_c,\mathfrak{e}}) \right)\\
&\geq
\sum_{\mathfrak{v}\in \mathrm{V}_{\widetilde{\Gamma}_c}^{0}} \left(\frac{1}{2} \sum_{\mathfrak{h}\in  \widetilde{\mathrm{H}}^{0}_{\widetilde{\Gamma}_c}(\mathfrak{v})}\left(-\frac{1}{2}-b_{\mathfrak{h}}-\frac{1}{2}-b_{\iota(\mathfrak{h})}\right) +  \sum_{\mathfrak{h}\in \widetilde{\mathrm{H}}^{+}_{\widetilde{\Gamma}_c}(\mathfrak{v})} \left(-\frac{1}{2}-b_{\mathfrak{h}}+1\right)\right)\\
&=
\sum_{\mathfrak{v}\in \mathrm{V}_{\widetilde{\Gamma}_c}^{0}} \left( \sum_{\mathfrak{h}\in  \widetilde{\mathrm{H}}^{0}_{\widetilde{\Gamma}_c}(\mathfrak{v})}\left(-\frac{1}{2}-b_{\mathfrak{h}}\right) +  \sum_{\mathfrak{h}\in \widetilde{\mathrm{H}}^{+}_{\widetilde{\Gamma}_c}(\mathfrak{v})} \left(-\frac{1}{2}-b_{\mathfrak{h}}+1\right)\right)
\end{aligned}
\end{equation*}
where the last line follows from
\begin{equation*}
\sum_{\mathfrak{v}\in \mathrm{V}_{\widetilde{\Gamma}_c}^{0}} \sum_{\mathfrak{h}\in  \widetilde{\mathrm{H}}^{0}_{\widetilde{\Gamma}_c}(\mathfrak{v})}\left(b_{\mathfrak{h}}+b_{\iota(\mathfrak{h})}\right)=2\sum_{\mathfrak{v}\in \mathrm{V}_{\widetilde{\Gamma}_c}^{0}} \sum_{\mathfrak{h}\in  \widetilde{\mathrm{H}}^{0}_{\widetilde{\Gamma}_c}(\mathfrak{v})} b_{\mathfrak{h}}.
\end{equation*}
Hence, we can re-organize the inequality above as
\begin{equation}\label{eqn:zero_color_vanishing_in_proof_eqn1}
\sum_{\mathfrak{v}\in \mathrm{V}_{\widetilde{\Gamma}_c}^{0}} \left( \frac{1}{2} \sum_{\mathfrak{e}\in \mathrm{E}_{\widetilde{\Gamma}_c}^{0}(\mathfrak{v})}\operatorname{ord}_t(\mathrm{Cont}^{\mathrm{A},\Ch^{r, \epsilon}}_{\widetilde{\Gamma}_c,\mathfrak{e}}) +  \sum_{\mathfrak{e}\in \mathrm{E}_{\widetilde{\Gamma}_c}^{+}(\mathfrak{v})} \operatorname{ord}_t(\mathrm{Cont}^{\mathrm{A},\Ch^{r, \epsilon}}_{\widetilde{\Gamma}_c,\mathfrak{e}}) \right)
\geq \sum_{\mathfrak{v}\in \mathrm{V}_{\widetilde{\Gamma}_c}^{0}} \left( \sum_{\mathfrak{h}\in  \widetilde{\mathrm{H}}_{\Gamma}(\mathfrak{v})}\left(-\frac{1}{2}-b_{\mathfrak{h}}\right) +  \sum_{\mathfrak{h}\in \widetilde{\mathrm{H}}^{+}_{\widetilde{\Gamma}_c}(\mathfrak{v})} 1\right).
\end{equation}
We also have
\begin{equation}\label{eqn:vertex_leg_order_bound_zero_color}
\begin{aligned}
&\sum_{\mathfrak{v}\in \mathrm{V}_{\widetilde{\Gamma}_c}^{0}} \left( \operatorname{ord}_t(\mathrm{Cont}^{\mathrm{A},\Ch^{r, \epsilon}}_{\widetilde{\Gamma}_c,\mathfrak{v}})
+ \sum_{\mathfrak{l}\in \mathrm{L}_{\Gamma}(\mathfrak{v})}\operatorname{ord}_t(\mathrm{Cont}^{\mathrm{A},\Ch^{r, \epsilon}}_{\widetilde{\Gamma}_c,\mathfrak{l}})\right)\\
&\geq
\sum_{\mathfrak{v}\in \mathrm{V}_{\widetilde{\Gamma}_c}^{0}} \left(\frac{3}{2}(2\mathrm{g}(\mathfrak{v})-2+\mathrm{n}(\mathfrak{v})+l_{\mathfrak{v}})+\sum_{j=1}^{l_{\mathfrak{v}}} \eta(T_{k_j},\widetilde{e}_{c(\mathfrak{v})})+ \sum_{\mathfrak{l}\in\mathrm{L}_{\Gamma}(\mathfrak{v})}\left( -\frac{1}{2}-{a_{\ell(\mathfrak{l})}}+{1}\right)\right)
\end{aligned}
\end{equation}
by Proposition \ref{prop:Graph_Sum_Formula_Shifted_Chiodo}, Lemma \ref{lem:xi_i_and_zeta_i_and_inverses_ring_containment} and Lemma \ref{lem:order_t_bound_zero_leg}.

Putting equations \eqref{eqn:t_order_prod_geq_00_0plus}, \eqref{eqn:prod_00_expansion}, \eqref{eqn:prod_0plus_expansion}, \eqref{eqn:vertex_leg_order_bound_zero_color} and Corollary \ref{cor:R_and_T_orders_in_t}, we see that
\begin{equation*}
\begin{aligned}
&\operatorname{ord}_t(\Product_{\widetilde{\Gamma}_c}^{A,\Ch^{r, \epsilon}})\\
&\geq \operatorname{ord}_t(\Product_{\widetilde{\Gamma}_c}^{A,\Ch^{r, \epsilon}}\{0,0\}) + \operatorname{ord}_t(\Product_{\widetilde{\Gamma}_c}^{A,\Ch^{r, \epsilon}}\{0,+\})\\
&\geq 
\sum_{\mathfrak{v}\in \mathrm{V}_{\widetilde{\Gamma}_c}^{0}} \left(\frac{3}{2}(2\mathrm{g}(\mathfrak{v})-2+\mathrm{n}(\mathfrak{v})+l_{\mathfrak{v}})+\sum_{j=1}^{l_{\mathfrak{v}}}\left( -k_j-\frac{1}{2}\right)+ \sum_{\mathfrak{l}\in\mathrm{L}_{\Gamma}(\mathfrak{v})}\left( -\frac{1}{2}-{a_{\ell(\mathfrak{l})}}+{1}\right)+\sum_{\mathfrak{h}\in  \widetilde{\mathrm{H}}_{\Gamma}(\mathfrak{v})}\left(-\frac{1}{2}-b_{\mathfrak{h}}\right) +  \sum_{\mathfrak{h}\in \widetilde{\mathrm{H}}^{+}_{\widetilde{\Gamma}_c}(\mathfrak{v})} 1 \right)\\
&\geq
\sum_{\mathfrak{v}\in \mathrm{V}_{\widetilde{\Gamma}_c}^{0}} \left(3\mathrm{g}(\mathfrak{v})-3+\frac{3}{2}\mathrm{n}(\mathfrak{v})+l_{\mathfrak{v}}-\sum_{j=1}^{l_{\mathfrak{v}}}k_j-\frac{1}{2}\vert \mathrm{L}_{\Gamma}(\mathfrak{v}) \vert-\mathsf{a}(\mathfrak{v}) -\frac{1}{2}\vert \widetilde{\mathrm{H}}_{\Gamma}(\mathfrak{v}) \vert-\mathsf{b}(\mathfrak{v}) +\vert \mathrm{L}_{\Gamma}(\mathfrak{v})\vert
+\vert \widetilde{\mathrm{H}}^{+}_{\widetilde{\Gamma}_c}(\mathfrak{v})\vert \right)
\end{aligned}
\end{equation*}

Using the fact that
\begin{equation*}
\mathrm{n}(\mathfrak{v})=\vert \mathrm{H}_{\Gamma}(\mathfrak{v})\vert = \vert \mathrm{L}_{\Gamma}(\mathfrak{v}) \vert + \vert \widetilde{\mathrm{H}}_{\Gamma}(\mathfrak{v}) \vert
\end{equation*}
we obtain
\begin{equation}\label{eqn:order_t_in_product_proof2}
\operatorname{ord}_t(\Product_{\widetilde{\Gamma}_c}^{A,\Ch^{r, \epsilon}})\geq \sum_{\mathfrak{v}\in \mathrm{V}_{\widetilde{\Gamma}_c}^{0}} \left(3\mathrm{g}(\mathfrak{v})-3+\mathrm{n}(\mathfrak{v})+l_{\mathfrak{v}}-\sum_{j=1}^{l_{\mathfrak{v}}}k_j-\mathsf{a}(\mathfrak{v}) -\mathsf{b}(\mathfrak{v}) +\vert \mathrm{L}_{\Gamma}(\mathfrak{v})\vert
+\vert \widetilde{\mathrm{H}}^{+}_{\widetilde{\Gamma}_c}(\mathfrak{v})\vert \right).
\end{equation}

By Proposition \ref{prop:Graph_Sum_Formula_Shifted_Chiodo}, we see that in order not to have a trivial vanishing,  the total $\psi$-degree at each vertex $\mathfrak v$
must not exceed $\dim \overline{\mathcal{M}}_{\mathrm g(\mathfrak v),\mathrm n(\mathfrak v)+l_{\mathfrak v}}$, i.e. 
\begin{equation*}
\mathsf{a}(\mathfrak{v})+\mathsf{b}(\mathfrak{v})+\sum_{j=1}^{l_{\mathfrak{v}}} k_j \leq 3 \mathrm{g}(\mathfrak{v})-3+\mathrm{n}(\mathfrak{v})+l_{\mathfrak{v}}.
\end{equation*}
Taking this condition into account together with inequality \eqref{eqn:orders_of_t_R_ij_different}, we see that
\begin{equation*}
\operatorname{ord}_t(\Product_{\widetilde{\Gamma}_c}^{A,\Ch^{r, \epsilon}})\geq \sum_{\mathfrak{v}\in \mathrm{V}_{\widetilde{\Gamma}_c}^{0}} \left(\vert \mathrm{L}_{\Gamma}(\mathfrak{v})\vert
+\vert \widetilde{\mathrm{H}}^{+}_{\widetilde{\Gamma}_c}(\mathfrak{v})\vert \right).
\end{equation*}
for a vertex $\mathfrak{v} \in \mathrm{V}_{\widetilde{\Gamma}_c}^0$. So, we need to show there is a vertex $\mathfrak{v}\in \mathrm{V}_{\widetilde{\Gamma}_c}^0$ with $\operatorname{ord}_t(\Product_{\widetilde{\Gamma}_c}^{A,\Ch^{r, \epsilon}})>0$.

Since $\widetilde{\Gamma}_c$ is connected as a graph and by our assumption it has at least one vertex with zero coloring, we have two possibilities:
\begin{itemize}
    \item[(i)] there is an edge $\mathfrak{e}\in \mathrm{E}_{\widetilde{\Gamma}_c}^{\{0,+\}}$ with one half-edge attached to a vertex of zero coloring and the other half-edge  attached to a vertex of non-zero coloring,
    \item[(ii)] all vertices have zero coloring.
\end{itemize}
In the case (i), we have a vertex $\mathfrak{v}\in \mathrm{V}_{\widetilde{\Gamma}_c}^0$ with zero coloring and the edge $\mathfrak{e}$ is attached to $\mathfrak{v}$ and the other end of $\mathfrak{e}$ has non-zero coloring. This means that for this vertex, we have
\begin{equation*}
\vert \widetilde{\mathrm{H}}^{+}_{\widetilde{\Gamma}_c}(\mathfrak{v})\vert >0.
\end{equation*}
In the case (ii), all vertices have zero coloring. By our assumption in the statement we have $n>0$. As a result, there is a leg attached to a vertex $v\in \mathfrak{v}\in \mathrm{V}_{\widetilde{\Gamma}_c}^0=\mathrm{V}_{\Gamma}$ with a leg attached to itself. This means that for this vertex, we have
\begin{equation*}
\vert \mathrm{L}_{\Gamma}(\mathfrak{v})\vert >0.
\end{equation*}
As a result, in each case, we conclude that
\begin{equation*}
\operatorname{ord}_t(\Product_{\widetilde{\Gamma}_c}^{A,\Ch^{r, \epsilon}})>0.
\end{equation*}
Therefore, summing over flag assignments which have non-zero contributions, we conclude that  has strictly positive $t$ order\footnote{The cancellations can not decrease the order of $t$.}. Hence, we see that
\begin{equation*}
\lim_{t\to 0}\mathrm{Cont}^{\Ch^{r, \epsilon}}_{\widetilde{\Gamma}_c}(v_{s_1} \otimes \cdots \otimes v_{s_n})=0.
\end{equation*}
This completes the proof.
\end{proof}

\begin{proposition}\label{prop:limit_of_graphs_as_t_to_zero}
Let $n>0$ and $s_i\in \{1,\ldots,r-1\}$. For a stable graph ${\Gamma}$, we have
\begin{equation*}
\lim_{t\to 0}\mathrm{Cont}^{\Ch^{r, \epsilon}}_{{\Gamma}}(v_{s_1} \otimes \cdots \otimes v_{s_n})=\mathrm{Cont}^{\Th^{r, \epsilon}}_{{\Gamma}}(v_{s_1} \otimes \cdots \otimes v_{s_n}).
\end{equation*}
\end{proposition}

\begin{proof}
Recall that $\mathrm{Cont}^{\Ch^{r, \epsilon}}_{{\Gamma}}(v_{s_1} \otimes \cdots \otimes v_{s_n})$ is defined by fixing a representative of the isomorphism class $\Gamma\in\mathsf G_{g,n}$ and summing over all its colorings. Hence, we can split $\mathrm{Cont}^{\Ch^{r, \epsilon}}_{{\Gamma}}(v_{s_1} \otimes \cdots \otimes v_{s_n})$  into two parts as follows:
\begin{equation*}
\mathrm{Cont}^{\Ch^{r, \epsilon}}_{{\Gamma}}(v_{s_1} \otimes \cdots \otimes v_{s_n})
=\sum_{c(\mathrm{V}_{\Gamma})\subset \{1,...,r-1\}} \mathrm{Cont}^{\Ch^{r, \epsilon}}_{\widetilde{\Gamma}_c}(v_{s_1} \otimes \cdots \otimes v_{s_n})+\sum_{c(\mathrm{V}_{\Gamma})\cap \{0\}\neq \varnothing} \mathrm{Cont}^{\Ch^{r, \epsilon}}_{\widetilde{\Gamma}_c}(v_{s_1} \otimes \cdots \otimes v_{s_n})
\end{equation*}
where the first sum is over colorings which has no zero colorings and  the second sum is over colorings that has at least coloring one vertex with zero coloring. Taking limit the second sum vanishes and we obtain
\begin{equation*}
\begin{aligned}
\lim_{t\to 0}\mathrm{Cont}^{\Ch^{r, \epsilon}}_{{\Gamma}}(v_{s_1} \otimes \cdots \otimes v_{s_n})
&=\sum_{c(\mathrm{V}_{\Gamma})\subset \{1,...,r-1\}} \lim_{t\to 0} \mathrm{Cont}^{\Ch^{r, \epsilon}}_{\widetilde{\Gamma}_c}(v_{s_1} \otimes \cdots \otimes v_{s_n})\\
&=\sum_{c \colon \mathrm{V}_{\Gamma}\rightarrow\{1,...,r-1\}} \mathrm{Cont}^{\Th^{r, \epsilon}}_{\widetilde{\Gamma}_c}(v_{s_1} \otimes \cdots \otimes v_{s_n})\\
&=\mathrm{Cont}^{\Th^{r, \epsilon}}_{{\Gamma}}(v_{s_1} \otimes \cdots \otimes v_{s_n})
\end{aligned}
\end{equation*}
which completes the proof.
\end{proof}

\begin{remark}\label{rem:strata_level}
  For the estimates above, we did not use any cohomological relations among the pairs $(\Gamma,\gamma)$:
  the graph sums of Section \ref{sec:reconstruction} are formal expressions in
  \begin{equation*}
    \strata_{g,n}\otimes\CC[\epsilon^{\pm\frac{1}{2(r-1)}}](\!(t^{\frac{1}{2}})\!),
  \end{equation*}
  the gluing maps $\left({\mathpzc{gl}_{\Gamma}}\right)_{\star}$ act formally, and the only
  geometric input is the dimension constraint
  $$\mathsf{a}(\mathfrak v)+\mathsf{b}(\mathfrak v)+\sum_j k_j\leq
  3\mathrm{g}(\mathfrak v)-3+\mathrm{n}(\mathfrak v)+l_{\mathfrak v}$$
  which is exactly
  the condition defining a formal basic class in Section \ref{subsec:strata_algebra}.
  Consequently, all the bounds of this section hold coefficientwise with respect to
  the basis of $\strata_{g,n}$, and Theorem~\ref{ThmD} (equivalently Proposition \ref{prop:limit_of_graphs_as_t_to_zero}) holds verbatim for the strata-valued
  reconstructions $\Ch^{r,\epsilon,\strata}$ and $\Th^{r,\epsilon,\strata}$ of
  Section \ref{ss:taut-cohft}.
\end{remark}


\section{Proof of the main theorem}\label{sec:main-proof}

We restate the main Theorem~\ref{ThmA}:
\begin{theorem}
  \label{thm:main}
  For any $r \ge 2$, the negative $r$-spin relations are a consequence of the $3$-spin relations.
  More precisely if $\cI_{g,n}^{\Om^\spin}$ and $\cI_{g,n}^{\Th^{r, \epsilon}}$ be the ideal of $3$-spin and negative $r$-spin relations as in Section \ref{ss:taut-cohft}, then we have
  \begin{equation*}
    \cI_{g,n}^{\Th^{r, \epsilon}} \subseteq \cI_{g,n}^{\Om^\spin}.
  \end{equation*}
\end{theorem}

The proof of Theorem~\ref{thm:main} will occupy the rest of this section.

We start by applying the results from \cite{paper1} to connect the $3$-spin relations to a set of relations from the Chiodo class.
Given $\Vec{\epsilon}=\sum_{i=0}^{r-1}\epsilon_i v_i$,  the \emph{$\Vec{\epsilon}$-shifted Chiodo CohFT} $\Ch_{g,n}^{r,\Vec{\epsilon}}$ is defined  by
\begin{equation*}
\Ch_{g,n}^{r,\Vec{\epsilon}}(v_{a_1}\otimes\cdots\otimes v_{a_n}) \coloneqq \sum_{m\geq 0}\frac{1}{m!}{\mathpzc{p}_{m}}_\star\Ch_{g,n+m}^{r}(v_{a_1}\otimes\cdots\otimes v_{a_n}\otimes v_{\Vec{\epsilon}}^{\otimes m}).
\end{equation*}
A priori, this CohFT is defined over the power series ring $\QQ[t^{\pm 1}][\![\epsilon_0, \dotsc, \epsilon_{r-1}]\!]$.
Applying the results from \cite{paper1} requires convergence in the $\epsilon$ parameters.
This is established in the following:
\begin{lemma}\label{lem:general_shift_polynomiality_lemma}
  For any $g, n$ such that $2g - 2 + n > 0$ and
  $v_{s_1}, \dotsc, v_{s_n} \in V$, we have
  \begin{equation*}
    \Ch^{r, \Vec\epsilon}_{g, n}(v_{s_1} \otimes \cdots \otimes v_{s_n} ) \in H^*(\ocM_{g, n})[\epsilon_0, \dotsc, \epsilon_{r - 2}, (1-\epsilon_{r-1})^{-1}, t].
  \end{equation*}
  In particular, if we fix a nonzero value for $t$, then $\Ch^{r}$ is a convergent CohFT in the sense of \cite
{paper1}.
\end{lemma}
\begin{proof}
A more detailed version of Lemma \ref{lem:general_shift_polynomiality_lemma} is stated as Proposition \ref{prop:general_shift_ring_containment} and proved  in Appendix \ref{appendix:general_shifts_for_r_s_bounds_polynomiality}.
\end{proof}

\begin{lemma}\label{lem:general-shift-chiodo-semisimple}
  Fix an arbitrary nonzero value for $t$ and $r \ge 2$.
  View $\Ch^{r, \Vec\epsilon}$ as a CohFT with parameters in $\QQ[\epsilon_0, \dotsc, \epsilon_{r - 2}, (1-\epsilon_{r-1})^{-1}]$.
  Then, the CohFT $\Ch^{r, \Vec\epsilon}$ is generically semisimple.
\end{lemma}
\begin{proof}
  We consider the homomorphism
  \begin{equation*}
    \varphi\colon \QQ[\epsilon_0, \dotsc, \epsilon_{r - 2}, (1-\epsilon_{r-1})^{-1}] \to \QQ[\epsilon],  \quad \epsilon_0 \mapsto \epsilon, \ \epsilon_1, \dotsc, \epsilon_{r-1} \mapsto 0.
  \end{equation*}
  Note that $v_0$ generates the Frobenius algebra associated to $\Ch^{r, \epsilon}$ in the sense of \cite
{paper1}.
  Furthermore, we may check that $\Ch^{r, \epsilon}$ as a CohFT over $\QQ[\epsilon]$ is generically semisimple as long as $t \neq 0$.
  Therefore, by \cite
{paper1} the CohFT $\Ch^{r, \Vec\epsilon}$ is also generically semisimple.
\end{proof}

\begin{proposition}
  \label{prop:3spin-to-chiodo-rels}
  Fix an arbitrary nonzero value for $t$ and $r \ge 2$.
  View $\Ch^{r, \Vec\epsilon}$ as a CohFT with parameters in $\QQ[\epsilon_0, \dotsc, \epsilon_{r - 2}, (1-\epsilon_{r-1})^{-1}]$, and let $\cI_{g,n}^{\Ch^{r, \Vec\epsilon}, t}$ be the corresponding ideal of tautological relations as in Section \ref{ss:taut-cohft}.
  We then have
  \begin{equation*}
    \cI^{\Ch^{r, \Vec\epsilon}, t} = \cI^{\Om^\spin}.
  \end{equation*}
\end{proposition}
\begin{proof}
  This follows from \cite
{paper1} whose hypothesis were established in Lemma \ref{lem:general_shift_polynomiality_lemma}.
\end{proof}

To make the connection to the negative $r$-spin relations, we consider intermediate ideals of tautological relations.
First, note that the CohFT $\Ch^{r, \Vec\epsilon}$ depends on many parameters.
Taking the limit $\epsilon_1, \dotsc, \epsilon_{r-1} \to 0$ recovers the shifted Chiodo CohFT $\Ch^{r, \epsilon_0}$ as studied in the previous sections.
For fixed nonzero $t$, we may view $\Ch^{r, \epsilon}$ as a CohFT with parameters in $\QQ[\epsilon]$, and we let $\cI_{g,n}^{\Ch^{r, \epsilon}, t}$ be the corresponding ideal of tautological relations as in Section \ref{ss:taut-cohft}.
Alternatively, if we do not specialize $t$, we may view $\Ch^{r, \epsilon}$ as a CohFT with parameters in $\QQ[\epsilon, t^{\pm 1}]$.
We denote the corresponding ideal of tautological relations by $\cI_{g,n}^{\Ch^{r, \epsilon}, t \neq 0}$.

\begin{proposition}
  \label{prop:chiodo-vec-to-chiodo-rels}
  Fix an arbitrary nonzero value for $t$ and $r \ge 2$.
  Then, we have
  \begin{equation*}
    \cI^{\Ch^{r, \epsilon}, t} \subseteq \cI^{\Ch^{r, \Vec\epsilon}, t}.
  \end{equation*}
\end{proposition}
\begin{proof}
  This follows from \cite
{paper1} applied to the homomorphism $\varphi$ from the proof of Lemma \ref{lem:general-shift-chiodo-semisimple}.
\end{proof}

\begin{proposition}
  \label{prop:chiodo-t-nonzero}
  Let $r \ge 2$.
  Then $\cI^{\Ch^{r, \epsilon}, t \neq 0}$ equals the ideal generated by the $\cI^{\Ch^{r, \epsilon}, t}$ for all choices of $t \neq 0$.
\end{proposition}
\begin{proof}
  For each choice of $t_0 \neq 0$, consider the evaluation homomorphism $\ev_{t_0}\colon \QQ[t^{\pm 1}, \epsilon] \to \QQ[\epsilon]$.
  Since $\Ch^{r, \epsilon}$ is semisimple for any choice of $t = t_0 \neq 0$, we may apply \cite
{paper1}  to this homomorphism, and find that $\cI_{g,n}^{\Ch^{r, \epsilon}, t \neq 0}$ contains the ideal generated by the $\cI_{g,n}^{\Ch^{r, \epsilon}, t_0}$ for all choices of $t_0 \neq 0$.

  In the other direction, consider any nonzero element in
  \begin{equation*}
    f \in \strata_{g,n} \otimes (\QQ[t^{\pm 1}, \epsilon, \Disc^{-1}]/ \QQ[t^{\pm 1}, \epsilon]),
  \end{equation*}
  and a linear form $\pi\colon \QQ[t^{\pm 1}, \epsilon, \Disc^{-1}]/ \QQ[t^{\pm 1}, \epsilon] \to \QQ$.
  We may choose some $k \ge 0$ such that
  \begin{equation*}
    f = \Disc^{-k} g
  \end{equation*}
  for a unique
  \begin{equation*}
    g \in (\strata_{g,n} \otimes \QQ[t^{\pm 1}])[\epsilon],
  \end{equation*}
  where the polynomial $g$ has degree less than $r$ in $\epsilon$.
  Choose $t_0, \dotsc, t_{r-1} \in \QQ$ such that
  \begin{equation*}
    \prod_{i = 0}^{r-1} \ev_{t_i}\colon \QQ[t^{\pm 1}, \epsilon] \to \prod_{i = 0}^{r-1} \QQ[\epsilon]
  \end{equation*}
  is injective on polynomials of degree less than $r$.
  We may thus choose a linear form
  \begin{equation*}
    \psi\colon \prod_{i = 0}^{r-1} \QQ[\epsilon]/(\Disc_t) \to \QQ
  \end{equation*}
  such that
  \begin{equation*}
    (\id_{\strata_{g,n}} \otimes \pi)(f) = \left(\id_{\strata_{g,n}} \otimes \left(\psi \circ \left(\prod_{i = 0}^{r-1} \ev_{t_i}\right)\right)\right)(f).
  \end{equation*}
  Therefore, the relation $(\id_{\strata_{g,n}} \otimes \pi)(f)$ lies in the ideal generated by the $\cI_{g,n}^{\Ch^{r, \epsilon}, t}$ for $t = t_0, \dotsc, t_{r-1}$.
  We conclude that $\cI_{g,n}^{\Ch^{r, \epsilon}, t \neq 0}$ is contained in the ideal generated by the $\cI_{g,n}^{\Ch^{r, \epsilon}, t}$ for all choices of $t \neq 0$.
\end{proof}

Combining propositions \ref{prop:3spin-to-chiodo-rels}, \ref{prop:chiodo-vec-to-chiodo-rels} and \ref{prop:chiodo-t-nonzero}, we obtain:
\begin{corollary}
  \label{cor:3spin-to-chiodo-rels}
  For any $r \ge 2$, we have
  \begin{equation*}
    \cI^{\Ch^{r, \epsilon}, t \neq 0} \subseteq \cI^{\Om^\spin}.
  \end{equation*}
\end{corollary}

First, observe that we have
\begin{equation*}
\Disc(p_t(x))^{-1}\in \QQ[\epsilon^{\pm 1}][\![t]\!]
\end{equation*}
from the expression \eqref{eqn:disc_of_p_t_of_x} of $\Disc(p_t(x))$.

We next consider the subring
\begin{equation*}
  \QQ[\epsilon, t, \Disc(p_t(x))^{-1}] \subset \QQ[\epsilon, t^{\pm 1}, \Disc(p_t(x))^{-1}],
\end{equation*}
which admits a homomorphism
\begin{equation*}
  \rho\colon \QQ[\epsilon, t, \Disc(p_t(x))^{-1}] \to \QQ[\epsilon^{\pm 1}]
\end{equation*}
defined by setting $t = 0$.

The following is the key technical result:
\begin{proposition}
  \label{lem:Chiodo-reconstruction-non-equivariant}
  For any $g$ and $n>0$ such that $2g - 2 + n > 0$, and for any
  $v_{s_1}, \dotsc, v_{s_n} \in \underline{V}$, the class
  \begin{equation*}
    \Ch^{r, \epsilon, \strata}_{g, n}(v_{s_1} \otimes \cdots \otimes v_{s_n}) \in \strata_{g, n}[\epsilon, t^{\pm 1}, \Disc(p_t(x))^{-1}]
  \end{equation*}
  lies in $\strata_{g, n}[\epsilon, t, \Disc(p_t(x))^{-1}]$, and we have
  \begin{equation*}
    \rho(\Ch^{r, \epsilon, \strata}_{g, n}(v_{s_1} \otimes \cdots \otimes v_{s_n}))
    = \Th^{r, \epsilon, \strata}_{g, n}(v_{s_1} \otimes \cdots \otimes v_{s_n}).
  \end{equation*}
\end{proposition}
\begin{proof}
  Let $g, n$ be such that $2g - 2 + n > 0$ and $v_{s_1}, \dotsc, v_{s_n} \in \underline{V}$.
  By Proposition~\ref{prop:limit_of_graphs_as_t_to_zero} and Remark~\ref{rem:strata_level}, we know that
  \begin{equation*}
    \Ch^{r, \epsilon, \strata}_{g, n}(v_{s_1} \otimes \cdots \otimes v_{s_n}) \in \strata_{g,n}\otimes\CC[\epsilon^{\pm\frac{1}{2(r-1)}}][\![t^{\frac{1}{2}}]\!].
  \end{equation*}
  The first part follows from the fact that
  \begin{equation*}
    \CC[\epsilon^{\pm\frac{1}{2(r-1)}}][\![t^{\frac{1}{2}}]\!] \cap \QQ[\epsilon, t^{\pm 1}, \Disc(p_t(x))^{-1}]
    = \QQ[\epsilon, t, \Disc(p_t(x))^{-1}].
  \end{equation*}
  The second part also follows from Proposition~\ref{prop:limit_of_graphs_as_t_to_zero}.
\end{proof}

\begin{proposition}
  \label{prop:chiodo-to-theta-rels}
  Let $r \ge 2$, and let $\cI^{\Th^{r, \epsilon}}$ be the ideal of negative $r$-spin relations as in \cite[Theorem B]{CGG25}.
  Then, we have
  \begin{equation*}
    \cI^{\Th^{r, \epsilon}} \subseteq \cI^{\Ch^{r, \epsilon}, t \neq 0}.
  \end{equation*}
\end{proposition}
\begin{proof}
  We take one of the relations generating $\cI^{\Th^{r, \epsilon}}$.
  So take any $g, n$ such that $2g - 2 + n > 0$ and $v_{s_1}, \dotsc, v_{s_n}  \in \underline{V}$ and any $\mathrm{pr}\colon\QQ[\epsilon^{\pm 1}] \to \QQ$ vanishing on $\QQ[\epsilon]$, and consider
  \begin{equation*}
    \mathrm{pr}(\Th_{g, n}^{r, \epsilon, \strata}(v_{s_1} \otimes \cdots \otimes v_{s_n})) \in \cI^{\Th^{r, \epsilon}}.
  \end{equation*}

  By definition of $\rho$, the composition $\mathrm{pr} \circ \rho$ vanishes on $\QQ[\epsilon, t]$.
  By the following lemma, there exists $\mathrm{pr}'\colon \QQ[\epsilon, t^{\pm 1}, \Disc(p_t(x))^{-1}] \to \QQ$ vanishing on $\QQ[\epsilon, t^{\pm 1}]$ such that $\mathrm{pr}' \circ i = \mathrm{pr} \circ \rho$, where $i\colon \QQ[\epsilon, t, \Disc(p_t(x))^{-1}] \to \QQ[\epsilon, t^{\pm 1}, \Disc(p_t(x))^{-1}]$ denotes the inclusion.

  We note that
  \begin{equation*}
    \mathrm{pr}'(\Ch_{g, n}^{r, \epsilon, \strata}(v_{s_1} \otimes \cdots \otimes v_{s_n})) \in \cI^{\Ch^{r, \epsilon}, t \neq 0}.
  \end{equation*}
  By Lemma \ref{lem:Chiodo-reconstruction-non-equivariant}, the coefficients of $\Ch_{g, n}^{r, \epsilon, \strata}(v_{s_1} \otimes \cdots \otimes v_{s_n})$ lie in the image of $i$, and hence by slight abuse of notation, we have
  \begin{equation*}
    \mathrm{pr}(\rho(\Ch_{g, n}^{r, \epsilon, \strata}(v_{s_1} \otimes \cdots \otimes v_{s_n}))) \in \cI^{\Ch^{r, \epsilon}, t \neq 0}.
  \end{equation*}
  By Lemma \ref{lem:Chiodo-reconstruction-non-equivariant}, we have
  \begin{equation*}
    \rho(\Ch_{g, n}^{r, \epsilon, \strata}(v_{s_1} \otimes \cdots \otimes v_{s_n}))
    = \Th_{g, n}^{r, \epsilon, \strata}(v_{s_1} \otimes \cdots \otimes v_{s_n}),
  \end{equation*}
  and we conclude
  \begin{equation*}
    \mathrm{pr}(\Th_{g, n}^{r, \epsilon, \strata}(v_{s_1} \otimes \cdots \otimes v_{s_n})) \in \cI^{\Ch^{r, \epsilon}, t \neq 0}.
  \end{equation*}
\end{proof}

\begin{lemma}
  Any $\QQ$ linear map $\mathrm{pr}\colon \QQ[\epsilon, t, \Disc(p_t(x))^{-1}] \to \QQ$ vanishing on $\QQ[\epsilon, t]$ may be extended to a linear map $\overline{\mathrm{pr}}\colon \QQ[\epsilon, t^{\pm 1}, \Disc(p_t(x))^{-1}] \to \QQ$ vanishing on $\QQ[\epsilon, t^{\pm 1}]$.
\end{lemma}
\begin{proof}
  As $\QQ$-vector spaces, we have
  \begin{equation*}
    \QQ[\epsilon, t^{\pm 1}, \Disc(p_t(x))^{-1}]
    = \QQ[\epsilon, t, \Disc(p_t(x))^{-1}] \oplus t^{-1} \QQ[\epsilon, t^{-1}, \Disc(p_t(x))^{-1}],
  \end{equation*}
  and so we may define $\widetilde{\mathrm{pr}}$ by extending $\mathrm{pr}$ by $0$ on $t^{-1} \QQ[\epsilon, t^{-1}, \Disc(p_t(x))^{-1}]$.
\end{proof}

The proof of Theorem~\ref{ThmA} is completed by combining Corollary \ref{cor:3spin-to-chiodo-rels} and Proposition \ref{prop:chiodo-to-theta-rels}.


\appendix

\section{Pairing and quantum product of Shifted Chiodo CohFT} \label{appendix:Metric_and_Quantum_Product}
In this part of the appendix, we prove the formula for pairing $\eta$ and quantum product $\bullet$ of $\epsilon v_0$-shifted Chiodo CohFT $\Ch^{r,\epsilon}$
\subsection{Calculation of the pairing}
\begin{lemma}
The pairing of $\Ch^{r,\epsilon}$ is given by
\begin{equation*}
\eta(v_a,v_b)
=
\begin{cases}
t^{-1}\quad \text{if}\quad a=b=0,\\
1\quad \text{if} \quad 1\leq a=r-b\leq r-1,\\
0\quad \text{otherwise}.
\end{cases}
\end{equation*}
\end{lemma}

\begin{proof}
Let $g=g_1+g_2$, $n=n_1+n_2$, and $\{a_1,\ldots,a_n\}=\{b_1,\ldots, b_{n_1}\} \bigsqcup \{c_1,\ldots, c_{n_2}\}$. Then, we have

\begin{equation*}
\begin{aligned}
\rank\left({\mathcal{V}_{g_1;b_1,\ldots,b_{n_1},i}^{r,-1}}\right)+
\rank\left({\mathcal{V}_{g_2;c_1,\ldots,c_{n_2},j}^{r,-1}}\right)
&=\frac{(r+2)(g_1-1)+n_1+1+\vert b\vert +i}{r} + \frac{(r+2)(g_2-1)+n_2+1+\vert c\vert +j}{r}\\
&=\frac{(r+2)(g-1)+n+\vert a\vert}{r}+\frac{-(r+2)+2+i+j}{r}\\
&=\rank\left({\mathcal{V}_{g;a_1,\ldots,a_{n}}^{r,-1}}\right)-1+\frac{i+j}{r}.
\end{aligned}
\end{equation*}
Also, by the modularity constraints we get
\begin{equation*}
\begin{aligned}
\vert a \vert 
&\equiv (2-2g-n)\quad (\mathrm{mod\,}r),\\
\vert b \vert +i 
&\equiv (2-2g_1-n_1-1)\quad (\mathrm{mod\,}r),\\
\vert c \vert +j 
&\equiv (2-2g_2-n_2-1)\quad (\mathrm{mod\,}r).
\end{aligned}
\end{equation*}
Putting these together, we get
\begin{equation*}
i+j \equiv 0\quad (\mathrm{mod\,}r).
\end{equation*}
As a result, for $0\leq i, j \leq r-1$, we have\footnote{Here $\delta_{i,j,k}$ is the product $\delta_{i,j}\delta_{j,k}$.}
\begin{equation*}
\begin{aligned}
\rank\left({\mathcal{V}_{g_1;b_1,\ldots,b_{n_1},i}^{r,-1}}\right)+
\rank\left({\mathcal{V}_{g_2;c_1,\ldots,c_{n_2},j}^{r,-1}}\right)
&=
\begin{cases}
\rank\left({\mathcal{V}_{g;a_1,\ldots,a_{n}}^{r,-1}}\right)-1 \quad &\text{if}\quad i=j=0,\\
\rank\left({\mathcal{V}_{g;a_1,\ldots,a_{n}}^{r,-1}}\right)\quad &\text{if}\quad 1\leq i,j \leq r-1
\end{cases}\\
&=\rank\left({\mathcal{V}_{g;a_1,\ldots,a_{n}}^{r,-1}}\right)-\delta_{i,j,0}.
\end{aligned}
\end{equation*}

Let $\mathpzc{r} \colon \ocM_{g_1,n_1+1}\times \ocM_{g_2,n_2+1}\rightarrow \ocM_{g,n}$ be the canonical gluing map of two curves. Observe the following
\begin{equation*}
\begin{aligned}
&\mathpzc{r}^{\star}\left(\Ch^{r}_{g,n}(v_{a_1}\otimes\cdots\otimes v_{a_n})\right)\\
&=
(-1)^nr^{1-g}{t}^{\rank\left({\mathcal{V}_{g;a_1,\ldots,a_n}^{r}}\right)}\mathpzc{r}^{\star}\left(\widetilde{\Ch}^{r,-1,r/t}_{g,n}(v_{a_1}\otimes\cdots\otimes v_{a_n})\right)\\
&=(-1)^nr^{1-g}{t}^{\rank\left({\mathcal{V}_{g;a_1,\ldots,a_n}^{r}}\right)}\sum_{0\leq i,j\leq r-1}\widetilde{\Ch}^{r,-1,r/t}_{g_1,n_1+1}(v_{a_1}\otimes\cdots\otimes v_{a_{n_1}}\otimes v_i)\widetilde{\eta}^{i,j}\widetilde{\Ch}^{r,-1,r/t}_{g_2,n_2+1}(v_{a_{n_1+1}}\otimes\cdots\otimes v_{a_n}\otimes v_j)\\
&=(-1)^{n_1+1}(-1)^{n-n_1+1}r^{1-g_1}r^{1-g_2}r^{-1}
\\
&\quad \cdot\sum_{0\leq i,j\leq r-1}{t}^{\rank\left({\mathcal{V}_{g_1;b_1,\ldots,b_{n_1},i}^{r}}\right)+
\rank\left({\mathcal{V}_{g_2;c_1,\ldots,c_{n_2},j}^{r}}\right)+\delta_{i,j,0}}\widetilde{\Ch}^{r,-1,r/t}_{g_1,n_1+1}(v_{a_1}\otimes\cdots\otimes v_{a_{n_1}}\otimes v_i)\widetilde{\eta}^{i,j}\widetilde{\Ch}^{r,-1,r/t}_{g_2,n_2+1}(v_{a_{n_1+1}}\otimes\cdots\otimes v_{a_n}\otimes v_j)\\
&=\sum_{0\leq i,j\leq r-1}\Ch^{r}_{g_1,n_1+1}(v_{a_1}\otimes\cdots\otimes v_{a_{n_1}}\otimes v_i)\widetilde{\eta}^{i,j}t^{\delta_{i,j,0}}r^{-1}\Ch^{r}_{g_2,n_2+1}(v_{a_{n_1+1}}\otimes\cdots\otimes v_{a_n}\otimes v_j).
\end{aligned}
\end{equation*}

Thus, we see that
\begin{equation*}
\eta^{i,j}=\widetilde{\eta}^{i,j}t^{\delta_{i,j,0}}r^{-1}.
\end{equation*}
If we write more explicitly, we have
\begin{equation*}
\eta(v_a,v_b)
=t^{-\delta_{a,b,0}}r\widetilde{\eta}(v_a,v_b)
=t^{-\delta_{a,b,0}}\delta_{a+b\equiv 0 (\mathrm{mod}\ r)}
=
\begin{cases}
t^{-1}\quad \text{if}\quad a=b=0,\\
1\quad \text{if} \quad 1\leq a=r-b\leq r-1,\\
0\quad \text{otherwise}.
\end{cases}
\end{equation*}
    
\end{proof}

\subsection{Calculation of the quantum product}
Recall that the $\epsilon v_0$-shifted Chiodo CohFT is defined by
\begin{equation*}
\Ch_{g,n}^{r,\epsilon}(v_{a_1}\otimes\cdots\otimes v_{a_n})=\sum_{m\geq 0}\frac{\epsilon^m}{m!}{\mathpzc{p}_{m}}_\star\Ch_{g,n+m}^{r}(v_{a_1}\otimes\cdots\otimes v_{a_n}\otimes v_0^{\otimes m}),
\end{equation*}
and that its quantum product $\bullet$ on $(V,\eta)$ is determined by the three-point classes via Frobenius identity 
\begin{equation*}
\eta(v_a\bullet v_b,v_c)=\Ch_{0,3}^{r,\epsilon}(v_{a}\otimes v_b \otimes v_{c}).
\end{equation*}
Equivalently, in terms of the inverse of the pairing \eqref{eqn:Chiodo_pairing_t},
\begin{equation}\label{eqn:product_from_three_point}
v_a\bullet v_b
=\sum_{0\leq i,j \leq r-1}\Ch_{0,3}^{r,\epsilon}(v_{a}\otimes v_b \otimes v_{i})\,\eta^{i,j}\,v_j
=t\,\Ch_{0,3}^{r,\epsilon}(v_{a}\otimes v_b \otimes v_{0})\,v_0
+\sum_{i=1}^{r-1}\Ch_{0,3}^{r,\epsilon}(v_{a}\otimes v_b \otimes v_{i})\,v_{r-i}.
\end{equation}
The computation of $\bullet$ therefore reduces to the evaluation of the classes $\Ch_{0,3}^{r,\epsilon}(v_{a}\otimes v_b \otimes v_{c})$. We first record the genus-zero Chiodo integrals entering this evaluation.
\begin{lemma}[\cite{CGG25}] \label{lem:raw_chiodo_calculations_from_CGG25} Let $0\leq a,b,c \leq r-1$. Then, we have
\begin{equation*}
\int_{\ocM_{0,3}}\widetilde{\Ch}_{0,3}^{r,-1,1}(v_{a}\otimes v_b \otimes v_{c})=\frac{1}{r} \quad \text{if } a+b+c \in \{r-1, 2r-1\},
\end{equation*}
\begin{equation*}
\int_{\ocM_{0,4}}\widetilde{\Ch}_{0,4}^{r,-1,1}(v_{a}\otimes v_b \otimes v_{c}\otimes v_0)=-\frac{1}{r^2}\quad \text{if } a+b+c=2r-2,
\end{equation*}
and
\begin{equation*}
\int_{\ocM_{0,5}}\widetilde{\Ch}_{0,5}^{r,-1,1}(v_{a}\otimes v_b \otimes v_{c}\otimes v_0^{\otimes 2})=\frac{2}{r^3}\quad \text{if } a+b+c=3r-3.
\end{equation*}
\end{lemma}
\begin{proof}
See, proof of \cite[Lemma 3.4]{CGG25}.
\end{proof}

\begin{proposition}
The quantum product for $\epsilon v_0$-shifted Chiodo CohFT is given by
\begin{equation*}
v_a \bullet v_b=
\begin{cases}
-v_{a+b+1}\quad &\text{if}\quad 0\leq a+b<r-1,\\
-\epsilon v_{a+b+2-r}-tv_{a+b+1-r}\quad &\text{if}\quad r-1\leq a+b <2r-2\\
-\epsilon^2v_1-tv_{r-1}-t\epsilon v_0 \quad &\text{if}\quad a+b=2r-2.
\end{cases}
\end{equation*}
\end{proposition}

\begin{proof}
Expanding the definition of the shift, we have
\begin{equation*}
\Ch_{0,3}^{r,\epsilon}(v_{a}\otimes v_b \otimes v_{c})
=\sum_{m\geq 0}\frac{\epsilon^m}{m!}\int_{\ocM_{0,m+3}}\Ch_{0,m+3}^{r}(v_{a}\otimes v_b \otimes v_{c}\otimes v_0^{\otimes m}).
\end{equation*}
We first determine which summands are zero. The class $\Ch_{0,m+3}^{r}(v_{a}\otimes v_b \otimes v_{c}\otimes v_0^{\otimes m})$ vanishes unless the modularity constraint
\begin{equation*}
a+b+c+m+1\equiv 0 \quad (\mathrm{mod\,}r)
\end{equation*}
holds. Moreover, since $\mathcal{V}_{0;a,b,c,0,\ldots,0}^{r}$ is a genuine vector bundle (see the discussion preceding \eqref{eqn:rank_general_s}), its $d^{\text{th}}$ Chern class vanish for $d>\rank(\mathcal{V}_{0;a,b,c,0,\ldots,0}^{r})$, and ${\mathpzc{f}}_\star$ preserves the cohomological degree as the forgetful map $\mathpzc{f}$ is finite. Hence, by the definition of $\Ch^{r}$ and the rank formula \eqref{eqn:rank_general_minus_one},
\begin{equation*}
\deg_{\mathrm{top}}\Ch_{0,m+3}^{r}(v_{a}\otimes v_b \otimes v_{c}\otimes v_0^{\otimes m})
\leq \rank\left(\mathcal{V}_{0;a,b,c,0,\ldots,0}^{r}\right)=\frac{m+1+a+b+c-r}{r}.
\end{equation*}
Since $\dim\ocM_{0,m+3}=m$, the integral of the class $\Ch_{0,m+3}^{r}(v_{a}\otimes v_b \otimes v_{c}\otimes v_0^{\otimes m})$ over $\ocM_{0,m+3}$ vanishes unless $m\leq \frac{m+1+a+b+c-r}{r}$, i.e.
\begin{equation*}
(r-1)(m+1)\leq a+b+c.
\end{equation*}
As $a+b+c\leq 3(r-1)$, this forces $m\leq 2$, and together with the modularity constraint it leaves precisely the pairs
\begin{equation*}
(m,\,a+b+c)\in\bigl\{(0,\,r-1),\ (0,\,2r-1),\ (1,\,2r-2),\ (2,\,3r-3)\bigr\}.
\end{equation*}
Next, we evaluate the corresponding integrals. As only the degree $m$ part of the integrand contributes to the integral over $\ocM_{0,m+3}$, by the definition of $\Ch^{r}$, we obtain
\begin{equation*}
\begin{aligned}
\int_{\ocM_{0,m+3}}\Ch_{0,m+3}^{r}(v_{a}\otimes v_b \otimes v_{c}\otimes v_0^{\otimes m})
&=(-1)^{m+3}\,r^{1}\,t^{\,\rank\left(\mathcal{V}_{0;a,b,c,0,\ldots,0}^{r}\right)}
\int_{\ocM_{0,m+3}}\widetilde{\Ch}_{0,m+3}^{r,-1,r/t}(v_{a}\otimes v_b \otimes v_{c}\otimes v_0^{\otimes m})\\
&=(-1)^{m+3}\,r^{m+1}\,t^{\,\rank\left(\mathcal{V}_{0;a,b,c,0,\ldots,0}^{r}\right)-m}
\int_{\ocM_{0,m+3}}\widetilde{\Ch}_{0,m+3}^{r,-1,1}(v_{a}\otimes v_b \otimes v_{c}\otimes v_0^{\otimes m})\\
&=(-1)^{m+3}\,r^{m+1}\,t^{\frac{a+b+c-(r-1)(m+1)}{r}}
\int_{\ocM_{0,m+3}}\widetilde{\Ch}_{0,m+3}^{r,-1,1}(v_{a}\otimes v_b \otimes v_{c}\otimes v_0^{\otimes m})
\end{aligned}
\end{equation*}
which is equal to
\begin{equation*}
\begin{aligned}
-r\int_{\ocM_{0,3}}\widetilde{\Ch}_{0,3}^{r,-1,1}(v_{a}\otimes v_b \otimes v_{c})
=-r\cdot\frac{1}{r}
&=-1\quad \text{for } (m,a+b+c)=(0,r-1),\\
-rt\int_{\ocM_{0,3}}\widetilde{\Ch}_{0,3}^{r,-1,1}(v_{a}\otimes v_b \otimes v_{c})
=-rt\cdot\frac{1}{r}
&=-t\quad \text{for } (m,a+b+c)=(0,2r-1),\\
r^2\int_{\ocM_{0,4}}\widetilde{\Ch}_{0,4}^{r,-1,1}(v_{a}\otimes v_b \otimes v_{c}\otimes v_0)
=r^2\cdot\frac{-1}{r^2}
&=-1 \quad \text{for } (m,a+b+c)=(1,2r-2),\\
-r^3\int_{\ocM_{0,5}}\widetilde{\Ch}_{0,5}^{r,-1,1}(v_{a}\otimes v_b \otimes v_{c}\otimes v_0^{\otimes 2})
=-r^3\cdot\frac{2}{r^3}
&=-2 \quad \text{for } (m,a+b+c)=(2,3r-3).
\end{aligned}
\end{equation*}
Putting all three cases together, we get
\begin{equation}\label{eqn:3point_shifted_Chiodo_compact}
\begin{aligned}
\Ch_{0,3}^{r,\epsilon}(v_{a}\otimes v_b \otimes v_{c})
&=\sum_{m=0}^{2}\frac{\epsilon^m}{m!}\int_{\ocM_{0,m+3}}\Ch_{0,m+3}^{r}(v_{a}\otimes v_b \otimes v_{c}\otimes v_0^{\otimes m})\\
&=-(\delta_{a+b+c,r-1}+t\delta_{a+b+c,2r-1}+\epsilon \delta_{a+b+c,2r-2}+\epsilon^2\delta_{a+b+c,3r-3}).
\end{aligned}
\end{equation}
Inserting equation \eqref{eqn:3point_shifted_Chiodo_compact} into formula \eqref{eqn:product_from_three_point}, we see that
\begin{equation*}
\begin{aligned}
v_a\bullet v_b
&=t \Ch_{0,3}^{r,\epsilon}(v_{a}\otimes v_b \otimes v_{0})v_0+\sum_{1\leq i \leq r-1}\Ch_{0,3}^{r,\epsilon}(v_{a}\otimes v_b \otimes v_{i})v_{r-i}\\
&=-t(\delta_{a+b,r-1}+\epsilon\delta_{a+b,2r-2})v_0-\sum_{1\leq i \leq r-1}(\delta_{a+b+i,r-1}+\epsilon \delta_{a+b+i,2r-2}+t\delta_{a+b+i,2r-1}+\epsilon^2\delta_{a+b+i,3r-3})v_{r-i}\\
&=-t(\delta_{a+b,r-1}+\epsilon\delta_{a+b,2r-2})v_0-\sum_{1\leq i \leq r-1}(\delta_{a+b+i,r-1}+\epsilon \delta_{a+b+i,2r-2}+t\delta_{a+b+i,2r-1}+\epsilon^2\delta_{a+b+i,3r-3})v_{r-i}.
\end{aligned}
\end{equation*}
If $0\leq a+b <r-1$, then we have\footnote{Note that if $r\geq 3$, we have $r-1<2r-2<2r-1<3r-3$.}
\begin{equation*}
v_a\bullet v_b=-\sum_{1\leq i\leq r-1}\delta_{a+b+i,r-1}v_{r-i}=-v_{a+b+1}.
\end{equation*}
If $a+b=r-1$, then we have
\begin{equation*}
v_a\bullet v_b=-tv_0-\epsilon v_1.
\end{equation*}
If $r\leq a+b <2r-2$, then we have
\begin{equation*}
v_a\bullet v_b=-\sum_{1\leq i \leq r-1}(\epsilon \delta_{a+b+i,2r-2}+t\delta_{a+b+i,2r-1})v_{r-i}=-\epsilon v_{a+b+2-r}-t v_{a+b+1-r}.
\end{equation*}
If $a+b=2r-2$, then we have
\begin{equation*}
v_a\bullet v_b=-t\epsilon v_0-t v_{r-1}-\epsilon^2 v_1.
\end{equation*}
This completes the proof.
\end{proof}


\section{General shifts for $\Ch_{g,n}^{r,s}$ when $-r+1\leq s \leq -1$} \label{appendix:general_shifts_for_r_s_bounds_polynomiality}

Let $\Vec{\epsilon}=\sum_{i=0}^{r-1}\epsilon_i v_i$. The \emph{$\Vec{\epsilon}$-shifted Chiodo CohFT} $\Ch_{g,n}^{r,s,\Vec{\epsilon}}$ is defined  by
\begin{equation*}
\Ch_{g,n}^{r,s,\Vec{\epsilon}}(v_{a_1}\otimes\cdots\otimes v_{a_n}) \coloneqq \sum_{m\geq 0}\frac{1}{m!}{\mathpzc{p}_{m}}_\star\Ch_{g,n+m}^{r}(v_{a_1}\otimes\cdots\otimes v_{a_n}\otimes v_{\Vec{\epsilon}}^{\otimes m}).
\end{equation*}
Expanding $\Vec{\epsilon}^{\otimes m}$ via the multinomial theorem, we obtain
\begin{equation}\label{eqn:general_shift_multinomial_formula}
\Ch_{g,n}^{r,s,\Vec{\epsilon}}(v_{a_1}\otimes\cdots\otimes v_{a_n}) =\sum_{m_0,\ldots,m_{r-1}\geq 0}\frac{\epsilon_0^{m_0}\cdot \cdots \cdot \epsilon_{r-1}^{m_{r-1}}}{m_0!\cdot \cdots \cdot m_{r-1}!}\left({\mathpzc{p}_{m}}\right)_\star\Ch_{g,n+m}^{r,s}(v_{a_1}\otimes\cdots\otimes v_{a_n}\otimes v_0^{\otimes m_0}\otimes\cdots \otimes v_{r-1}^{\otimes m_{r-1}})
\end{equation}
where we set
\begin{equation*}
m\coloneqq m_0+\cdots+m_{r-1}.
\end{equation*}.
\begin{lemma}
The summand $\left({\mathpzc{p}_{m}}\right)_\star\Ch_{g,n+m}^{r,s}(v_{a_1}\otimes\cdots\otimes v_{a_n}\otimes v_0^{\otimes m_0}\otimes\cdots \otimes v_{r-1}^{\otimes m_{r-1}})$ vanishes unless
\begin{equation}\label{eqn:vanishing_inequality_for_general_shifts_with_s}
\sum_{i=0}^{r-1}(r+s-i)\cdot m_i\leq K_{g,n}^{r,s}(\vert a \vert)
\end{equation}
where
\begin{equation}\label{eqn:K_gnrsa_constant}
K_{g,n}^{r,s}(\vert a \vert)\coloneqq (2g-2+n)(-s)+r(g-1)+\vert a \vert.
\end{equation}
\end{lemma}
\begin{proof}
With $n+m$ markings carrying weights
\begin{equation*}
(b_1,\ldots,b_{n+m})\coloneqq (a_1,\ldots,a_n,\underbrace{0,\ldots,0}_{m_0-\text{times}},
\underbrace{1,\ldots,1}_{m_1-\text{times}},\ldots,
\underbrace{r{-}1,\ldots,r{-}1}_{m_{r-1}-\text{times}}),
\end{equation*}
the associated vector bundle $\mathcal{V}_{g;b_1,\ldots,b_{n+m}}^{r,s}$ has the rank
\begin{equation*}
\rank \left(\mathcal{V}_{g;b_1,\ldots,b_{n+m}}^{r,s}\right)=\frac{(2g-2+n+m)(-s)+r(g-1)+\vert a \vert +\sum_{i=0}^{r-1}i\cdot m_i}{r}
\end{equation*}
by the rank formula \eqref{eqn:rank_general_s}. As the pushforward class $\left({\mathpzc{p}_{m}}\right)_\star\Ch_{g,n+m}^{r,s}(v_{a_1}\otimes\cdots\otimes v_{a_n}\otimes v_0^{\otimes m_0}\otimes\cdots \otimes v_{r-1}^{\otimes m_{r-1}})$ has cohomological degree at most 
\begin{equation*}
2\left(\rank\left(\mathcal{V}_{g;b_1,\ldots,b_{n+m}}^{r,s}\right)-m\right),
\end{equation*}
it vanishes trivially unless
\begin{equation*}
m\leq \rank\left(\mathcal{V}_{g;b_1,\ldots,b_{n+m}}^{r,s}\right)=\frac{(2g-2+n+m)(-s)+r(g-1)+\vert a \vert +\sum_{i=0}^{r-1}i\cdot m_i}{r}.
\end{equation*}
Re-arranging the terms in this inequality using $m=\sum_{i=0}^{r-1}m_i$, we conclude the proof.
\end{proof}

\begin{remark}\label{rmk:coefficient_signs_in_Kgnrsa_inequality}
Since $-r+1\leq s\leq -1$, the coefficients $(r+s-i)$ in
\eqref{eqn:vanishing_inequality_for_general_shifts_with_s}
is
\begin{itemize}
    \item positive if $0\leq i \leq r+s-1$,
    \item zero if $i=r+s$,
    \item negative if $r+s+1\leq i \leq r-1$.
\end{itemize}
We see that the inequality \eqref{eqn:vanishing_inequality_for_general_shifts_with_s} gives us
\begin{equation}\label{eqn:vanishing_inequality_for_general_shifts_with_s_expanded}
\underbrace{(r+s)\cdot m_0+\cdots+1\cdot m_{r+s-1}}_{\text{each} \,m_i\, \text{has positive coefficient}}+0\cdot m_{r+s}+ \underbrace{(-1)\cdot m_{r+s+1}+\cdots+(s+1)\cdot m_{r-1}}_{\text{each} \,m_i\, \text{has negative coefficient}}\leq K_{g,n}^{r,s}(\vert a \vert).
\end{equation}
When $s\leq -2$, we will have both positive and negative coefficients on the left-hand side of the inequality \eqref{eqn:vanishing_inequality_for_general_shifts_with_s_expanded}. Hence, the inequality will not impose any bound on $m_i$ values unless all $m_{r+s+1},\ldots,m_{r-1}$ are fixed.
\end{remark}

\begin{remark}
For the rest of this part of the appendix, we will focus on the case $s=-1$.  
\end{remark}

In the case of $s=-1$, we see that inequality \eqref{eqn:vanishing_inequality_for_general_shifts_with_s_expanded}, gives us
\begin{equation*}
(r-1)\cdot m_0+(r-2)\cdot m_1+\cdots+1\cdot m_{r-2}\leq (r+2)(g-1)+n+\vert a \vert.
\end{equation*}
Hence, we conclude that by Remark \ref{rmk:coefficient_signs_in_Kgnrsa_inequality}, the sum
\eqref{eqn:general_shift_multinomial_formula} is a polynomial in
$\epsilon_0,\ldots,\epsilon_{r-2}$, and a formal power series in
$\epsilon_{r-1}$:
\begin{equation}\label{eqn:general_shift_minus_one_before_power_series}
\Ch_{g,n}^{r,\Vec{\epsilon}}(v_{a_1}\otimes\cdots\otimes v_{a_n})\in H^*(\ocM_{g, n})[t,\epsilon_0,\ldots,\epsilon_{r-2}][\![\epsilon_{r-1}]\!]
\end{equation}
In fact, we will show that the dependence on $\epsilon_{r-1}$ is analytic around $\epsilon_{r-1}=0$. The key tool is the following property of the (unshifted) Chiodo CohFT $\Ch^{r}$.

\begin{lemma}\label{lem:v_r_minus_1_pullback_propety}
Let $\mathpzc{p}\colon \ocM_{g,k+1}\to \ocM_{g,k}$ be the forgetful map forgetting the last marked point.
For all $v_{b_1},\ldots,v_{b_k}\in V$ with $2g-2+k>0$, we have
\begin{equation}\label{eqn:non_flat_unit_insertion}
\Ch_{g,k+1}^{r}(v_{b_1}\otimes\cdots\otimes v_{b_k}\otimes v_{r-1})
= -(t-\psi_{k+1})\cdot
\mathpzc{p}^*\Ch_{g,k}^{r}(v_{b_1}\otimes\cdots\otimes v_{b_k}).
\end{equation}
\end{lemma}
\begin{proof}
By \cite[Theorem 4.1 (ii)-(iv)]{GLN23}, we have
\begin{equation}\label{eqn:Harer_Zagier_Paper_ii_iv}
\widetilde{\Ch}_{g,k+1}^{r,-1,\tau}(v_{b_1}\otimes\cdots\otimes v_{b_k}\otimes v_{r-1})=\left(1-\frac{\tau}{r}\psi_{k+1}\right)\mathpzc{p}^*\widetilde{\Ch}_{g,k}^{r,-1,\tau}(v_{b_1}\otimes\cdots\otimes v_{b_k}).
\end{equation}
By the definition of $\Ch_{g,k+1}^{r}$, the rank formula \eqref{eqn:rank_general_minus_one},  and equation \eqref{eqn:Harer_Zagier_Paper_ii_iv}, we have
\begin{equation*}
\begin{aligned}
\Ch_{g,k+1}^{r}(v_{b_1}\otimes\cdots\otimes v_{b_k}\otimes v_{r-1})
&=(-1)^{k+1}r^{1-g}t^{\rank \left(\mathcal{V}^{r,-1}_{g;b_1,\ldots,b_k,r-1}\right)}\widetilde{\Ch}_{g,k+1}^{r,-1,r/t}(v_{b_1}\otimes\cdots\otimes v_{b_k}\otimes v_{r-1})\\
&=(-1)^{k+1}r^{1-g}t^{\rank \left(\mathcal{V}^{r,-1}_{g;b_1,\ldots,b_k}\right)+1}\left(1-\frac{1}{t}\psi_{k+1}\right)\mathpzc{p}^*\widetilde{\Ch}_{g,k}^{r,-1,r/t}(v_{b_1}\otimes\cdots\otimes v_{b_k})\\
&=-(t-\psi_{k+1})(-1)^kr^{1-g}t^{\rank \left(\mathcal{V}^{r,-1}_{g;b_1,\ldots,b_k}\right)}\mathpzc{p}^*\widetilde{\Ch}_{g,k}^{r,-1,r/t}(v_{b_1}\otimes\cdots\otimes v_{b_k})\\
&=-(t-\psi_{k+1})\cdot
\mathpzc{p}^*\Ch_{g,k}^{r}(v_{b_1}\otimes\cdots\otimes v_{b_k}).
\end{aligned}
\end{equation*}
Hence, we complete the proof.
\end{proof}

\begin{proposition}
\label{prop:general_shift_ring_containment}
For any $v_{a_1},\ldots,v_{a_n}\in V$, we have
\begin{equation}\label{eqn:general_shift_closed_form}
\begin{aligned}
&\Ch_{g,n}^{r,\Vec{\epsilon}}(v_{a_1}\otimes\cdots\otimes v_{a_n})\\
&=\sum_{m_0,\ldots,m_{r-2}\geq 0}
\frac{\epsilon_0^{m_0}\cdots \epsilon_{r-2}^{m_{r-2}}}
     {m_0!\cdots m_{r-2}!}
(1-\epsilon_{r-1})^{-(2g-2+n+\overline{m})}\,
{\mathpzc{p}_{\overline{m}}}_\star\Ch_{g,n+\overline{m}}^{r}
(v_{a_1}\otimes\cdots\otimes v_{a_n}
 \otimes v_0^{\otimes m_0}\otimes\cdots \otimes v_{r-2}^{\otimes m_{r-2}}),
\end{aligned}
\end{equation}
where we set
\begin{equation*}
\overline{m}\coloneq m-m_{r-1}=m_0+\cdots+m_{r-2}.
\end{equation*}
In particular,
\begin{equation}\label{eqn:general_shift_ring_final}
\Ch_{g,n}^{r,\Vec{\epsilon}}(v_{a_1}\otimes\cdots\otimes v_{a_n})
\in H^*(\ocM_{g,n})[\epsilon_0,\ldots,\epsilon_{r-2},(1-\epsilon_{r-1})^{-1},t].
\end{equation}
\end{proposition}
\begin{proof}Firstly, to express the equations more compactly, we set
\begin{equation*}
\Phi\coloneqq v_{a_1}\otimes\cdots\otimes v_{a_n}.
\end{equation*}By Lemma \ref{lem:v_r_minus_1_pullback_propety}, we have
\begin{equation}\label{eqn:v_r_minus_one_pull_back_p_m}
\begin{aligned}
&{\mathpzc{p}_{m}}_\star
\Ch_{g,n+m}^{r}
(\Phi
 \otimes v_0^{\otimes m_0}\otimes\cdots \otimes v_{r-1}^{\otimes m_{r-1}})\\
&\quad=-{\mathpzc{p}_{m}}_\star\left[
(t-\psi_{n+m})\cdot
\mathpzc{p}^*\Ch_{g,n+m-1}^{r}
(\Phi
 \otimes v_0^{\otimes m_0}\otimes\cdots
 \otimes v_{r-1}^{\otimes m_{r-1}-1})
\right].
\end{aligned}
\end{equation}

Note that the projection formula yields
\begin{equation}\label{eqn:projection_formula_p_star}
\mathpzc{p}_\star\left[
(t-\psi_{n+m})\cdot \mathpzc{p}^*(\Phi)
\right]
= -\mathpzc{p}_\star(\psi_{n+m})\cdot\Phi
= -\kappa_0\cdot\Phi=-(2g-2+n+m-1)\Phi
\end{equation}
for any class $\Phi \in  H^*(\ocM_{g, n+m-1})$, since
$\mathpzc{p}_\star(1)=0$ and
$\mathpzc{p}_\star(\psi_{n+m})=\kappa_0 = 2g-2+(n+m-1)$
on $\ocM_{g,n+m-1}$.

By the identity ${\mathpzc{p}_{m}}_\star = {\mathpzc{p}_{m-1}}_\star \circ \mathpzc{p}_\star$, projection formula \eqref{eqn:projection_formula_p_star}, and pullback equation \eqref{eqn:v_r_minus_one_pull_back_p_m}, we see that
\begin{equation*}
\begin{aligned}
&{\mathpzc{p}_{m}}_\star\Ch_{g,n+m}^{r}(\Phi\otimes v_0^{\otimes m_0}\otimes\cdots \otimes v_{r-1}^{\otimes m_{r-1}})\\
&=-{\mathpzc{p}_{m}}_\star(t-\psi_{n+m-1})\mathpzc{p}^* \Ch_{g,n+m}^{r}(\Phi\otimes v_0^{\otimes m_0}\otimes\cdots \otimes v_{r-1}^{\otimes m_{r-1}-1})\\
&=\kappa_0{\mathpzc{p}_{m-1}}_\star\Ch_{g,n+m-1}^{r}(\Phi\otimes v_0^{\otimes m_0}\otimes\cdots \otimes v_{r-1}^{\otimes m_{r-1}-1})\\
&=(2g-2+n+m-1){\mathpzc{p}_{m-1}}_\star\Ch_{g,n+m-1}^{r}(\Phi\otimes v_0^{\otimes m_0}\otimes\cdots \otimes v_{r-1}^{\otimes m_{r-1}-1}).
\end{aligned}
\end{equation*}
Inductively proceeding on $m$ exactly $m_{r-1}$ times, we get
\begin{equation}\label{eqn:v_r_minus_one_pullback_pochhammer}
\begin{aligned}
 &{\mathpzc{p}_{m}}_\star\Ch_{g,n+m}^{r}(\Phi\otimes v_0^{\otimes m_0}\otimes\cdots \otimes v_{r-1}^{\otimes m_{r-1}})\\
 &=(2g-2+n+m-m_{r-1})_{m_{r-1}}{\mathpzc{p}_{m-m_{r-1}}}_\star\Ch_{g,n+m-m_{r-1}}^{r}(\Phi\otimes v_0^{\otimes m_0}\otimes\cdots \otimes v_{r-2}^{\otimes m_{r-2}})\\
 &=(2g-2+n+\overline{m})_{m_{r-1}}{\mathpzc{p}_{\overline{m}}}_\star\Ch_{g,n+\overline{m}}^{r}(\Phi\otimes v_0^{\otimes m_0}\otimes\cdots \otimes v_{r-2}^{\otimes m_{r-2}})
\end{aligned}
\end{equation}
where
\begin{equation*}
(x)_{k}\coloneqq x(x+1) \cdot \cdots \cdot (x+k-1) \quad \text{for}\ k\geq 1, \quad \text{and}\quad (x)_{0}\coloneqq 1
\end{equation*}
is the Pochhammer symbol.

Now, observe the following computation
\begin{equation*}
\begin{aligned}
&\Ch_{g,n}^{r,\Vec{\epsilon}}(\Phi)\\
&= \sum_{m_0,\ldots,m_{r-1}\geq 0 }\frac{\epsilon_0^{m_0}\cdot \cdots \cdot \epsilon_{r-1}^{m_{r-1}}}{m_0!\cdot \cdots \cdot m_{r-1}!}{\mathpzc{p}_{m}}_\star\Ch_{g,n+m}^{r}(\Phi\otimes v_0^{\otimes m_0}\otimes\cdots \otimes v_{r-1}^{\otimes m_{r-1}})\\
&=\sum_{m_0,\ldots,m_{r-1}\geq 0}\frac{\epsilon_0^{m_0}\cdot \cdots \cdot \epsilon_{r-1}^{m_{r-1}}}{m_0!\cdot \cdots \cdot m_{r-1}!}(2g-2+n+\overline{m})_{m_{r-1}}{\mathpzc{p}_{\overline{m}}}_\star\Ch_{g,n+\overline{m}}^{r}(\Phi\otimes v_0^{\otimes m_0}\otimes\cdots \otimes v_{r-2}^{\otimes m_{r-2}})\\
&=\sum_{m_0,\ldots,m_{r-2}\geq 0}\frac{\epsilon_0^{m_0}\cdot \cdots \cdot \epsilon_{r-2}^{m_{r-2}}}{m_0!\cdot \cdots \cdot m_{r-2}!}\left(\sum_{m_{r-1}\geq0}\frac{(2g-2+n+\overline{m})_{m_{r-1}}}{m_{r-1}!}\epsilon_{r-1}^{m_{r-1}}\right){\mathpzc{p}_{\overline{m}}}_\star\Ch_{g,n+\overline{m}}^{r}(\Phi\otimes v_0^{\otimes m_0}\otimes\cdots \otimes v_{r-2}^{\otimes m_{r-2}})
\end{aligned}
\end{equation*}

By generalized binomial expansion \cite[Section 4.6 (ii)]{NIST:DLMF}, we have the following:
\begin{equation*}
(1-y)^{-x}=\sum_{k\geq 0}\frac{(x)_{k}}{k!}y^k
\end{equation*}
for any $x,y \in \mathbb{C}$ with $\vert y \vert <1$. Hence, we get
\begin{equation*}
\Ch_{g,n}^{r,\Vec{\epsilon}}(\Phi)
=\sum_{m_0,\ldots,m_{r-2}\geq 0}\frac{\epsilon_0^{m_0}\cdot \cdots \cdot \epsilon_{r-2}^{m_{r-2}}}{m_0!\cdot \cdots \cdot m_{r-2}!}(1-\epsilon_{r-1})^{-(2g-2+n+\overline{m})}{\mathpzc{p}_{\overline{m}}}_\star\Ch_{g,n+\overline{m}}^{r}(\Phi\otimes v_0^{\otimes m_0}\otimes\cdots \otimes v_{r-2}^{\otimes m_{r-2}}).
\end{equation*}
As a result, we conclude the proof by equation \eqref{eqn:general_shift_minus_one_before_power_series}.
\end{proof}

\bibliographystyle{abbrv}
\bibliography{biblio}

\end{document}